\documentclass[leqno,12pt]{amsart} 
\usepackage{amssymb} 
 
\theoremstyle{plain} 
\newtheorem{theorem}{\indent\sc Theorem}[section] 

\newtheorem{corollary}[theorem]{\indent\sc Corollary}
\newtheorem{proposition}[theorem]{\indent\sc Proposition}
\newtheorem{claim}[theorem]{\indent\sc Claim}
\theoremstyle{definition} 
\newtheorem{definition}[theorem]{\indent\sc Definition}
\newtheorem{remark}[theorem]{\indent\sc Remark}

\begin{document}

\title[Ricci curvature]{Ricci curvature and $L^p$-convergence} % 論文タイトル []内は柱用

\author[Shouhei Honda]{Shouhei Honda} 

\subjclass[2000]{Primary 53C20.}

\keywords{Gromov-Hausdorff convergence, Ricci curvature, and Laplacian.}

\address{ 
Faculty of Mathmatics \endgraf
Kyushu University \endgraf 
744, Motooka, Nishi-Ku, \endgraf 
Fukuoka 819-0395 \endgraf 
Japan
}
\email{honda@math.kyushu-u.ac.jp}

\maketitle

\begin{abstract}
We give the definition of $L^p$-convergence of tensor fields with respect to the Gromov-Hausdorff topology and several fundamental properties of the convergence.
We apply this to establish a Bochner-type inequality which keeps the term of Hessian on the Gromov-Hausdorff limit space of a sequence of Riemannian manifolds with a lower Ricci curvature bound and to give a geometric explicit formula for the Dirichlet Laplacian on a limit space defined by Cheeger-Colding.
We also prove a continuity of the first eigenvalues of the $p$-Laplacian with respect to the Gromov-Hausdorff topology.
\end{abstract}

\section{Introduction}
Let $n \in \mathbf{N}$, $K \in \mathbf{R}$ and let $(M_{\infty}, m_{\infty}, \upsilon)$ be the Gromov-Hausdorff limit metric measure space of a sequence of renormalized pointed complete $n$-dimensional Riemannian manifolds $\{(M_i, m_i, \underline{\mathrm{vol}})\}_{i \in \mathbf{N}}$ with $\mathrm{Ric}_{M_i} \ge K(n-1)$ and $M_{\infty} \neq \{m_{\infty}\}$, where  $\underline{\mathrm{vol}}:=\mathrm{vol}/\mathrm{vol}\,B_1(m_i)$.

In \cite{ch-co3} Cheeger-Colding showed that the cotangent bundle $\pi_1^0: T^*M_{\infty} \to M_{\infty}$ of $M_{\infty}$ exists in some sense.
It is a fundamental property of the cotangent bundle that every Lipschitz function $f$ on a Borel subset $A$ of $M_{\infty}$ has the canonical section $df (x) \in T^*_xM_{\infty}$ (called the \textit{differential} of $f$) for a.e. $x \in A$. 
We also define the tangent bundle $\pi^1_0: TM_{\infty} \to M_{\infty}$ of $M_{\infty}$ by the dual vector bundle of $T^*M_{\infty}$ and denote the dual section of $df$ by $\nabla f: A \to TM_{\infty}$.
For $r, s \in \mathbf{Z}_{\ge 0}$, let $\pi^r_s: T^r_sM_{\infty}:=\bigotimes_{i=1}^rTM_{\infty} \otimes \bigotimes _{i=r+1}^{r+s}T^*M_{\infty} \to M_{\infty}$.
For $A \subset M_{\infty}$, we put $T^r_sA:=(\pi^r_s)^{-1}(A)$.
We will denote by $\langle \cdot, \cdot \rangle$ the canonical metric on $T^r_sM_{\infty}$ (defined by the \textit{Riemannian metric $g_{M_{\infty}}$ of $M_{\infty}$}) for brevity and  by $L^p(T^r_sA)$ the space of $L^p$-sections of $T^r_sA$ over $A$.
Note $g_{M_{\infty}} \in L^{\infty}(T^0_2M_{\infty})$. 

Let $r, s \in \mathbf{Z}_{\ge 0}, R>0, 1<p < \infty$ and $T_i \in L^p(T^r_sB_R(m_{i}))$ for every $i\le \infty$ with $\sup_{i \le \infty}||T_i||_{L^p}<\infty$, where $B_R(m_{i}):=\{x_{i} \in M_{i}; \overline{x_{i}, m_{i}}<R\}$ and $\overline{x_{i}, m_{i}}$ is the distance between $x_{i}$ and $m_{i}$.

The main purpose of this paper is to give the following two definitions and applications: 
\begin{enumerate}
\item[\textbf{(W)}] \textit{$T_i$ $L^p$-converges weakly to $T_{\infty}$.}
\item[\textbf{(S)}] \textit{$T_i$ $L^p$-converges strongly to $T_{\infty}$.}
\end{enumerate}
Note that in \cite{KS2} Kuwae-Shioya gave the definitions above for the case of $r=s=0$ (i.e., each $T_i$ is a function) and showed several important properties.
A difficulty to give the definitions above for tensor fields is that we can NOT consider the difference `$T_i - T_{\infty}$' canonically
because it would be hard to compare between $T^r_sM_i$ and $T^r_sM_{\infty}$.
We first give equivalent versions of the definitions:
\begin{definition}[Definitions \ref{tensor weak conv}, \ref{strongdef} and Proposition \ref{hon}]
$\,\,\,\,\,\,\,\,\,\,\,\,\,\,\,\,\,\,\,\,\,\,\,\,$
\begin{enumerate}
\item[\textbf{(W)}] We say that \textit{$T_i$ $L^p$-converges weakly to $T_{\infty}$ on $B_R(m_{\infty})$} if for every $x_{\infty} \in B_R(m_{\infty})$, every $\{z_i\}_{1 \le i \le r+s} \subset M_{\infty}$ and every $r>0$ with $B_r(x_{\infty}) \subset B_R(m_{\infty})$  we have 
\[\lim_{j \to \infty}\int_{B_r(x_j)}\left\langle T_j, \bigotimes _{i=1}^r\nabla r_{z_{i, j}} \otimes \bigotimes_{i=r+1}^{r+s}dr_{z_{i, j}}\right\rangle d\underline{\mathrm{vol}}=\int_{B_r(x_{\infty})}\left\langle T_{\infty}, \bigotimes _{i=1}^r\nabla r_{z_{i}} \otimes \bigotimes_{i=r+1}^{r+s}dr_{z_{i}} \right\rangle d\upsilon,\]
where $x_j \to x_{\infty}$, $z_{i, j} \to z_i$ as $j \to \infty$ and $r_z$ is the distance function from $z$.
\item[\textbf{(S)}] We say that \textit{$T_i$ $L^p$-converges strongly to $T_{\infty}$ on $B_R(m_{\infty})$} if $T_i$ $L^p$-converges weakly to $T_{\infty}$ on $B_R(m_{\infty})$ and $\limsup_{i \to \infty}||T_i||_{L^p(B_R(m_{i}))}\le ||T_{\infty}||_{L^p(B_R(m_{\infty}))}$.
\end{enumerate}
\end{definition}
Compare with the definition of the convergence of the differentials of Lipschitz functions with respect to the Gromov-Hausdorff topology \cite[Definition $4.4$]{Ho}.
It is important that if $(M_i, m_i, \underline{\mathrm{vol}}) \equiv (M_{\infty}, m_{\infty}, \upsilon)$ holds for every $i$, then $T_i$ $L^p$-converges strongly to $T_{\infty}$ on $B_R(m_{\infty})$ in the sense of \textbf{(S)} if and only if $||T_i-T_{\infty}||_{L^p(B_R(m_{\infty}))} \to 0$ as $i \to \infty$.

As an important example we first observe about $L^p$-convergence of Riemannian metrics $g_{M_i}$ of $M_i$ with respect to the Gromov-Hausdorff topology:
\begin{theorem}\label{metric}
We see that $g_{M_i}$ $L^{\hat{p}}$-converges weakly to $g_{M_{\infty}}$ on $B_{\hat{R}}(m_{\infty})$ for every $\hat{R}>0$ and every $1<\hat{p}<\infty$.
Moreover, $g_{M_i}$ $L^{\hat{p}}$-converges strongly to $g_{M_{\infty}}$ on $B_{\hat{R}}(m_{\infty})$ for some (or every) $\hat{R}>0$ and some (or every) $1<\hat{p}<\infty$ if and only if $(M_{\infty}, m_{\infty})$ is the noncollapsed limit space of $\{(M_{i}, m_{i})\}_i$ (i.e., the Hausdorff dimension of $M_{\infty}$ is equal to $n$).
\end{theorem}
Roughly speaking, this theorem says that \textit{a Gromov-Hausdorff convergence always yields $L^p$-weak convergence of the Riemannian metrics.}

Let us denote by $H_{1, p}(U)$ the $H_{1, p}$-Sobolev space on an open subset $U$ of $M_{\infty}$.
Note that every $f \in H_{1, p}(U)$ also has the canonical section $df(x) \in T^*_xM_{\infty}$ for a.e. $x \in U$ with $||f||_{H_{1, p}}=||f||_{L^p}+||df||_{L^p}$.

In \cite{ch-co3} Cheeger-Colding defined the Dirichlet Laplacian  on $L^2(M_{\infty})$ as the self adjoint operator by the closable bilinear form
\[\int_{M_{\infty}}\langle df, dg\rangle d\upsilon\]
if $M_{\infty}$ is compact. They also showed continuities of eigenvalues and of eigenfunctions with respect to the Gromov-Hausdorff topology which  solve a conjecture given by Fukaya in \cite{fu}. 
Kuwae-Shioya proved the existence of the Dirichlet Laplacian  on $L^2(B_R(m_{\infty}))$ and similar continuities for noncompact case in \cite{KS}.

In this paper we use the following notation: 
For every open subset $U$ of $M_{\infty}$, let $\mathcal{D}^2(\Delta^{\upsilon}, U)$ be the space of $f \in H_{1, 2}(U)$ satisfying  that there exists $h \in L^2(U)$ such that 
\[\int_{U}\langle df, dg \rangle d\upsilon = \int_{U}hgd\upsilon\]
holds for every Lipschitz function $g$ on $U$ with compact support.
Since $h$ is unique, we denote $h$ by $\Delta^{\upsilon}f$.

On the other hand in \cite{Ho3} we knew that $M_{\infty}$ has a \textit{second order differential structure} in some weak sense.
More precisely, by taking a subsequence in advance without loss of generality we can assume that there is such a second order differential structure \textit{associated with $\{(M_i, m_i, \underline{\mathrm{vol}})\}_i$} in some sense. 
See subsection $2.5.7$ for the precise definition.
We will always consider such structure. 

It was also proved in \cite{Ho3} that the Riemannian metric $g_{M_{\infty}}$ is differentiable at a.e. $x \in M_{\infty}$ with respect to the structure, in particular we knew that there exists the \textit{Levi-Civita connection}. 
It is important that these facts allow us to define a \textit{weakly twice differentiable function} and the \textit{Hessian} of a weakly twice differentiable function naturally.

We will apply several fundamental properties of \textbf{(W)} and of \textbf{(S)} to the study of the second order differential structure on $M_{\infty}$.
In this section we introduce the following four applications only.
The first application is about $L^2$-weak convergence of Hessians:
\begin{theorem}\label{main}
Let $f_i \in L^2(B_R(m_i))$ for every $i \le \infty$.
Assume that $f_i \in C^2(B_R(m_i))$ holds for every $i<\infty$, $\sup_{i<\infty}(|| f_i||_{H_{1, 2}(B_R(m_i))}+||\Delta f_i||_{L^2(B_R(m_i))})<\infty$ and that $f_i$ $L^2$-converges weakly to $f_{\infty}$ on $B_R(m_{\infty})$.
Then there exists $p_1:=p_1(n, K, R)>1$ depending only on $n, R$ such that the following hold:
\begin{enumerate}
\item\label{a} $f_i$ $L^2$-converges strongly to $f_{\infty}$ on $B_R(m_{\infty})$.
\item\label{b} $f_{\infty} \in \mathcal{D}^2(\Delta^{\upsilon}, B_R(m_{\infty})) \cap H_{1, 2p_1}(B_r(m_{\infty}))$ for every $r<R$.
\item\label{c} $f_i, \nabla f_i$ $L^{2p_1}$-converge strongly to $f_{\infty}, \nabla f_{\infty}$ on $B_r(m_{\infty})$ for every $r<R$, respectively.
\item\label{d} $|\nabla f_{\infty}|^2 \in H_{1,  p_1}(B_r(m_{\infty}))$ for every $r<R$.
\item\label{e} $\nabla|\nabla f_{i}|^2$ $L^{p_1}$-converges weakly to $\nabla|\nabla f_{\infty}|^2$ on $B_r(m_{\infty})$ for every $r<R$.
\item\label{f} $\Delta f_i$ $L^2$-converges weakly to $\Delta^{\upsilon}f_{\infty}$ on $B_R(m_{\infty})$.
\item\label{g} $f_{\infty}$ is a weakly twice differentiable function on $B_R(m_{\infty})$.
\item\label{h} The Hessian $\mathrm{Hess}_{f_{\infty}}$ of $f_{\infty}$ is in $L^2(T^0_2B_r(m_{\infty}))$ for every $r<R$.
\item\label{i} $\mathrm{Hess}_{f_i}$ $L^2$-converges weakly to $\mathrm{Hess}_{f_{\infty}}$ on $B_r(m_{\infty})$ for every $r<R$.
\end{enumerate}
\end{theorem}
The second application is the following \textit{Bochner-type inequality} on $M_{\infty}$ which keeps the term of Hessian:
\begin{theorem}\label{bochner1}
Let $\{f_i\}_{i \le \infty}$ be as in Theorem \ref{main}.
Moreover assume that $\Delta f_i$ $L^2$-converges strongly to $\Delta^{\upsilon}f_{\infty}$ on $B_r(m_{\infty})$ for every $r<R$.
Then 
\begin{align*}
-\frac{1}{2}\int_{B_R(m_{\infty})}\langle d\phi_{\infty}, d|df_{\infty}|^2 \rangle d\upsilon &\ge \int_{B_R(m_{\infty})}\phi_{\infty}|\mathrm{Hess}_{f_{\infty}}|^2d\upsilon \\
&+ \int_{B_R(m_{\infty})}\left(-\phi_{\infty}(\Delta^{\upsilon}f_{\infty})^2+\Delta^{\upsilon}f_{\infty}\langle d\phi_{\infty}, df_{\infty}\rangle \right)d\upsilon \\
&+K(n-1)\int_{B_R(m_{\infty})}\phi_{\infty}|df_{\infty}|^2d\upsilon
\end{align*}
holds for every nonnegatively valued Lipschitz function $\phi_{\infty}$ on $B_R(m_{\infty})$ with compact support.
\end{theorem}
Note that the Bochner-type inequality holds on a dense subspace in $L^2(B_R(m_{\infty}))$ and that this is stronger than \textit{$\Gamma_2$-condition}.
See Remarks \ref{eigen conti} and \ref{gamma2}.
It is worth pointing out that recently Zhang-Zhu proved a similar result on an Alexandrov space in \cite{zz}.

On the other hand in \cite{Ho3} the author defined the (geometric) Laplacian $\Delta^{g_{M_{\infty}}}f$ for a weakly twice differentiable function $f$ by taking the trace of $-\mathrm{Hess}_f$:
\[\Delta ^{g_{M_{\infty}}}f:= \frac{-1}{\sqrt{\mathrm{det} (g_{ab})}}\sum_{i, j=1}^k\frac{\partial}{\partial x_i}\left(g^{ij}\sqrt{\mathrm{det} (g_{ab})}\frac{\partial f}{\partial x_j} \right)\]
on each \textit{$k$-dimensional rectifiable coordinate patch} $(U, \phi)$, where $\phi(p)=(x_1(p), x_2(p), \ldots, x_k(p)) \in \mathbf{R}^k$, $g_{ij}=g_{M_{\infty}}(\partial /\partial x_i, \partial / \partial x_j)$ and $(g^{ij})_{ij}=(g_{ij})_{ij}^{-1}$.

We now consider the following question:

\textbf{Question}:\, When does $\Delta^{\upsilon}f=\Delta^{g_{M_{\infty}}}f$ hold?

For example if $M_{\infty}$ is a $k$-dimensional smooth Riemannian manifold and $\upsilon$ is the $k$-dimensional Hausdorff measure, then 
 $\Delta^{g_{M_{\infty}}}f = \Delta^{\upsilon}f$ holds for every $f \in C^2(B_R(m_{\infty}))$ with $\Delta f \in L^2(B_R(m_{\infty}))$.
This is a direct consequence of the divergence formula on a manifold.
Note that in general $\Delta^{\upsilon}f \neq \Delta^{g_{M_{\infty}}}f$.
See Remark \ref{notexample} for an example.

The third application is to give a sufficient condition for $f$ in order to satisfy $\Delta^{\upsilon}f=\Delta^{g_{M_{\infty}}}f$:
\begin{theorem}\label{laplacian}
Let $\{f_i\}_{i \le \infty}$ be as in Theorem \ref{main}.
Then we have the following:
\begin{enumerate}
\item If $(M_{\infty}, m_{\infty})$ is the noncollapsed limit space of $\{(M_i, m_i)\}_i$, then $\Delta^{g_{M_{\infty}}} f_{\infty} = \Delta^{\upsilon}f_{\infty}$ on $B_R(m_{\infty})$.
\item If $\mathrm{Hess}_{f_i}$ $L^2$-converges strongly to $\mathrm{Hess}_{f_{\infty}}$ on $B_r(m_{\infty})$ for every $r<R$, then  $\Delta^{g_{M_{\infty}}} f_{\infty} = \Delta^{\upsilon}f_{\infty}$ on $B_R(m_{\infty})$ and $\Delta f_i$ $L^2$-converges strongly to $\Delta^{\upsilon}f_{\infty}$ on $B_r(m_{\infty})$ for every $r<R$. 
\end{enumerate}
\end{theorem}
In particular we see that an answer of the question above for noncollapsing case is POSITIVE on a dense subspace in $L^2(B_R(m_{\infty}))$.
Note that in \cite{KMS} Kuwae-Machigashira-Shioya showed a similar result about an explicit formula for the Dirichlet Laplacian on an Alexandrov space.

For $d >0$, let $\mathcal{M}(n, d, K)$ be the space of $n$-dimensional compact Riemannian manifolds $(M, \mathrm{Vol})$ with $\mathrm{diam}\,M \le d$ and $\mathrm{Ric}_M \ge K(n-1)$, where $\mathrm{Vol} := \mathrm{vol}/ \mathrm{vol}\,M$ . 
We denote by $\overline{\mathcal{M}(n, d, K)}$ the measured Gromov-Hausdorff compactification of $\mathcal{M}(n, d, K)$.
For $(X, \nu) \in \overline{\mathcal{M}(n, d, K)}$, we define \textit{the first eigenvalue of the $p$-Laplacian} by
\[\lambda_{1, p}(X):=\inf \left\{\frac{||df||^p_{L^p(X)}}{||f||^p_{L^p(X)}}; f \in H_{1, p}(X), f \not \equiv 0, \int_X |f|^{p-2}fd\nu=0 \right\} >0\]
if $X$ is not a single point, $\lambda_{1, p}(X):=\infty$ if $X$ is a single point.
It is well-known that $\lambda_{1, p}(M)$ coincides the smallest positive eigenvalue of the following nonlinear eigenvalue problem on  $(M, \mathrm{Vol}) \in \mathcal{M}(n, d, K)$:
\[\Delta_p\phi=\lambda |\phi|^{p-2}\phi,\]
where $\Delta_p \phi := -\mathrm{div} (|\nabla \phi|^{p-2} \nabla \phi)$.
The fourth application is the following:
\begin{theorem}\label{p-eigen}
The function $\lambda_{1, p}: \overline{\mathcal{M}(n, d, K)} \to (0, \infty]$ is continuous. 
\end{theorem}
This theorem is a generalization of the result about continuity of the eigenvalues of the Laplacian with respect to the Gromov-Hausdorff topology proved by Cheeger-Colding in \cite{ch-co3} to the case of the first eigenvalues of the $p$-Laplacian.

The organization of this paper is as follows:

In Section $2$ we will fix several notation and recall fundamental properties of metric measure spaces and of limit spaces of Riemannian manifolds.

In Section $3$ we will discuss $L^p$-convergence with respect to the Gromov-Hausdorff topology.
More precisely in subsection $3.1$ we will give the definitions of \textbf{(W)} and of \textbf{(S)} for the case of functions by a somewhat different way from Kuwae-Shioya given in \cite{KS2} and their fundamental properties.
We will also show that this formulation is equivalent to that by Kuwae-Shioya.
In particular we will extend a compactness result about $L^2$-weak convergence given by Kuwae-Shioya to $L^p$-case.
See Proposition \ref{weak com}, Corollary \ref{KSequiv} and Proposition \ref{strong com}.
In subsection $3.2$ we will give the original definitions of \textbf{(W)} and of \textbf{(S)} for tensor fields.
Roughly speaking their fundamental properties include the following:
\begin{enumerate}
\item Every $L^p$-bounded sequence has an $L^p$-weak convergent subsequence.
\item $L^p$-norms are lower semicontinuous with respect to the $L^p$-weak convergence.
\item $L^p$-strong (or $L^p$-weak) convergence is stable for every contraction under a suitable assumption.
\end{enumerate}
It is worth pointing out that a key notion to give the definitions of \textbf{(W)} and of \textbf{(S)} is the \textit{angle} $\angle xyz \in [0, \pi]$ given in \cite{Ho3}.
See subsection $2.5.5$ for the precise definition.

In Section $4$ we will apply several results given in Section $3$ to prove theorems introduced in this section.
Moreover we will show a compactness result about Sobolev functions with respect to the Gromov-Hausdorff topology which is a generalization of a Kuwae-Shioya's result about  $L^2$-energy functionals given in \cite{KS, KS2} to $L^p$-case.
See Theorem \ref{7} and Remark \ref{rmk}.
As an application of it, we will prove Theorem \ref{p-eigen}.
We will also discuss a Bochner-type \textit{formula} and the scalar curvature of a limit space.
See Theorem \ref{bochner7} and Corollary \ref{scal conv}.

\textbf{Acknowledgments.}
The author would like to express my appreciation to Professors Kazumasa Kuwada, Kazuhiro Kuwae, Shin-ichi Ohta and Takashi Shioya for helpful comments.
He was supported  by Grant-in-Aid for Young Scientists (B) $24740046$.
\section{Preliminaries}
\subsection{Fundamental notation.}
For $a, b \in \mathbf{R}$ and $\epsilon>0$, throughout this paper, we use the following notation:
\[a=b \pm \epsilon \Longleftrightarrow |a-b|<\epsilon.\]

Let us denote by $ \Psi (\epsilon_1, \epsilon _2 ,\ldots ,\epsilon_k  ; c_1, c_2,\ldots, c_l )$
some positive valued function on $\mathbf{R}_{>0}^k \times \mathbf{R}^l $ satisfying 
\[\lim_{\epsilon_1, \epsilon_2,\ldots ,\epsilon_k \to 0}\Psi (\epsilon_1, \epsilon_2,\ldots ,\epsilon_k  ; c_1, c_2 ,\ldots ,c_l)=0\] 
for fixed real numbers $c_1, c_2,\ldots ,c_l$.
We often denote by $C(c_1, c_2,\ldots ,c_l)$ some positive constant depending only on fixed real numbers $c_1, c_2,\ldots ,c_l$.

Let $X$ be a metric space and $x \in X$. 
For $r>0$ and $A \subset X$, put $B_r(x):=\{w \in X; \overline{x, w}<r\}$, $\overline{B}_r(x):=\{w \in X; \overline{x, w}\le r\}$ and $B_r(A):=\{w \in X; \overline{w, A}<r\}$.
We say that \textit{$X$ is proper} if every bounded closed subset of $X$ is compact.
We also say that \textit{$X$ is a geodesic space} if for every $p, q \in X$ there exists an isometric embedding $\gamma : [0, \overline{p, q}] \to X$ such that $\gamma (0)=p$ and $\gamma (\overline{p, q})=q$ (we call $\gamma$ \textit{a minimal geodesic from $p$ to $q$}).

Let $\upsilon$ be a Borel measure on $X$, $Y$ a metric space and $f$ a Borel map from $X$ to $Y$.
We say that $f$ is \textit{weakly Lipschitz} (or \textit{differentiable at a.e. $x \in X$}) if there exists a countable collection $\{A_i\}_i$ of Borel subsets $A_i$ of $X$ such that $\upsilon (X \setminus \bigcup_iA_i)=0$ and that each $f|_{A_i}$ is a Lipschitz map.

Let $V$ be an  $n$-dimensional real Hilbert space with the inner product $\langle \cdot, \cdot \rangle$, $v \in V$ and $1<p<\infty$.
We often use the following notation: $v^{(p-1)}:=|v|^{p-2}v$ if $v \neq 0$, $v^{(p-1)}:=0$ if $v=0$.
For $k \le n$, $\epsilon >0$ and $\{e_i\}_{1 \le i \le k} \subset V$,
we say that \textit{$\{e_i\}_{1 \le i \le k}$ is an $\epsilon$-orthogonal collection on $V$} if $\langle e_i, e_j \rangle =\delta_{ij} \pm \epsilon$ holds for every $i, j$.
Moreover, if $k=n$, then we say that  \textit{$\{e_i\}_{1 \le i \le n}$ is an $\epsilon$-orthogonal basis on $V$}.
It is easy to check the following:
\begin{proposition}\label{ep}
Let $\epsilon>0$, $L >0$, $r \ge 1$, $T \in \bigotimes_{i=1}^rV^*$ with $|T| \le L$, and let $\{e_i\}_{1 \le i \le k}$ be an $\epsilon$-orthogonal collection on $V$.
Then we have the following:
\begin{enumerate}
\item If $k=n$, then $|T|^2= \sum_{i_1, \ldots, i_r}\left(T(e_{i_1}, \ldots, e_{i_r})\right)^2 \pm \Psi (\epsilon; n, r, L)$.
\item If $r=2$ and $|T|^2=\sum_{i, j}(T(e_i, e_j))^2 \pm \epsilon$, then $\mathrm{Tr}\,T= \sum_{i=1}^kT(e_i, e_i) \pm \Psi(\epsilon; n, L)$, where $\mathrm{Tr}\,T$ is the trace of $T$.
\end{enumerate}
\end{proposition}
Let $r, s \in \mathbf{Z}_{\ge 0}$ and $T^r_s(V):=\bigotimes_{i=1}^rV \otimes \bigotimes_{i=r+1}^{r+s}V^*$.
We also denote  by $\langle \cdot, \cdot \rangle$ the canonical inner product on $T^r_s(V)$ for brevity.
For $1 \le l \le r$ and $r+1 \le k \le r+s$, let $C_k^l: T^r_s(V) \to T^{r-1}_{s-1}(V)$ be the contraction defined by 
$C_k^l(\bigotimes_{i=1}^rv_i\otimes \bigotimes_{i=r+1}^{r+s}v_{i}^*):=v_k^*(v_l)\bigotimes_{i=1}^{l-1}v_i \otimes \bigotimes_{i=l+1}^{r} v_{i} \otimes  \bigotimes_{i=r+1}^{k-1}v_i^* \otimes \bigotimes_{i=k+1}^{r+s} v_{i}^*$.
For $1 \le l < k \le r$, let 
$C_k^l: T^r_s(V) \to T^{r-2}_{s}(V)$ be the linear map defined by
 $C_k^l(\bigotimes_{i=1}^rv_i\otimes \bigotimes_{i=r+1}^{r+s}v_{i}^*):=\langle v_l, v_k \rangle \bigotimes_{i=1}^{l-1}v_i \otimes \bigotimes_{i=l+1}^{k-1} v_{i} \otimes \bigotimes_{i=k+1}^r v_i \otimes  \bigotimes_{i=r+1}^{r+s}v_i^*$.
Similarly we define $C_k^l: T^r_s(V) \to T^{r}_{s-2}(V)$ for $r+1 \le l<k \le r+s$.
For $l \le r, k \le s$, let $C^{l, r}_{k, s}:T^{r}_{s}(V) \times T^{l}_{k}(V) \to T^{r-l}_{s-k}(V)$ be the linear map defined by $C^{l, r}_{k, s}(\bigotimes_{i=1}^rv_i\otimes \bigotimes_{i=r+1}^{r+s}v_{i}^*,  \bigotimes_{i=1}^{l}w_i\otimes \bigotimes_{i=l+1}^{l+k}w_{i}^*):=\prod_{i=1}^{l}\langle v_i, w_i\rangle \prod_{i=r+1}^{r+k}\langle v_i^*, w_{i-r+l}^* \rangle \bigotimes_{i=l+1}^{r}v_i \otimes \bigotimes_{i=r+k+1}^{r+s}v^*_i$.
To simplify notation, we write $T(S)$ instead of $C^{l, r}_{k, s}(T, S)$.
Note $|T(S)|\le |T||S|$.
For $\hat{r}, \hat{s} \in \mathbf{Z}_{\ge 0}$, define the bilinear map $f : T^r_s(V) \times T^{\hat{r}}_{\hat{s}}(V) \to T^{r+\hat{r}}_{s+\hat{s}}(V)$ by
$f(\bigotimes_{i=1}^{r}v_i\otimes \bigotimes_{i=r+1}^{r+s}v_{i}^*, \bigotimes_{i=1}^{\hat{r}}w_i\otimes \bigotimes_{i=\hat{r}+1}^{\hat{r}+\hat{s}}w_{i}^*):=\bigotimes_{i=1}^{r}v_i \otimes \bigotimes_{i=1}^{\hat{r}}w_i \otimes \bigotimes_{i=r+1}^{r+s}v_{i}^* \otimes \bigotimes_{i=\hat{r}+1}^{\hat{r}+\hat{s}}w_{i}^*$.
For simplicity of notation, we write $v \otimes w$ instead of $f(v, w)$.
\subsection{Differentiability of Lipschitz functions on a Borel subset of Euclidean space.}
Let $A$ be a Borel subset of $\mathbf{R}^k$, $f$ a Lipschitz function on $A$ and $y \in \mathrm{Leb}\,A:=\{ a \in A; \lim_{r \to 0}H^k(A \cap B_r(a))/H^k(B_r(a))=1\}$, where $H^k$ is the $k$-dimensional spherical Hausdorff measure.
Then we say that \textit{$f$ is differentiable at $y$} if there exists a Lipschitz function $\hat{f}$ on $\mathbf{R}^k$ such that $\hat{f}|_A\equiv f$ and that $\hat{f}$ is  differentiable at $y$.
Note that if $f$ is differentiable at $y$, then a vector
$(\partial \hat{f}/\partial x_1(y), \ldots, \partial \hat{f}/\partial x_n(y))$
does not depend on the choice of such $\hat{f}$.
Thus we denote the vector by
$J(f)(y)=(\partial f/\partial x_1(y), \ldots, \partial f/\partial x_n(y)).$
Let $F=(f_1, \ldots, f_m)$ be a Lipschitz map from $A$ to $\mathbf{R}^m$.
We say that \textit{$F$ is differentiable at $y$} if every $f_i$ is differentiable at $y$.
Note that by Rademacher's theorem \cite{rad}, $F$ is differentiable at a.e. $x \in A$.
Let us denote by $J(F)(x)=(\partial f_i/\partial x_j(x))_{ij}$ the Jacobi matrix of $F$ at $x$ if $F$ is differentiable at $x \in \mathrm{Leb}\,A$. 
We also say that a Borel map $G$ from $A$ to $\mathbf{R}^m$ is \textit{weakly twice differentiable on $A$} if $G$ is weakly Lipschitz on $A$ and if $J(G)$ is weakly Lipschitz on $A$. 

Let $X=\sum_{a \in \Lambda} X_{a}\bigotimes_{i=1}^r\nabla x_{a(i)} \otimes \bigotimes_{i=r+1}^{r+s} dx_{a(i)}$ be a tensor field of type $(r, s)$ on $A$, where $\Lambda:=\mathrm{Map}(\{1, \ldots, r+s\} \to \{1, \ldots, k\})$.
We say that $X$ is a \textit{Borel tensor field on $A$} if every $X_{a}$ is a Borel function.
We also say that $X$ is \textit{weakly Lipschitz on $A$} (or \textit{differentiable at a.e. $x \in A$}) if every $X_{a}$ is weakly Lipschitz on $A$.

For two Borel tensor fields $\{X_i\}_{i=1,2}$ of type $(r, s)$ on $A$, we say that \textit{$X_1$ is equivalent to $X_2$ on $A$} if $X_1(x)=X_2(x)$ holds for a.e. $x \in A$.
Let us denote by $[X]$ the equivalent class of $X$, by $\Gamma_{\mathrm{Bor}}(T^r_sA)$ the set of equivalent classes, and 
by $\Gamma_{1}(T^r_sA)$ the set of equivalent classes represented by a weakly Lipschitz tensor field of type $(r, s)$.
We often write $X=[X]$ for brevity. 
See subsection $3.1$ in \cite{Ho3} for the details of this subsection.
\subsection{Rectifiable metric measure spaces.}
Let $X$ be a proper geodesic space and $\upsilon$ a Borel measure on $X$. 
In this paper we say that $(X, \upsilon)$ is a \textit{metric measure space} if $\upsilon(B_r(x))>0$ holds for every $x \in X$ and every $r>0$.
We now recall the notion of \textit{rectifiability} for metric measure spaces given by Cheeger-Colding in \cite{ch-co3}:
\begin{definition}[Cheeger-Colding, \cite{ch-co3}]\label{rectifiable}
Let $(X, \upsilon)$ be a metric measure space.
We say that \textit{$X$ is $\upsilon$-rectifiable} if there exist $m \in \mathbf{N}$,  collections $\{C_{i}^l\}_{1 \le l \le m, i \in \mathbf{N}}$ of Borel subsets $C_i^l$ of $X$, and of bi-Lipschitz embedding maps $\{ \phi_{i}^l: C_{i}^l \rightarrow \mathbf{R}^l \}_{l, i}$ such that the following three conditions hold:
\begin{enumerate}
\item $\upsilon(X \setminus \bigcup_{l,i}C_{i}^l)=0$.
\item For every $i$ and every $l$, $\upsilon$ is Ahlfors $l$-regular at every $x \in C_{i}^l$, i.e., there exist $C \ge 1$ and $r>0$ such that 
$C^{-1} \le \upsilon (B_t(x)) / t^l \le C$ holds 
for every $0<t<r$.
\item For every $l$, every $x \in \bigcup_{i \in \mathbf{N}}C_{i}^l$ and every $0 < \delta < 1$, there exists $i$ such that $x \in C_{i}^l$ and that 
the map $\phi_{i}^l$ is $(1 \pm \delta)$-bi-Lipschitz to the image $\phi_{i}^l(C_{i}^l)$.
\end{enumerate}
\end{definition}
See \cite[Definition $5.3$]{ch-co3} and the condition $iii)$ of page $60$ in \cite{ch-co3}.
In this paper we say that a family $\mathcal{A}:=\{(C_{i}^l, \phi_{i}^l)\}_{l, i}$ as in Definition \ref{rectifiable} is  a \textit{rectifiable coordinate system (or structure) of $(X, \upsilon)$} and that each $(C_i^l, \phi_i^l)$ is an \textit{$l$-dimensional rectifiable coordinate patch}.
It is important that the cotangent bundle on a rectifiable metric measure space exists in some sense.
We first give several fundamental properties of the cotangent bundle:
\begin{theorem}[Cheeger, Cheeger-Colding, \cite{ch1, ch-co3}]\label{29292}
Let $(X, \upsilon)$ be a rectifiable metric measure space.
Then there exist a topological space $T^*X$ and a Borel map  $\pi^0_1:T^*X \rightarrow X$ with the following properties:
\begin{enumerate}
\item $\upsilon(X \setminus \pi^0_1 (T^*X))=0$.
\item For every $w \in \pi^0_1(T^*X)$, $(\pi^0_1)^{-1}(w) (=T^*_wX)$ is a finite dimensional real Hilbert space with the inner product $\langle \cdot, \cdot \rangle_w$. Let $|v|(w):= \sqrt{\langle v, v\rangle_w}$.
\item For every Lipschitz function $f$ on $X$, there exist a Borel subset $V$ of $X$, and a Borel map $df$ from $V$ to $T^*X$ such that $\upsilon(X \setminus V)=0$, $\pi^0_1 \circ df\equiv id_V$ and that $|df|(w)=\mathrm{Lip}f(w)=Lipf(w)$ holds for every $w \in V$, where 
\begin{enumerate}
\item $\mathrm{Lip}f(x)=\lim_{r \to 0}(\sup_{y \in B_r(x)\setminus \{x\}}(|f(x)-f(y)|/\overline{x, y}))$ and
\item $Lip f(x)=\liminf_{r \to 0}(\sup_{y \in \partial B_r(x)}(|f(x)-f(y)|/\overline{x, y})).$
\end{enumerate}
\end{enumerate} 
\end{theorem}
Assume that $(X, \upsilon)$ is a rectifiable metric measure space.
We now give a short review of the construction of the cotangent bundle $T^*X$ as in Theorem \ref{29292}:
Let $\{(C_{i}^l, \phi_{i}^l)\}_{l, i}$ be a rectifiable coordinate system of $(X, \upsilon)$.
By Rademacher's theorem and Definition \ref{rectifiable}, without loss of generality we can assume that the following hold:
\begin{enumerate}
\item Every $\phi_i^l \circ (\phi_j^l)^{-1} : \phi_j^l(C_i^l \cap C_j^l) \to \phi_i^l(C_i^l \cap C_j^l)$ is differentiable at every $w \in \phi_j^l(C_i^l \cap C_j^l)$.
\item For every $i, l$, $x \in C_i^l$ and every $(a_1, \ldots, a_l), (b_1, \ldots, b_l) \in \mathbf{R}^l$, we have the following:
\begin{enumerate}
\item $\mathrm{Lip}\left(\sum_ja_j\phi_{i,j}^l\right)(x)=Lip \left(\sum_ja_j\phi^l_{i, j}\right)(x)$.
\item $\mathrm{Lip}\left(\sum_ja_j\phi_{i,j}^l\right)(x)=0$ holds if and only if $(a_1, \ldots, a_l)=0$ holds.
\item $\mathrm{Lip}\left(\sum_j(a_j+b_j)\phi^l_{i,j}\right)(x)^2+\mathrm{Lip}\left(\sum_j(a_j-b_j)\phi^l_{i,j}\right)(x)^2=2\mathrm{Lip}\left(\sum_ja_j\phi^l_{i,j}\right)(x)^2+2\mathrm{Lip}\left(\sum_jb_j\phi^l_{i,j}\right)(x)^2$.
\end{enumerate}
\item For every Lipschitz function $f$ on $X$, we see that $\mathrm{Lip}f(x)=Lip f(x)$ holds for a.e. $x \in X$.
\end{enumerate}
Let $\sim$ be the equivalent relation on $\bigsqcup_{i, l} (\phi_i^l(C_i^l) \times \mathbf{R}^l)$ defined by $(x, u) \sim (y, v)$ if $x = \phi_i^l \circ (\phi_j^l)^{-1}(y)$ and $u=J(\phi_i^l \circ (\phi_j^l)^{-1})(y)^tv$ for some $i, j, l$.
Put $T^*X:= \left(\bigsqcup_{i, l}( \phi_i^l(C_i^l) \times \mathbf{R}^l)\right)/ \sim$ and define the map $\pi^0_1: T^*X \to X$ by $\pi (x,u):=(\phi_i^l)^{-1}(x)$ if $x \in \phi_i^l(C_i^l)$.
The condition $(b)$ above yields that for every $x \in \pi^0_1(T^*X)$ with $x \in C_i^l$, we see that $|a|_x=\mathrm{Lip}\left(\sum_ja_j\phi^l_{i,j}\right)(x)$ is a norm on $\mathbf{R}^l$.
The condition $(c)$ above yields that the norm comes from an inner product $\langle \cdot, \cdot \rangle_x$ on $\mathbf{R}^l$. 
Then it is easy to check that $(T^*X, \pi^0_1, \langle \cdot, \cdot \rangle_x)$ satisfies the desired conditions as in Theorem \ref{29292}.
See Section $6$ in \cite{ch-co3} and page $458-459$ of \cite{ch1} for the details.

Note that similarly, we can define the ($L^{\infty}$-)vector bundle: $\pi^r_s: \bigotimes_{i=1}^r TX \otimes \bigotimes_{j=1}^sT^*X \to X$ for every $r, s \in \mathbf{Z}_{\ge 0}$.
Let $A$ be a Borel subset of $X$ and $T^r_sA$ (or $\bigotimes_{i=1}^r TA \otimes \bigotimes_{j=1}^sT^*A) :=(\pi^r_s)^{-1}(A)$.
For two Borel tensor fields $T_1, T_2$ of type $(r, s)$ on $A$, we say that \textit{$T_1$ is equivalent to $T_2$ on $A$} if
$T_1(x) =T_2(x)$ holds for a.e. $x \in A$. 
Let $\Gamma_{\mathrm{Bor}}(T^r_sA)$ be the space of equivalent classes of Borel sections of $T^r_sA$ over $A$.
Note that for every $T \in \Gamma_{\mathrm{Bor}}(T^r_sA)$, each restriction $T|_{C_i^l \cap A}$ of $T$ to $C_i \cap A$ can be regarded as in $\Gamma_{\mathrm{Bor}}(T^r_s\phi_i^l(C_i^l \cap A))$ and that every weakly Lipschitz function $f$ on $A$ has the canonical section $df \in \Gamma_{\mathrm{Bor}}(T^*A)$. 

We also denote the canonical metric on each fiber of $T^r_sX$ by $\langle \cdot, \cdot \rangle$ for short. 
In particular we call the canonical metric on $TX$ the \textit{Riemannian metric of} $(X, \upsilon)$ and denote it by $g_{X}$.
For every $1 \le p \le \infty$, let $L^p(T^r_sA):=\{T \in \Gamma_{\mathrm{Bor}}(T^r_sA); |T| \in L^p(A)\}$.
Note that $L^p(T^r_sA)$ with the $L^p$-norm is a Banach space and that $g_{X} \in L^{\infty}(T^0_2X)$.
For every weakly Lipschitz function $f$ on $A$ and every $V \in \Gamma_{\mathrm{Bor}}(TA)$, let $\nabla^{g_X} f:=(df)^* \in \Gamma_{\mathrm{Bor}}(TA)$ and $V(f):=\langle V, \nabla f\rangle$, where $^*$ is the canonical isometry $T^*_xX \cong T_xX$ by the Riemannian metric $g_X$.
See subsection $3.3$ in \cite{Ho3} for the detail.

Let $U$ be an open subset of $X$.
Let us denote by $\mathcal{D}^1_{\mathrm{loc}}(\mathrm{div}^{\upsilon}, U)$ the set of $T \in L^1_{\mathrm{loc}}(TU) := \{S \in \Gamma_{\mathrm{Bor}}(TU); |S| \in L^1_{\mathrm{loc}}(U)\}$ satisfying that there exists a unique $h \in L^1_{\mathrm{loc}}(U)$ such that 
\[-\int_{U}fhd\upsilon=\int_{U}\langle \nabla f, T\rangle d\upsilon\]
holds for every Lipschitz function $f$ on $U$ with compact support.
Write $\mathrm{div}^{\upsilon}T=h$.
For $1 \le p \le \infty$, let $\mathcal{D}^p(\mathrm{div}^{\upsilon}, U)$ be the set of $T \in \mathcal{D}^1_{\mathrm{loc}}(\mathrm{div}^{\upsilon}, U)$ satisfying that $T \in L^p(TU)$ and $\mathrm{div}^{\upsilon}T \in L^p(U)$ hold.
Note that for $\mathcal{D}^2(\Delta^{\upsilon}, U)$ defined as in Section $1$, we see that $f \in \mathcal{D}^2(\Delta^{\upsilon}, U)$ holds if and only if $f \in H_{1, 2}(U)$ and $\nabla f \in \mathcal{D}^2(\mathrm{div}^{\upsilon}, U)$ hold. 
\subsection{Weakly second order differential structure on rectifiable metric measure spaces.}
In this subsection we recall the definition of a weakly second order differential structure on a rectifiable metric measure space and their fundamental properties given in \cite{Ho3}.

Let $(X, \upsilon)$ be a metric measure space and $\mathcal{A}:=\{(C_i^l, \phi_i^l)\}_{i, l}$ a rectifiable coordinate system of $(X, \upsilon)$.
We say that \textit{$\mathcal{A}$ is a weakly second order differential structure (or system) on $(X, \upsilon)$} if each map $\phi_i^l \circ (\phi_j^l)^{-1}$ is weakly twice differentiable on $\phi_j^l(C_i^l \cap C_j^l)$.

Assume that $\mathcal{A}$ is a weakly second order differential structure on $(X, \upsilon)$.
Let $A$ be a Borel subset of $X$.
We say that \textit{$T \in \Gamma_{\mathrm{Bor}}(T^r_sA)$ is weakly Lipschitz} if each $X|_{C_i^l \cap A} $ (which can be regarded as in $\Gamma_{\mathrm{Bor}}(T^r_s\phi_i^l(C_i^l \cap A)))$ is a weakly Lipschitz tensor field of type $(r, s)$ on  $\phi_i^l(C_i \cap A)$.
Let us denote by $\Gamma_1(T^r_sA; \mathcal{A})$ the set of equivalent classes of Borel tensor fields of type $(r, s)$ on $A$ represented by a weakly Lipschitz tensor field.
We often write $\Gamma_1(T^r_sA):=\Gamma_1(T^r_sA; \mathcal{A})$ for brevity.
Recall that it was proved in \cite{Ho3} that for $U, V \in \Gamma_1(TA)$, the Lie bracket $[U, V] \in \Gamma_{\mathrm{Bor}}(TA)$ is well-defined in the ordinary way.

Let $f$ be a Borel function on $A$.
We say that \textit{$f$ is weakly twice differentiable on $A$ (with respect to $\mathcal{A}$)} if $f$ is weakly Lipschitz on $A$ and if $df \in \Gamma_1(T^*A)$.  
The following theorem is a main result of \cite{Ho3}:
\begin{theorem}\cite[Theorem $3.25$]{Ho3}\label{levi3}
Assume $g_X \in \Gamma_1(T^0_2X)$.  
Then there exists the Levi-Civita connection $\nabla^{g_X}$ on $X$ uniquely in the following sense:  
\begin{enumerate}
\item $\nabla^{g_X}$ is a map from $\Gamma_{\mathrm{Bor}} (TX) \times \Gamma_{1} (TX)$ to $\Gamma_{\mathrm{Bor}} (TX)$. Let $\nabla^{g_X}_UV:=\nabla^{g_X}(U, V)$.
\item $\nabla^{g_X}_U(V+W)=\nabla^g_UV + \nabla^{g_X}_U W$ holds for every $U \in \Gamma_{\mathrm{Bor}} (TX)$ and every $V, W \in \Gamma_{1} (TX)$.
\item $\nabla^{g_X}_{fU+hV}W =f\nabla^{g_X}_UW + h\nabla^{g_X}_VW$ holds for every $U, V \in \Gamma_{\mathrm{Bor}} (TX)$, every $W \in \Gamma_{1} (TX)$ and every Borel functions $f, h$ on $X$.
\item $\nabla^{g_X}_U(fV)=U(f)V + f\nabla^{g_X}_UV$ holds for every $U \in \Gamma_{\mathrm{Bor}} (TX)$, every $V \in \Gamma_{1} (TX)$ and every weakly Lipschitz function $f$ on $X$.
\item $\nabla^{g_X}_UV - \nabla^{g_X}_VU=[U, V]$ holds for every $U, V \in \Gamma_{1} (TX)$.
\item $Ug(V, W) = {g_X}( \nabla^{g_X}_UV, W) + g(V, \nabla^{g_X}_UW)$ holds for every $U \in \Gamma_{\mathrm{Bor}} (TX)$ and every $V, W \in \Gamma_{1} (TX)$.   
\end{enumerate}
\end{theorem}
\begin{remark}
$\nabla^{g_X}$ is \textit{local}, i.e., for every Borel subset $A$ of $X$, the Levi-Civita connection induces the map $\nabla^{g_X}|_A:\Gamma_{\mathrm{Bor}} (TA) \times \Gamma_{1} (TA) \to \Gamma_{\mathrm{Bor}} (TA)$ by letting $\nabla^{g_X}|_A(U, V):=\nabla^{g_X}_{1_AU}(1_AV)$.
Thus we use the same notation: $\nabla^{g_X}=\nabla^{g_X}|_A$ in this paper for brevity.
See Section $3$ in \cite{Ho3} for the detail.
\end{remark}
The Levi-Civita connection above allows us to give the definitions of the Hessian of a weakly twice differentiable function, and of the divergence of a weakly Lipschitz vector field in the ordinary way of Riemannian geometry.
We only give several fundamental properties of them:
\begin{proposition}\cite[Theorem $3.26$]{Ho3}\label{hess2}
Assume $g_X \in \Gamma_1(T^0_2X)$. 
Let $A$ be a Borel subset of $X$, $f$ a weakly twice differentiable function on $A$, $\omega \in \Gamma_1(T^*A)$ and $Y \in \Gamma_{1}(TA)$.
Then there exist uniquely
\begin{enumerate}
\item $\nabla^{g_X} \omega \in \Gamma_{\mathrm{Bor}}(T^0_2A)$ such that $\nabla^{g_X} \omega(U, V)=g_{X}(\nabla_{V}\omega^*, U)$ holds for every $U, V \in \Gamma_{\mathrm{Bor}}(TA)$,
\item the Hessian $\mathrm{Hess}^{g_X}_f :=\nabla^{g_X}df \in \Gamma_{\mathrm{Bor}}(T^0_2A)$,
\item a Borel function  $\mathrm{div}^{g_X}\,Y:= \mathrm{tr} (\nabla^{g_X}Y^*)$ on $A$,
\item a Borel function $\Delta^{g_X} f:=-\mathrm{div}^{g_X}\,(\nabla^{g_X}f)=- \mathrm{tr} (\mathrm{Hess}_{f}^{g_X})$ on $A$.
\end{enumerate}
Moreover we have the following:
\begin{itemize}
\item[(a)] $\mathrm{Hess}^{g_X}_f(x)$ is symmetric for a.e. $x \in A$.
\item[(b)] $\mathrm{div}^{g_X}\,(hY)=h\mathrm{div}^{g_X} Y +{g_X}(\nabla^{g_X} h, Y)$ holds for every weakly Lipschitz function $h$ on $A$. 
\item[(c)] $\Delta^{g_X} (fh)=h\Delta^{g_X} f-2{g_X}(\nabla^{g_X} f, \nabla^{g_X} h)+f\Delta^{g_X} h$ holds for every weakly twice differentiable function $h$ on $A$.
\end{itemize}
\end{proposition}
\begin{remark}
We can define the \textit{covariant derivative of tensor fields} $\nabla^{g_X}:\Gamma_1(T^r_sA) \to \Gamma_{\mathrm{Bor}}(T^r_{s+1}A)$ in the ordinary way of Riemannian geometry.
Then it is easy to check the \textit{torsion free condition}: $\nabla^{g_X}g_X\equiv 0$.
\end{remark}
\begin{definition}
Let $\hat{\mathcal{A}}$ be a weakly second order differential structure on $(X, \upsilon)$.
We say that \textit{$\mathcal{A}$ and $\hat{\mathcal{A}}$ are compatible } if so is $\mathcal{A} \cup \hat{\mathcal{A}}$.
\end{definition}
It is trivial that if $\mathcal{A}$ and $\hat{\mathcal{A}}$ are compatible, then the notions introduced here are compatible, i.e., for instance we see that 
a function $f$ on a Borel subset $A$ of $X$ is weakly twice differentiable on $A$ with respect to $\mathcal{A}$ if and only if so is $f$ with respect to $\hat{\mathcal{A}}$, $\Gamma_1(T^r_sA; \mathcal{A})=\Gamma_1(T^r_sA;  \hat{\mathcal{A}})$ and so on.
\begin{remark}
It is known that similar results given here hold on Alexandrov spaces.
See for instance \cite{bgp, KMS, ot, ot-sh, pere, pere1}.
\end{remark}
\subsection{Limit spaces of Riemannian manifolds.}
In this subsection we recall several fundamental properties about limit spaces of Riemannian manifolds with lower Ricci curvature bounds.
\subsubsection{Gromov-Hausdorff convergence.}
Let $\{(X_i, x_i)\}_{1 \le i \le \infty}$ be a sequence of pointed proper geodesic spaces.
We say that \textit{$(X_i, x_i)$  Gromov-Hausdorff converges to $(X_{\infty}, x_{\infty})$} if 
there exist sequences of positive numbers $\epsilon_i \to 0$, $R_i \to \infty$ and of maps $\psi_i:  B_{R_i}(x_i) \to B_{R_i}(x_{\infty})$
 (called \textit{$\epsilon_i$-almost isometries}) with $|\overline{x,y}-\overline{\psi_i(x), \psi_i(y)}|<\epsilon_i$ for every $x, y \in B_{R_i}(x_i)$,
$B_{R_i}(x_{\infty}) \subset B_{\epsilon_i}(\mathrm{Image} (\psi_i))$, and 
$\psi_i(x_i) \to x_{\infty}$ (then we denote it by $x_i \to x_{\infty}$ for short).
See \cite{gr}.
We denote it by $(X_i, x_i) \stackrel{(\psi_i, \epsilon_i, R_i)}{\to} (X_{\infty}, x_{\infty})$ or $(X_i, x_i) \to (X_{\infty}, x_{\infty})$ for short.

Assume $(X_i, x_i) \stackrel{(\psi_i, \epsilon_i, R_i)}{\to} (X_{\infty}, x_{\infty})$.
For a sequence $\{A_i\}_i$ of compact subsets $A_i$ of $B_{R_i}(x_i)$ for every $i \le \infty$,
we say that \textit{$A_i$ Gromov-Hausdorff converges to $A_{\infty}$ with respect to the convergence $(X_i, x_i) \to (X_{\infty}, x_{\infty})$} if $\psi_i(A_i)$ Hausdorff converges to $A_{\infty}$.
Then we often denote $A_{\infty}$ by $\lim_{i \to \infty}A_i$.
Moreover, for a sequence $\{\upsilon_i\}_{1\le i \le \infty}$ of Borel measures $\upsilon_i$ on $X_i$, 
we say that \textit{$\upsilon_{\infty}$ is the limit measure of $\{\upsilon_i\}_i$} if $\upsilon_i(B_r(y_i)) \to \upsilon_{\infty}(B_r(y_{\infty}))$ holds 
for every $r>0$ and every  $y_i \to y_{\infty}$.
See \cite{ch-co1, fu}.
Then we denote it by $(X_i, x_i, \upsilon_i) \stackrel{(\psi_i, \epsilon_i, R_i)}{\to} (X_{\infty}, x_{\infty}, \upsilon_{\infty})$ or $(X_i, x_i, \upsilon_i) \to (X_{\infty}, x_{\infty}, \upsilon_{\infty})$ for brevity.
\subsubsection{Ricci limit spaces.}
Let $n \in \mathbf{N}$, $K \in \mathbf{R}$ and let $(M_{\infty}, m_{\infty})$ be a pointed proper geodesic space.
We say that \textit{$(M_{\infty}, m_{\infty})$ is an $(n, K)$-Ricci limit space (of $\{(M_i, m_i)\}_i$)} if there exist
sequences of real numbers $K_i \to K$ and  of pointed complete $n$-dimensional Riemannian manifolds $\{(M_i, m_i)\}_i$ with $\mathrm{Ric}_{M_i} \ge K_i(n-1)$ such that
$(M_i, m_i) \to (M_{\infty}, m_{\infty})$.
We call an $(n, -1)$-Ricci limit space a \textit{Ricci limit space} for brevity.
Moreover we say that a Radon measure $\upsilon$ on $M_{\infty}$ is the \textit{limit measure of $\{(M_i, m_i)\}_i$}
if $\upsilon$ is the limit measure of $\{\mathrm{vol}/\mathrm{vol}\,B_1(m_i)\}_i$.
Then we say that $(M_{\infty}, m_{\infty}, \upsilon)$ \textit{is the Ricci limit space of} $\{(M_i, m_i, \mathrm{vol}/\mathrm{vol}\,B_1(m_i))\}_i$.
Throughout this paper we use the notation: $\underline{\mathrm{vol}}:=\mathrm{vol}/\mathrm{vol}\,B_1(m_i)$ for brevity.
\subsubsection{Poincar\'e inequality and Sobolev spaces.}
Let $(M_{\infty}, m_{\infty}, \upsilon)$ be the Ricci limit space of $\{(M_i, m_i, \underline{\mathrm{vol}})\}_i$ with $M_{\infty} \neq \{m_{\infty}\}$ (we will use the same notation in the subsections later).
Then it is known that the following hold: 
\begin{enumerate}
\item $(M_{\infty}, m_{\infty})$ satisfies a weak Poincar\'e inequality of type $(1, p)$ for every $1\le p<\infty$.
\item For every $1< p < \infty$ and every open subset $U \subset M_{\infty}$, $(1, p)$-Sobolev space $H_{1,p}(U)$ is well-defined.
\item For every $1<p<\infty$ and every $f \in H_{1,p}(U)$, $f$ is weakly Lipschitz on $U$ and $||f||_{H_{1,p}}=||f||_{L^p}+||df||_{L^p}$. 
\item The space of locally Lipschitz functions on $B_R(x_{\infty}) (\subset M_{\infty})$ in  $H_{1, p}(B_R(x_{\infty}))$ is dense in $H_{1, p}(B_R(x_{\infty}))$ for every $1<p<\infty$. 
\end{enumerate}
See \cite[Corollary $2.25$, $(4.3)$, Theorems $4.14$ and $4.47$]{ch1} and \cite[$(1. 6)$ and Theorem $2.15$]{ch-co3} for the details.
\subsubsection{Regular set.}
A pointed proper metric space $(X, x)$ is said to be a \textit{tangent cone of $M_{\infty}$ at $z_{\infty} \in M_{\infty}$} if there exists $r_i \to 0$ such that $(M_{\infty}, z_{\infty}, r_i^{-1}d_{M_{\infty}}) \to (X, x)$.
Let $\mathcal{R}_k:=\{z_{\infty} \in M_{\infty};$ Every tangent cone at $z_{\infty}$ of $M_{\infty}$ is isometric to $(\mathbf{R}^k, 0_k).\}$ and $\mathcal{R}=\bigcup_{i=1}^n\mathcal{R}_i$.    
We call $\mathcal{R}_k$ \textit{the $k$-dimensional regular set of $M_{\infty}$} and \textit{$\mathcal{R}$ the regular set of $M_{\infty}$}.
Cheeger-Colding showed that
$\upsilon (M_{\infty} \setminus \mathcal{R})=0$ \cite[Theorem $2.1$]{ch-co1}.
On the other hand, recently, Colding-Naber proved that
there exists a unique $k$ such that $\upsilon (\mathcal{R} \setminus \mathcal{R}_k)=0$ \cite[Theorem $1.12$]{co-na1} .
We call $k$ \textit{the dimension of $M_{\infty}$} and denote it by $\mathrm{dim}\,M_{\infty}$.
Note that an argument similar to the proof of \cite[Lemma $3.5$]{Ho} yields that for every rectifiable coordinate system $\mathcal{A}$ on $(M_{\infty}, m_{\infty})$ there exists a subrectifiable coordinate system $\hat{\mathcal{A}}$ of $\mathcal{A}$ such that each patch of $\hat{\mathcal{A}}$ is $k$-dimensional, where we say that a rectifiable coordinate system $\hat{A}$ on $(M_{\infty}, m_{\infty})$ is a \textit{subrectifiable coordinate system of $\mathcal{A}$} if for every $(\hat{C}_i^k, \hat{\phi}_i^k) \in \hat{\mathcal{A}}$ there exists $(C_{j}^k, \phi_j^k) \in \mathcal{A}$ such that $\hat{C}_i^k \subset C_j^k$ and $\phi_j^k|_{\hat{C}_i^k}\equiv \hat{\phi}_i^k$.
\subsubsection{Angles.}
We introduce a key notion \textit{angles} to give the definitions of \textbf{(W)} and of \textbf{(S)} as in Section $1$.
Let $C_x$ be \textit{the cut locus of $x \in M_{\infty}$} defined by $C_x:=\{y \in M_{\infty};\overline{x, y}+ \overline{y, z}>\overline{x,z}$ holds for every $z \in M_{\infty}$ with $z \neq y .\}$.
It is known $\upsilon(C_x)=0$ \cite[Theorem $3.2$]{ho}.
Let $p, x, q \in M_{\infty}$ with $x \not \in C_p \cup C_q$.
Then there exists a unique $\angle pxq \in [0, \pi]$ such that 
\[\cos \angle pxq=\lim_{t \to 0}\frac{2t^2-\overline{\gamma_p(t), \gamma_q(t)}^2}{2t^2}\]
holds for every minimal geodesics $\gamma_p$ from $x$ to $p$ and every $\gamma_q$ from $x$ to $q$.
Note that for every $p, q \in M_{\infty}$ we see that $\langle dr_p, dr_q\rangle(x)=\cos \angle pxq$ holds for a.e. $x \in M_{\infty}$. 
See \cite[Theorem $1.2$]{Ho3} for the details, and see \cite[Theorems $1.2$ and $1.3$]{co-na} for a very interesting example.
\subsubsection{Rectifiability.}
Cheeger-Colding proved that $(M_{\infty}, \upsilon)$ is rectifiable by harmonic functions.
More precisely, by combining with Colding-Naber's result \cite[Theorem $1.12$]{co-na1}, we have:
\begin{theorem}\cite[Theorems $3.3, 5.5$ and $5.7$]{ch-co3}\label{harm3}
There exists a rectifiable coordinate system $\mathcal{A}_h:=\{(C_{i}, \phi_{i})\}_{i} (\phi_i=(\phi_{i, 1}, \ldots, \phi_{i, k}): C_i \to \mathbf{R}^k)$ of $(M_{\infty}, \upsilon)$ such that the following holds:
There exists a subsequence $\{i(j)\}_j$ such that 
for every $l$, there exist $r>0$, sequences $\{x_{i(j)}\}_{j \le \infty}$ of $x_{i(j)} \in M_{i(j)}$ with $x_{i(j)} \to x_{\infty}$, and
 $\{f_{i(j), s}\}_{j \le \infty, s \le k}$ of $C(n)$-Lipschitz harmonic functions $f_{i(j), s}$ on $B_r(x_{i(j)})$ such that $C_l \subset B_r(x_{\infty})$, $f_{\infty, s}|_{C_l} \equiv \phi_{l, s}$
$f_{i(j), s} \to f_{\infty, s}$ on $B_r(x_{\infty})$ as $j \to \infty$ for every $s$.
\end{theorem}
See subsection $3.1.1$ for the definition of the pointwise convergence of $C^0$-functions: $f_i \to f_{\infty}$ with respect to the Gromov-Hausdorff topology.
On the other hand, the author proved that $(M_{\infty}, \upsilon)$ is rectifiable by distance functions:
\begin{theorem}\cite[Theorem $3.1$]{Ho}\label{dist3}
There exists a rectifiable coordinate system $\mathcal{A}_d:=\{(C_{i}, \phi_{i})\}_{i<\infty}$ of $(M_{\infty}, \upsilon)$ such that 
every $\phi_{i, s}$ is the distance function from a point in $M_{\infty}$.
\end{theorem}
\begin{remark}\label{radial}
Moreover it was shown in \cite{Ho} that 
for every dense subset $A$ of $M_{\infty}$,
there exists a rectifiable coordinate system $\mathcal{A}_d=\{(C_{i}, \phi_{i})\}_{i<\infty}$ such that every $\phi_{i, s}$ is the distance function from a point in $A$.
\end{remark}
Note that Theorems \ref{harm3} and \ref{dist3} perform crucial roles in Section $3$ and $4$.
\begin{definition}
Let $\mathcal{A}:=\{(C_i, \phi_i)\}_i$ be a rectifiable coordinate system on $(M_{\infty}, \upsilon)$.
We say that \textit{$\mathcal{A}$ is a rectifiable coordinate system associated with $\{(M_i, m_i, \underline{\mathrm{vol}})\}_{i < \infty}$} if 
for every $i <\infty$ there exists a sequence $\{\phi_{i, l, j}\}_{1 \le l \le k, j\le \infty}$ of Lipschitz functions $\phi_{i, l, j}$ on $M_j$ such that $\sup_{1 \le l \le k, j \le \infty}\mathbf{Lip}\phi_{i, l, j}<\infty$,  $\phi_{i, l, \infty}|_{C_i} \equiv \phi_{i, l}$ and that $(\phi_{i, l, j}, d\phi_{i, l, j}) \to (\phi_{i, l, \infty}, d\phi_{i, l, \infty})$ holds on $C_i$ as $j \to \infty$ for every $l$,
where $\mathbf{Lip}\phi_{i, l, j}$ is the Lipschitz constant of $\phi_{i, l, j}$: $\mathbf{Lip}\phi_{i, l, j}:=\sup_{x \neq y}|\phi_{i, l, j}(x)-\phi_{i, l, j}(y)|/\overline{x, y}$.
\end{definition}
See \cite[Definition $4.4$]{Ho} (or Definition \ref{inftyinfty}) for the definition of a pointwise convergence $df_i \to df_{\infty}$ of the differentials of Lipschitz functions with respect to the Gromov-Hausdorff topology.
Note that by \cite[Corollary $4.5$]{Ho} (or Proposition \ref{harm5}) and \cite[Proposition $4.8$]{Ho} (or Proposition \ref{dist}), for $\mathcal{A}_{h}$ and $\mathcal{A}_{d}$ as in Theorems \ref{harm3} and \ref{dist3}, respectively, we have the following:
\begin{enumerate}
\item $\mathcal{A}_{h}$ is a rectifiable coordinate system of $(M_{\infty}, \upsilon)$ associated with $\{(M_{i(j)}, m_{i(j)}, \underline{\mathrm{vol}})\}_{j < \infty}$.
\item $\mathcal{A}_d$ is a rectifiable coordinate system of $(M_{\infty}, \upsilon)$ associated with $\{(M_{i}, m_{i}, \underline{\mathrm{vol}})\}_{i < \infty}$. 
\end{enumerate}
\subsubsection{Weakly second order differential structure.}
In \cite{Ho3} it was proved that $(M_{\infty}, \upsilon)$ has a weakly second order differential structure.
More precisely, we have:
\begin{theorem}\cite[Theorem $4.13$]{Ho3}\label{2nd}
Let $\mathcal{A}_{2nd}:=\{(C_i, \phi_i)\}_{i<\infty}$ be a rectifiable coordinate system on $M_{\infty}$.
Assume that 
for every $i$ there exist $r>0$, sequences $\{x_j\}_{j \le \infty}$ of $x_{j} \in M_{j}$ and $\{\phi_{i, l, j}\}_{1 \le l \le k, j \le \infty}$ of Lipschitz functions $\phi_{i, l, j}$ on $B_r(x_j)$ such that $\sup_{1 \le l \le k, j \le \infty}\mathbf{Lip}\phi_{i, l, j}<\infty$, $x_j \to x_{\infty}$, $C_i \subset B_r(x_{\infty})$, $\phi_{i, l, \infty}|_{C_i} \equiv \phi_{i, l}$, $\phi_{i, l, j} \in C^2(B_r(x_j))$ holds for every $j<\infty$ with $\sup_{1\le l \le k, j<\infty}||\Delta \phi_{i, l, j}||_{L^2(B_r(x_j))}<\infty$, and that $\phi_{i, l, j} \to \phi_{i, l, \infty}$ on $B_r(x_{\infty})$ as $j \to \infty$. 
Then we see that $\mathcal{A}_{2nd}$ is a weakly second order differential structure on $M_{\infty}$, and that the Riemannian metric $g_{M_{\infty}}$ is weakly Lipschitz with respect to $\mathcal{A}_{2nd}$.
\end{theorem}
We say that \textit{$\mathcal{A}_{2nd}$ as in Theorem \ref{2nd} is a weakly second order differential structure on $(M_{\infty}, \upsilon)$ associated with $\{(M_i, m_i, \underline{\mathrm{vol}})\}_i$}.
Note that Theorem \ref{harm3} and \cite[Corollary $4.5$]{Ho} yield the following: 
\begin{enumerate}
\item There exist a subsequence $\{(M_{i(j)}, m_{i(j)}, \underline{\mathrm{vol}})\}_j$ and a weakly second order differential structure $\mathcal{A}$ on $(M_{\infty}, \upsilon)$ such that $\mathcal{A}$ is associated with $\{(M_{i(j)}, m_{i(j)}, \underline{\mathrm{vol}})\}_j$.
\item Let $\mathcal{A}_{2nd}$ be a weakly second order differential structure on $(M_{\infty}, \upsilon)$ associated with $\{(M_i, m_i, \underline{\mathrm{vol}})\}_i$.
\begin{enumerate}
\item $\mathcal{A}_{2nd}$ is a rectifiable coordinate system of $(M_{\infty}, \upsilon)$ associated with $\{(M_i, m_i, \underline{\mathrm{vol}})\}_i$.
\item Let $\{i(j)\}_j$ be a subsequence of $\mathbf{N}$ and $\hat{\mathcal{A}}_{2nd}$ a weakly second order differential structure on $(M_{\infty}, \upsilon)$ associated with $\{(M_{i(j)}, m_{i(j)}, \underline{\mathrm{vol}})\}_j$.
Then $\mathcal{A}_{2nd}$ and $\hat{\mathcal{A}}_{2nd}$ are compatible. 
\end{enumerate}
\end{enumerate}
\section{$L^p$-convergence}
\subsection{Functions.}
In this subsection 
we will discuss several convergences of functions with respect to the Gromov-Hausdorff topology.
Throughout this subsection we will always consider the following setting:
Let $\{(X_i, x_i)\}_{1 \le i \le \infty}$ be a sequence of pointed proper geodesic spaces, $R>0$ and $\upsilon_i$ a Radon measure on $X_i$ for every $i \le \infty$ satisfying the following: 
\begin{enumerate}
\item $(X_i, x_i, \upsilon_i) \stackrel{(\psi_i, \epsilon_i, R_i)}{\to} (X_{\infty}, x_{\infty}, \upsilon_{\infty})$ and $X_{\infty} \neq \{x_{\infty}\}$. 
\item For every $\hat{R}>0$ there exists $\kappa=\kappa (\hat{R})\ge 0$ such that $\upsilon_i(B_{2r}(z_i))\le 2^{\kappa}\upsilon_i(B_r(z_i))$ holds for every $i\le \infty$, every $r<\hat{R}$ and every $z_i \in X_i$.
\item $\upsilon_i(B_1(x_i))=1$ for every $i \le \infty$.
\end{enumerate}
\subsubsection{$C^0$-functions}
We first give the definition of a pointwise convergence of continuous functions with respect to the Gromov-Hausdorff topology:
\begin{definition}
Let $f_i \in C^0(B_R(x_i))$ for every $i \le \infty$.
We say that \textit{$f_i$ converges to $f_{\infty}$ at $z_{\infty} \in B_R(x_{\infty})$} if $f_i(z_i) \to f_{\infty}(z_{\infty})$ holds for every $z_i \to z_{\infty}$.
Then we denote it by $f_i \to f_{\infty}$ at $z_{\infty}$.
\end{definition}
We also say that \textit{$f_i$ converges to $f_{\infty}$ on a subset $A$ of $B_R(x_{\infty})$} if $f_i$ converges to $f_{\infty}$ at every $z_{\infty} \in A$ (then we denote it by $f_i \to f_{\infty}$ on $A$).
\begin{definition}
Let $f_i \in C^0(B_R(x_i))$ for every $i < \infty$.
We say that \textit{$\{f_i\}_i$ is asymptotically uniformly equicontinuous on $B_R(x_{\infty})$} if for every $\epsilon>0$ there exist $i_0$ and $\delta>0$ such that 
$|f_j(\alpha_j)-f_j(\beta_j)|<\epsilon$ holds for every $j \ge i_0$ and every $\alpha_j, \beta_j \in B_R(x_j)$ with $\overline{\alpha_j, \beta_j}<\delta$.
\end{definition}
The following compactness result performs a crucial role in the next subsection:
\begin{proposition}\label{conti com}
Let $f_i \in C^0(B_R(x_i))$ for every $i<\infty$ with $\sup_{i<\infty}||f_i||_{L^{\infty}}< \infty$.
Assume that $\{f_i\}_{i< \infty}$ is asymptotically uniformly equicontinuous on $B_R(x_{\infty})$.
Then there exist $f_{\infty} \in C^0(B_R(x_{\infty}))$ and a subsequence $\{f_{i(j)}\}_j$ of $\{f_i\}_i$ such that $f_{i(j)} \to f_{\infty}$ on $B_R(x_{\infty})$.
\end{proposition}
\begin{proof}
Let $\{z_i\}_i$ be a countable dense subset of $B_R(x_{\infty})$.
Since $\sup_{i<\infty}||f_i||_{L^{\infty}}< \infty$, there exists a subsequence $\{i(j)\}_j$ such that $\{f_{i(j)}(z_k)\}_j$ is a convergent sequence in $\mathbf{R}$ for every $k$.
Let us denote by $a(x_k)$ the limit.
The assumption yields that the function $a: \{z_k\}_k \to \mathbf{R}$ is uniformly continuous.
Therefore there exists a unique $f_{\infty} \in C^0(B_R(x_{\infty}))$ such that $f_{\infty}(x_k)=a(x_k)$ holds for every $k$. 
Then it is not difficult to check that $f_{i(j)} \to f_{\infty}$ on $B_R(x_{\infty})$.
\end{proof}
\subsubsection{$L^1_{\mathrm{loc}}$-functions and $L^{\infty}$-functions.}
Let $f_i \in L^{1}_{\mathrm{loc}}(B_R(x_i))$ for every $i \le \infty$.
We start this subsection by giving the definition of a pointwise convergence of \textbf{(W)} as in Section $1$ for $L^1_{\mathrm{loc}}$-functions:
\begin{definition}\label{weak1}
We say that \textit{$\{f_i\}_i$ is weakly upper semicontinuous at $z_{\infty} \in B_R(x_{\infty})$} if 
\[\liminf_{r \to 0}\left(\frac{1}{\upsilon_{\infty}(B_r(z_{\infty}))}\int_{B_r(z_{\infty})}f_{\infty}d\upsilon_{\infty}-\limsup_{i \to \infty}\frac{1}{\upsilon_{i}(B_r(z_i))}\int_{B_r(z_i)}f_id\upsilon_i \right) \ge 0\] 
holds for every $z_i \to z_{\infty}$.
We say that \textit{$\{f_i\}_i$ is weakly lower semicontinuous at $z_{\infty} \in B_R(x_{\infty})$} if 
\[\liminf_{r \to 0} \left( \liminf_{i \to \infty}\frac{1}{\upsilon_i(B_r(z_i))}\int_{B_r(z_i)}f_id\upsilon_i- \frac{1}{\upsilon_{\infty}(B_r(z_{\infty}))}\int_{B_r(z_{\infty})}f_{\infty}d\upsilon_{\infty}\right) \ge 0\] 
holds for every $z_i \to z_{\infty}$.
We say that \textit{$f_i$ converges weakly to $f_{\infty}$ at $z_{\infty}$} if $\{f_i\}_i$ is weakly upper and lower semicontinuous at $z_{\infty}$.  
\end{definition}
We first give a fundamental property of the lower semicontinuity:
\begin{proposition}\label{135}
Assume that $f_i \ge 0$ holds for every $i \le \infty$, and that $\{f_i\}_i$ is weakly lower semicontinuous at a.e. $z_{\infty} \in B_R(x_{\infty})$.
Then
\[\liminf_{i \to \infty}\int_{B_R(x_i)}f_id\upsilon_i \ge \int_{B_R(x_{\infty})}f_{\infty}d\upsilon_{\infty}.\] 
\end{proposition}
\begin{proof}
Without loss of generality we can assume $\sup_{i<\infty}||f_i||_{L^1(B_R(m_i))}<\infty$.
There exists $K \subset B_R(x_{\infty})$ with $\upsilon_{\infty}(B_R(x_{\infty}) \setminus K)=0$ such that for every $z_{\infty} \in K$ and every $\epsilon>0$ there exists $r=r(z_{\infty}, \epsilon)>0$ such that  
\[\liminf_{i \to \infty}\frac{1}{\upsilon_i(B_t(z_i))}\int_{B_t(z_i)}f_id\upsilon_i \ge \frac{1}{\upsilon_{\infty}(B_t(z_{\infty}))}\int_{B_t(z_{\infty})}f_{\infty}d\upsilon_{\infty}-\epsilon\]
holds for every $t<r$.
Fix $\epsilon>0$.
A standard covering argument (c.f. \cite[Proposiiton $2.2$]{Ho}) yields that there exists a countable pairwise disjoint collection 
$\{\overline{B}_{r_i}(w_i)\}_i$ such that $\overline{B}_{5r_i}(w_i) \subset B_R(x_{\infty})$, $w_i \in K$, $5r_i<r(w_i, \epsilon)$ and that 
$K \setminus \bigcup_{i=1}^N\overline{B}_{r_i}(w_i) \subset \bigcup_{i=N+1}^{\infty}\overline{B}_{5r_i}(w_i)$ holds for every $N$.
Let $N_0$ with $\sum_{i=N_0 +1}^{\infty}\upsilon_{\infty}(B_{5r_i}(w_i))<\epsilon$ and $K^{\epsilon}=K \cap \bigcup_{i=1}^{N_0}\overline{B}_{r_i}(w_i)$.
Then we see that
\begin{align*}
\int_{K^{\epsilon}}f_{\infty}d\upsilon_{\infty}\le \sum_{i=1}^{N_0}\int_{B_{r_i}(w_i)}f_{\infty}d\upsilon_{\infty}&\le \sum_{i=1}^{N_0}\left(\int_{B_{r_i}(w_{i, j})}f_j\upsilon_j + \epsilon \upsilon_j(B_{r_i}(w_{i, j}))\right) \\
&\le \int_{B_R(x_j)}f_jd\upsilon_j +\epsilon \upsilon_j(B_R(x_j))
\end{align*}
holds for every sufficiently large $j$, where $w_{i, j} \to w_{i}$ as $j \to \infty$.
Since $\upsilon_{\infty}(B_R(x_{\infty}) \setminus K^{\epsilon})<\epsilon$, by letting $j \to \infty$ and $\epsilon \to 0$, the dominated convergence theorem yields the assertion.
\end{proof}
\begin{corollary}\label{ell1}
Assume that $f_i$ converges weakly to $f_{\infty}$ at a.e. $z_{\infty} \in B_R(x_{\infty})$.
Then $\liminf_{i \to \infty}||f_i||_{L^1(B_R(x_i))} \ge ||f_{\infty}||_{L^1(B_R(x_{\infty}))}$.
\end{corollary}
\begin{proof}
Lebesgue's differentiation theorem yields that there exists $K_{\infty} \subset B_R(x_{\infty})$ such that $\upsilon_{\infty}(B_R(x_{\infty}) \setminus K_{\infty})=0$ and that
\[\lim_{r \to 0}\left|\frac{1}{\upsilon_{\infty}(B_r(z_{\infty}))}\int_{B_r(z_{\infty})}f_{\infty}d\upsilon_{\infty}\right|=\lim_{r \to 0}\frac{1}{\upsilon_{\infty}(B_r(z_{\infty}))}\int_{B_r(z_{\infty})}|f_{\infty}|d\upsilon_{\infty}\]
holds for every $z_{\infty} \in K_{\infty}$.
Thus for a.e. $z_{\infty} \in K_{\infty}$ and every $\epsilon>0$ there exists $r>0$ such that for every $t<r$ we see that 
\begin{align*}
\frac{1}{\upsilon_{\infty}(B_t(z_{\infty}))}\int_{B_t(z_{\infty})}|f_{\infty}|d\upsilon_{\infty} &\le \left| \frac{1}{\upsilon_{\infty}(B_t(z_{\infty}))}\int_{B_t(z_{\infty})}f_{\infty}d\upsilon_{\infty} \right|+\epsilon \\
&\le \left| \frac{1}{\upsilon_{i}(B_t(z_{i}))}\int_{B_t(z_{i})}f_{i}d\upsilon_i \right|+2\epsilon \le \frac{1}{\upsilon_{i}(B_t(z_{i}))}\int_{B_t(z_{i})}|f_{i}|d\upsilon_i +2\epsilon
\end{align*}
holds for every sufficiently large $i$, i.e., $\{|f_i|\}_i$ is weakly lower semicontinuous at a.e. $z_{\infty} \in B_R(x_{\infty})$.
Thus the assertion follows directly from Proposition \ref{135}.
\end{proof}
Note that in general the weak convergence of $f_i \to f_{\infty}$ does NOT imply the weak convergence of $|f_i| \to |f_{\infty}|$. 
See for instance Remark \ref{zigzag}.
We will give more fundamental properties about the weak convergence in the next subsection.

We now give the definition of a strong convergence of $L^1_{\mathrm{loc}}$-functions:
\begin{definition}\label{8}
We say that \textit{$f_i$ converges strongly to $f_{\infty}$ at $z_{\infty}$} if 
\[\lim_{t \to 0}\left(\limsup_{i \to \infty} \frac{1}{\upsilon_i (B_t(z_i))}\int_{B_t(z_i)}\left|f_i-\frac{1}{\upsilon_{\infty} (B_t(z_{\infty}))}\int_{B_t(z_{\infty})}f_{\infty}d\upsilon_{\infty}\right|d\upsilon_i \right)=0 \]
and 
\[\lim_{t \to 0}\left( \limsup_{i \to \infty} \frac{1}{\upsilon_{\infty} (B_t(z_{\infty}))}\int_{B_t(z_{\infty})}\left|f_{\infty}-\frac{1}{\upsilon_{i} (B_t(z_{i}))}\int_{B_t(z_{i})}f_{i}d\upsilon_{i}\right|d\upsilon_{\infty}\right)=0\]
hold for every $z_i \to z_{\infty}$.
\end{definition}
\begin{remark}\label{1234}
Let $g_i \in C^0(B_R(x_i))$ and $z_i \in B_R(x_i)$ for every $i \le \infty$ with $z_i \to z_{\infty}$.
Assume that there exists $r>0$ such that $\{g_i|_{B_r(z_i)}\}_{i<\infty}$ is asymptotically uniformly equicontinuous at $z_{\infty}$.
Then it is not difficult to check that $g_i$ converges strongly $g_{\infty}$ at $z_{\infty}$ if and only if $g_i$ converges weakly to $g_{\infty}$ at $z_{\infty}$ if and only if $g_i \to g_{\infty}$ at $z_{\infty}$.
\end{remark}
\begin{remark}\label{easy}
It is easy to check that if $f_i$ converges strongly to $f_{\infty}$ at $z_{\infty}$, then the following hold:
\begin{enumerate}
\item $f_i$ converges weakly to $f_{\infty}$ at $z_{\infty}$.
\item $|f_i|$ converges strongly to $|f_{\infty}|$ at $z_{\infty}$.
\end{enumerate}
\end{remark}
\begin{remark}\label{zigzag}
Let $g_n$ be a smooth function on $\mathbf{R}$ satisfying that
\[g_n(x)=(-1)^i\left(\frac{2n^2-2}{n-2}x-\frac{(n^2-1)(2i+1)}{n(n-2)}\right)\]
holds for every $x \in [i/n+1/n^2, (i+1)/n-1/n^2]$ and every $i \in \mathbf{Z}$, and that 
\[|g_n(x)-(-1)^{i-1}|\le \frac{100}{n^2}\]
holds for every $x \in [i/n-1/n^2, i/n+1/n^2]$ and every $i \in \mathbf{Z}$.
Note that $g_n((i+1)/n-1/n^2)=(-1)^i(1-1/n^2)=-g_n(i/n+1/n^2)$ holds for every $i \in \mathbf{Z}$.
Then it is easy to check that under the canonical convergence $(\mathbf{R}, 0,  H^1) \stackrel{(id_{\mathbf{R}}, \epsilon_i, R_i)}{\to} (\mathbf{R}, 0, H^1)$, for every $t \in \mathbf{R}$ and every $p>0$, we have the following:
\begin{enumerate}
\item $g_n$ converges weakly to $0$ at $t$.
\item $g_n$ does NOT converges strongly to $0$ at $t$.
\item $|g_n|^p$ converges weakly to $1/(p+1)$ at $t$.
\item $|g_n|^p$ does NOT converges strongly to $1/(p+1)$ at $t$.
\end{enumerate}
\end{remark}
We now recall a fundamental property of the strong convergence for $L^{\infty}$-functions given in \cite{Ho}:
\begin{proposition}\cite[Proposition $4.1$]{Ho}\label{strong3}
Let $k \in \mathbf{N}$, $R>0$, $w_{\infty} \in B_R(x_{\infty})$, $\{F_i\}_{1 \le i \le \infty} \subset C^0(\mathbf{R}^k)$ and let $\{f_i^l\}_{1 \le l \le k, i \le \infty}$ be a sequence of $L^{\infty}$-functions $f_i^l$ on $B_R(x_i)$ with
$\sup_{i, l} ||f_i^l||_{L^{\infty}(B_R(x_i))} < \infty$.
Assume that $f_i^l$ converges strongly to $f_{\infty}^l$ at $w_{\infty}$ for every $l$, and that 
$F_i$ converges to $F_{\infty}$ with respect to the compact uniform topology.
Then $F_i(f_i^1, \ldots, f_i^k)$ converges strongly  
to $F_{\infty}(f_{\infty}^1, \ldots, f_{\infty}^k)$ at $w_{\infty}$.
\end{proposition}
For $L^2_{\mathrm{loc}}$-functions we have the following:
\begin{proposition}\label{strong}
Let $A \subset B_R(x_{\infty})$.
Assume that $f_i \in L^2_{\mathrm{loc}}(B_R(x_i))$ holds for every $i \le \infty$, $f_i$ converges weakly to $f_{\infty}$ at a.e. $z_{\infty} \in A$, and that $\{(f_i)^2\}_i$ is weakly upper semicontinuous at a.e. $z_{\infty} \in A$. 
Then
$f_i$ converges strongly to $f_{\infty}$ at a.e. $z_{\infty} \in A$.
\end{proposition}
\begin{proof}
Lebesgue's differentiation theorem yields that the following holds for a.e. $z_{\infty} \in A$: For every $\epsilon >0$ there exists $r>0$ such that for every $t<r$ there exists $i_0 \in \mathbf{N}$ such that   
\begin{align*}
&\frac{1}{\upsilon_i(B_t(z_i))}\int_{B_t(z_i)}\left|f_i-\frac{1}{\upsilon_{\infty}(B_t(z_{\infty}))}\int_{B_t(z_{\infty})}f_{\infty}d\upsilon_{\infty}\right|^2d\upsilon_i\\
&=\frac{1}{\upsilon_i(B_t(z_i))}\int_{B_t(z_i)}|f_i|^2d\upsilon_i-2\left(\frac{1}{\upsilon_i(B_t(z_i))}\int_{B_t(z_i)}f_id\upsilon_i\right)
\left(\frac{1}{\upsilon_{\infty}(B_t(z_{\infty}))}\int_{B_t(z_{\infty})}f_{\infty}d\upsilon_{\infty}\right)\\
& \,\,\,+\left(\frac{1}{\upsilon_{\infty}(B_t(z_{\infty}))}\int_{B_t(z_{\infty})}f_{\infty}d\upsilon_{\infty}\right)^2\\
&\le \frac{1}{\upsilon_{\infty}(B_t(z_{\infty}))}\int_{B_t(z_{\infty})}|f_{\infty}|^2d\upsilon_{\infty}-2\left(\frac{1}{\upsilon_{\infty}(B_t(z_{\infty}))}\int_{B_t(z_{\infty})}f_{\infty}d\upsilon_{\infty}\right)
\left(\frac{1}{\upsilon_{\infty}(B_t(z_{\infty}))}\int_{B_t(z_{\infty})}f_{\infty}d\upsilon_{\infty}\right)\\
& \,\,\,+\left(\frac{1}{\upsilon_{\infty}(B_t(z_{\infty}))}\int_{B_t(z_{\infty})}f_{\infty}d\upsilon_{\infty}\right)^2 +\epsilon < 2\epsilon
\end{align*}
holds for every $i \ge i_0$.
Similarly, for every $i \ge i_0$ we have 
\[\frac{1}{\upsilon_{\infty}(B_t(z_{\infty}))}\int_{B_t(z_{\infty})}\left|f_{\infty}-\frac{1}{\upsilon_{i}(B_t(z_{i}))}\int_{B_t(z_{i})}f_{i}d\upsilon_{i}\right|^2d\upsilon_{\infty}<2\epsilon.\]
Thus the assertion follows from the Cauchy-Schwartz inequality.
\end{proof}
\begin{remark}\label{77}
Propositions \ref{strong3} and \ref{strong} yield that if $f_i \in L^{\infty}(B_R(x_i))$ holds for every $i \le \infty$ with $\sup_{i \le \infty}||f_i||_{L^{\infty}}<\infty$, then the assumptions of Proposition \ref{strong} hold if and only if $f_i$ converges strongly to $f_{\infty}$ at a.e. $z_{\infty} \in A$.
\end{remark}
The following proposition performs a crucial rule in Section $4$.
Note that it was proved essentially in \cite{Ho, Ho3}:
\begin{proposition}\label{angle2}
Assume that there exist $n \in \mathbf{N}$ and $K \in \mathbf{R}$ such that every $(X_i, x_i, \upsilon_i)$ is an $(n, K)$-Ricci limit space.
Let $p_i, q_i \in X_i$ for every $i \le \infty$ with $p_i \to p_{\infty}$ and $q_i \to q_{\infty}$, and put $h_i(x):=\cos \angle p_ixq_i$ for every $x \not \in C_{p_i} \cup C_{q_i}$.
Then $h_i$ converges strongly to $h_{\infty}$ on $X_{\infty} \setminus (C_{p_{\infty}} \cup C_{q_{\infty}})$.
\end{proposition}
\begin{proof}
Since $\langle dr_{p_i}, dr_{q_i} \rangle (x) = \cos \angle p_ixq_i$ holds for a.e. $x \in X_i$, 
the assertion follows from \cite[Proposition $4.3$]{Ho}. 
\end{proof}
\subsubsection{$L^p$-functions}
Let $1 < p \le \infty$, $x_i \in X_i$, $f_i \in L^p(B_R(x_i))$ for every $i \le \infty$ with
$L:=\sup_{i \le \infty}||f_i||_{L^p(B_R(x_i))} < \infty$ and let $q$ be the conjugate exponent of $p$ (i.e., $p^{-1}+q^{-1}=1$ holds).
\begin{proposition}\label{weak0}
Let $z_i, \hat{z}_i \in B_R(x_i)$ for every $i\le \infty$ with $z_i \to z_{\infty}$, $\hat{z}_i \to \hat{z}_{\infty}$ and $z_{\infty}=\hat{z}_{\infty}$.
Then 
\[\lim_{i \to \infty}\left|\frac{1}{\upsilon_i(B_r(z_i))}\int_{B_r(z_i)}f_id\upsilon_i-\frac{1}{\upsilon_i(B_r(\hat{z}_i))}\int_{B_r(\hat{z}_i)}f_id\upsilon_i\right|=0\]
holds for every $r>0$ with $B_r(z_{\infty}) \subset B_R(x_{\infty})$.
\end{proposition}
\begin{proof}
\cite[Lemma $3.3$]{co-mi7} yields $\lim_{i \to \infty}\upsilon_i(B_r(z_i) \Delta B_r(\hat{z}_i))=0$, where $A \Delta B :=(A \setminus B) \cup (B \setminus A)$.
Thus the H$\ddot{\text{o}}$lder inequality yields
\[\left|\int_{B_r(z_i)}f_id\upsilon_i-\int_{B_r(\hat{z}_i)}f_id\upsilon_i\right| \le \left(\upsilon_i(B_r(z_i) \Delta B_r(\hat{z}_i))\right)^{1/q}||f_i||_{L^p(B_R(x_i))} \to 0\]
as $i \to \infty$.
\end{proof}
It is easy to check the following proposition.
Compare with \cite[Proposition $4.5$]{Ho}:
\begin{proposition}\label{linear}
Let $g_i \in L^{\infty}(B_R(x_i))$ for every $i \le \infty$ with $\sup_{i\le \infty}||g_i||_{L^{\infty}}<\infty$.
Assume that $g_i$ converges strongly to $g_{\infty}$ at $z_{\infty} \in B_R(x_{\infty})$ and that $f_i$ converges weakly to $f_{\infty}$ at $z_{\infty}$.
Then $g_if_i$ converges weakly to $g_{\infty}f_{\infty}$ at $z_{\infty}$.  
\end{proposition}
The next proposition is a fundamental result of the weak convergence.
Compare with Corollary \ref{ell1}:
\begin{proposition}\label{weak2}
Assume  that 
\[\liminf_{r \to 0}\left(\limsup_{i \to \infty}\left| \frac{1}{\upsilon_i (B_r(z_i))}\int_{B_r(z_i)}f_id\upsilon_i-\frac{1}{\upsilon_{\infty} (B_r(z_{\infty}))}\int_{B_r(z_{\infty})}f_{\infty}d\upsilon_{\infty} \right|\right)=0\]
holds for a.e. $z_{\infty} \in B_R(x_{\infty})$, where $z_i \to z_{\infty}$.
Then 
\[\lim_{i \to \infty}\int_{B_R(x_i)}f_id\upsilon_i=\int_{B_R(x_{\infty})}f_{\infty}d\upsilon_{\infty}.\]
\end{proposition}
\begin{proof}
By Proposition \ref{weak0}, there exists $K_{\infty} \subset B_R(x_{\infty})$ such that $\upsilon_{\infty}(B_R(x_{\infty}) \setminus K_{\infty})=0$ and that for every $z_{\infty} \in K_{\infty}$, every $\epsilon>0$ and every $\delta>0$ there exists $r:=r(z_{\infty}, \epsilon, \delta)>0$ with $r<\delta$ such that
\[\limsup_{i \to \infty}\left|\frac{1}{\upsilon_i(B_r(z_i))}\int_{B_r(z_i)}f_id\upsilon_i-\frac{1}{\upsilon_{\infty}(B_r(z_{\infty}))}\int_{B_r(z_{\infty})}f_{\infty}d\upsilon_{\infty}\right|<\epsilon \]
holds for every $z_i \to z_{\infty}$. 
Fix $\epsilon>0$.
Applying a standard covering argument to $\mathcal{B}:=\{\overline{B}_{r(z_{\infty}, \epsilon, 1/k)}(z_{\infty})\}_{z_{\infty} \in K_{\infty}, k \in \mathbf{N}}$ yields that
there exists a countable pairwise disjoint collection $\{\overline{B}_{r_i}(z_i)\}_i \subset \mathcal{B}$ such that  $K_{\infty} \setminus \bigcup_{i=1}^N\overline{B}_{r_i}(z_i) \subset \bigcup_{i=N+1}^{\infty}\overline{B}_{5r_i}(z_i) \subset B_R(x_{\infty})$ holds for every $N$.
Fix $N_0$ with $\sum_{i=N_0+1}^{\infty}\upsilon_{\infty}(B_{5r_i}(z_i)) <\epsilon$.
Then the H$\ddot{\text{o}}$lder inequality yields that 
\begin{align*}
\int_{B_R(x_j)\setminus \bigcup_{i=1}^{N_0}B_{r_i}(z_{i, j})}|f_{j}|d\upsilon_j &\le ||f_j||_{L^p(B_R(x_j))}\left( \upsilon_j \left(B_R(x_j)\setminus \bigcup_{i=1}^{N_0}B_{r_i}(z_{i, j})\right) \right)^{1/q} \\
& \le L\epsilon^{1/q}
\end{align*}
holds for every sufficiently large $j \le \infty$, where $z_{i, j} \to z_i$ as $j \to \infty$.
Thus 
\begin{align*}
\int_{B_R(x_{\infty})}f_{\infty}d\upsilon_{\infty}&=\sum_{i=1}^{N_0}\int_{B_{r_i}(z_i)}f_{\infty}d\upsilon_{\infty} \pm \Psi (\epsilon; L, p)\\
&=\sum_{i=1}^{N_0}\int_{B_{r_i}(z_{i,j})}f_{j}d\upsilon_{j} \pm \Psi (\epsilon; L, p, \upsilon_{\infty}(B_R(x_{\infty})))\\
&=\int_{B_R(x_{j})}f_{j}d\upsilon_{j} \pm \Psi (\epsilon; L, p, \upsilon_{\infty}(B_R(x_{\infty})))
\end{align*}
holds for every sufficiently large $j$.
Therefore we have the assertion.
\end{proof}
The next corollary is a direct consequence of Proposition \ref{weak2}.
\begin{corollary}\label{098}
If $f_i$ converges weakly to $f_{\infty}$ at a.e. $z_{\infty} \in B_R(x_{\infty})$, then $f_i$ converges weakly to $f_{\infty}$ on $B_R(x_{\infty})$.
\end{corollary}
We now give a compactness result for the weak convergence:
\begin{proposition}\label{weak com}
Let $g_i \in L^p(B_R(x_i))$ for every $i<\infty$ with $\sup_{i<\infty}||g_i||_{L^p}<\infty$.
Then there exist $g_{\infty} \in L^p(B_R(x_{\infty}))$ and a subsequence $\{g_{i(j)}\}_j$ of $\{g_i\}_i$ such that $g_{i(j)}$ converges weakly to $g_{\infty}$ on $B_R(x_{\infty})$. 
\end{proposition}
\begin{proof}
We only give a proof of the case $p<\infty$ because the proof of the case $p=\infty$ is similar.
Define $g_i(w)\equiv 0$ on $M_i \setminus B_R(x_i)$.
By a decomposition $g\equiv g_+ - g_-$, where $g_+:= \max \{g, 0\}$ and $g_-:=\max \{-g, 0\}$,
without loss of generality we can assume that $g_i \ge 0$ holds for every $i<\infty$.
For every $r>0$ and every $i<\infty$, we define a function $g_i^r$ on $B_R(m_i)$ by
\[g_i^r(z):=\frac{1}{\upsilon_i(B_r(z))}\int_{B_r(z)}g_id\upsilon_i.\]
By \cite[Lemma $3.3$]{co-mi7} and the H$\ddot{\text{o}}$lder inequality, for every $r>0$ it is easy to check that $\{g_i^r\}_i$ 
is asymptotically uniformly equicontinuous on $B_R(x_{\infty})$ and that  $\sup_{i<\infty}||g_i^r||_{L^{\infty}}<\infty$.
Thus Proposition \ref{conti com} yields that there exist $\{g_{\infty}^r\}_{r \in \mathbf{Q}_{>0}} \subset C^0(B_R(x_{\infty}))$ and a subsequence $\{i(j)\}_j$ such that $g_{i(j)}^r \to g_{\infty}^r$ on $B_{R}(x_{\infty})$ for every $r \in \mathbf{Q}_{>0}$.
Let $g_{\infty}(z_{\infty}) := \liminf_{r \to 0}g^r_{\infty}(z_{\infty})$.
The H$\ddot{\text{o}}$lder inequality, Fubini's theorem,  Remark \ref{1234}, Proposition \ref{strong3} and \ref{weak2} yield that
\begin{align*}
\int_{B_R(x_{\infty})}|g^r_{\infty}|^pd\upsilon_{\infty}&=\lim_{j \to \infty}\int_{B_R(x_{i(j)})}|g^r_{i(j)}|^pd\upsilon_{i(j)}\\
&\le\liminf_{j \to \infty}\int_{X_{i(j)}}\int_{X_{i(j)}} \frac{1_{B_r(z)}(x)}{\upsilon_{i(j)}(B_r(z))}|g_{i(j)}|^p(x)d\upsilon_{i(j)}(x)d\upsilon_{i(j)}(z)\\
&=\liminf_{j \to \infty}\int_{X_{i(j)}}\int_{X_{i(j)}} \frac{1_{B_r(z)}(x)}{\upsilon_{i(j)}(B_r(z))}|g_{i(j)}|^p(x)d\upsilon_{i(j)}(z)d\upsilon_{i(j)}(x)\\
&=\liminf_{j \to \infty}\int_{X_{i(j)}}|g_{i(j)}|^p(x)\int_{X_{i(j)}} \frac{1_{B_r(z)}(x)}{\upsilon_{i(j)}(B_r(z))}d\upsilon_{i(j)}(z)d\upsilon_{i(j)}(x)\\
&\le \liminf_{j \to \infty}\int_{X_{i(j)}}|g_{i(j)}|^p(x)\int_{X_{i(j)}} \frac{2^{2\kappa}1_{B_r(x)}(z)}{\upsilon_{i(j)}(B_{r}(x))}d\upsilon_{i(j)}(z)d\upsilon_{i(j)}(x)\\
&\le 2^{2\kappa}\liminf_{j \to \infty}\int_{X_{i(j)}}|g_{i(j)}|^pd\upsilon_{i(j)},
\end{align*}
holds for every $r \in \mathbf{Q}_{>0}$ with $r<1$, where $\kappa=\kappa(1)$.
Therefore Fatou's lemma yields $g_{\infty} \in L^p(B_R(x_{\infty}))$.
Since it is easy to check that a sequence $\{g_{i(j)}\}_{j \le \infty}$ satisfies the assumption of Proposition \ref{weak2}, the assertion follows directly from Proposition \ref{weak2}. 
\end{proof}
\begin{corollary}\label{1}
Let $g_i \in L^p(B_R(x_i))$ for every $i<\infty$ and $g_{\infty} \in L^1_{\mathrm{loc}}(B_R(x_{\infty}))$.
Assume that $\sup_{i < \infty} ||g_i||_{L^p}< \infty$ and that for a.e. $z_{\infty} \in B_R(x_{\infty})$, we see that
\[\liminf_{t \to \infty}\left(\limsup_{i \to \infty} \left| \frac{1}{\upsilon_i (B_t(z_i))}\int_{B_t(z_i)}g_id\upsilon_i-\frac{1}{\upsilon_{\infty} (B_t(z_{\infty}))}\int_{B_t(z_{\infty})}g_{\infty}d\upsilon_{\infty}\right| \right)=0\]
holds for every $z_i \to z_{\infty}$.
Then $g_{\infty} \in L^p(B_R(x_{\infty}))$.
\end{corollary}
\begin{proof}
Proposition \ref{weak com} yields that there exist $\hat{g}_{\infty} \in L^p(B_R(x_{\infty}))$ and a subsequence $\{g_{i(j)}\}_j$ of $\{g_i\}_i$ such that $g_{i(j)}$ converges weakly to $\hat{g}_{\infty}$ on $B_R(x_{\infty})$.
The assumption and Lebesgue's differentiation theorem yield that $\hat{g}_{\infty}(z_{\infty})=g_{\infty}(z_{\infty})$ holds for a.e. $z_{\infty} \in B_R(x_{\infty})$.
Therefore we have the assertion.
\end{proof}
\begin{definition}\label{app}
Assume $p<\infty$.
Let $f_{i,j} \in L^{\infty}(B_R(x_i))$ for every $i \le \infty$ and every $j<\infty$.
We say that \textit{$\{f_{i,j}\}_{i,j}$ is an $L^p$-approximate sequence of $f_{\infty}$} if the following three conditions hold:
\begin{enumerate}
\item $\sup_{i\le \infty}||f_{i,j}||_{L^{\infty}}<\infty$ for every $j$.
\item $f_{i, j}$ converges strongly to $f_{\infty, j}$ at a.e. $z_{\infty} \in B_R(x_{\infty})$ as $i \to \infty$ for every $j$.
\item $||f_{\infty}-f_{\infty, j}||_{L^p} \to 0$ as $j \to \infty$.
\end{enumerate}
\end{definition}
\begin{proposition}\label{876}
Assume $p<\infty$.
Then for every $g_{\infty} \in L^p(B_R(x_{\infty}))$ there exists an $L^p$-approximate sequence of $g_{\infty}$.
\end{proposition}
\begin{proof}
Since $L^{\infty}(B_R(x_{\infty}))$ is dense in $L^p(B_R(x_{\infty}))$, without loss of generality we can assume $g_{\infty} \in L^{\infty}$.
Lebesgue's differentiation theorem yields that  
there exists $K_{\infty} \subset B_R(x_{\infty})$ such that $\upsilon_{\infty}(B_R(x_{\infty})\setminus K_{\infty})=0$ and that 
\[\lim_{r \to 0}\frac{1}{\upsilon_{\infty}(B_r(z_{\infty}))}\int_{B_r(z_{\infty})}\left|g_{\infty}-\frac{1}{\upsilon_{\infty}(B_r(z_{\infty}))}\int_{B_r(z_{\infty})}g_{\infty}d\upsilon_{\infty}\right|^pd\upsilon_{\infty}=0\]
holds for every $z_{\infty} \in K_{\infty}$.
Fix $j$.
A standard covering argument yields that there exists a pairwise disjoint collection $\{\overline{B}_{r_i}(z_i)\}_i$ such that $z_i \in K_{\infty}$, $\overline{B}_{5r_i}(z_i) \subset B_R(x_{\infty})$,
\[\frac{1}{\upsilon_{\infty}(B_{r_i}(z_{i}))}\int_{B_{r_i}(z_{i})}\left|g_{\infty}-\frac{1}{\upsilon_{\infty}(B_{r_i}(z_{i}))}\int_{B_{r_i}(z_{i})}g_{\infty}d\upsilon_{\infty}\right|^pd\upsilon_{\infty}<j^{-1}\]
and that $K_{\infty} \setminus \bigcup_{i=1}^N\overline{B}_{r_i}(z_i) \subset \bigcup_{i=N+1}^{\infty}\overline{B}_{5r_i}(z_i)$  holds for every $N$.
Fix $N_0$ with $\sum_{i=N_0+1}^{\infty}\upsilon_{\infty}(B_{5r_i}(z_i))<j^{-1}$.
Let 
\[g_{i,j}:=\sum_{k=1}^{N_0}\frac{1_{B_{r_i}(z_{i, j})}}{\upsilon_{\infty}(B_{r_i}(z_{i}))}\int_{B_{r_i}(z_{i})}g_{\infty}d\upsilon_{\infty}\]
where $z_{i,j} \to z_i$.
Then by the H$\ddot{\text{o}}$lder inequality, it is easy to check that $\{g_{i, j}\}_{i,j}$ is an $L^p$-approximate sequence of $g_{\infty}$.
\end{proof}
\begin{remark}\label{7890}
By the proof of Proposition \ref{876} and using suitable cutoff functions, it is easy to check that there exists an $L^p$-approximate sequence $\{g_{i, j}\}_{i, j}$ of $g_{\infty}$ such that $g_{i, j} \in C^0(B_R(x_i))$ holds for every $i, j$ and that $\{g_{i, j}\}_i$ is asymptotically uniformly equicontinuous on $B_R(x_{\infty})$ for every $j$.
\end{remark}
The following is a direct consequence of Proposition \ref{strong3}.
It means that roughly speaking, the $L^p$-approximate sequence is `unique':
\begin{proposition}\label{55559}
Assume $p<\infty$.
Let $g_{\infty} \in L^p(B_R(x_{\infty}))$ and $\{g_{i,j}\}_{i,j}, \{\hat{g}_{i,j}\}_{i, j}$ be $L^p$-approximate sequences of $g_{\infty}$.
Then 
\[\lim_{j \to \infty}\left(\limsup_{i \to \infty}||g_{i,j}-\hat{g}_{i,j}||_{L^p(B_R(x_i))}\right)=0.\]
\end{proposition}
We are now in a position to give the definition of \textbf{(S)} as in Section $1$ for functions:
\begin{definition}\label{strongdeff}
We say that \textit{$f_i$ $L^p$-converges strongly to $f_{\infty}$ on $B_R(x_{\infty})$} if 
\[\lim_{j \to \infty}\left(\limsup_{i \to \infty}||f_i-f_{i,j}||_{L^p(B_R(x_i))}\right)=0\]
holds for every (or some) $L^p$-approximate sequence $\{f_{i,j}\}_{i,j}$ of $f_{\infty}$.
\end{definition}
\begin{corollary}\label{61746174}
Let $g_i \in L^p(B_R(x_i))$ and $a_i, b_i \in L^{\infty}(B_R(x_i))$ for every $i \le \infty$ with $\sup_{i \le \infty}(||g_i||_{L^p}+||a_i||_{L^{\infty}}+||b_i||_{L^{\infty}})<\infty$.
Assume that $p<\infty$, $f_i, g_i$ $L^p$-converge strongly to $f_{\infty}, g_{\infty}$ on $B_R(x_{\infty})$, respectively and that
$a_i, b_i$ converge strongly to $a_{\infty}, b_{\infty}$ at a.e. $z_{\infty} \in B_R(x_{\infty})$, respectively.
Then $a_if_i + b_ig_i$ $L^p$-converges strongly to $a_{\infty}f_{\infty} + b_{\infty}g_{\infty}$ on $B_R(x_{\infty})$. 
\end{corollary}
\begin{proof}
Let $\{f_{i, j}\}_{i, j}, \{g_{i, j}\}_{i, j}$ be $L^p$-approximate sequences of $f_{\infty}, g_{\infty}$, respectively.
Then Proposition \ref{strong3} yields that $\{a_if_{i, j}+b_ig_{i, j}\}_{i, j}$ is a $L^p$-approximate sequence of $a_{\infty}f_{\infty} + b_{\infty}g_{\infty}$.
Therefore the assertion follows from Proposition \ref{55559}.
\end{proof}
\begin{proposition}\label{weak3}
Let $g_i \in L^q(B_R(x_i))$ for every $i \le \infty$ with $\sup_{i \le \infty}||g_i||_{L^q}<\infty$.
Assume that $p < \infty$, $f_i$ converges weakly to $f_{\infty}$ on $B_R(x_{\infty})$ and that $g_i$ $L^q$-converges strongly to $g_{\infty}$ on $B_R(x_{\infty})$.
Then
\[\lim_{i \to \infty}\int_{B_R(x_i)}f_ig_id\upsilon_i=\int_{B_R(x_{\infty})}f_{\infty}g_{\infty}d\upsilon_{\infty}.\]
\end{proposition}
\begin{proof}
Let $\{g_{i,j}\}_{i,j}$ be an $L^q$-approximate sequence of $g_{\infty}$.
Propositions \ref{linear} and \ref{weak2} yield that
\[\lim_{i \to \infty}\int_{B_R(x_i)}f_ig_{i, j}d\upsilon_i=\int_{B_R(x_{\infty})}f_{\infty}g_{\infty, j}d\upsilon_{\infty}\]
holds for every $j$.
On the other hand, the H$\ddot{\text{o}}$lder inequality yields
\[\left| \int_{B_R(x_i)}f_ig_{i, j}d\upsilon_i-\int_{B_R(x_i)}f_ig_{i}d\upsilon_i \right| \le ||f_i||_{L^p}||g_i-g_{i,j}||_{L^q}.\]
Therefore by letting $i \to \infty$ and $j \to \infty$, we have the assertion.
\end{proof}
The following is a direct consequence of Proposition \ref{weak3} and the triangle inequality.
Compare with Remark \ref{easy}:
\begin{corollary}\label{rere}
Assume that $p<\infty$ and that $f_i$ $L^p$-converges strongly to $f_{\infty}$ on $B_R(x_{\infty})$.
Then we have the following:
\begin{enumerate}
\item $f_i$ converges weakly to $f_{\infty}$ on $B_R(x_{\infty})$.
\item $|f_i|$ $L^p$-converges strongly to $|f_{\infty}|$ on $B_R(x_{\infty})$.
\end{enumerate}
\end{corollary}
Next we give a lower semicontinuity of $L^p$-norms with respect to the weak convergence:
\begin{proposition}\label{lower}
If $f_i$ converges weakly to $f_{\infty}$ on $B_R(x_{\infty})$, then 
$\liminf_{i \to \infty}||f_i||_{L^p(B_R(x_{i}))}\ge ||f_{\infty}||_{L^p(B_R(x_{\infty}))}.$
\end{proposition}
\begin{proof}
If $p=\infty$, then the assertion follows directly from Lebesgue's differentiation theorem.
Assume $p<\infty$.
Since $(L^p)^*=L^q$, there exists $g_{\infty} \in L^q$ such that $||g_{\infty}||_{L^q}\le 1$ and 
\[\int_{B_R(x_{\infty})}f_{\infty}g_{\infty}d\upsilon_{\infty} =||f_{\infty}||_{L^p}.\]
Let $\{g_{i,j}\}_{i,j}$ be an $L^q$-approximate sequence of $g_{\infty}$.
Proposition \ref{weak3} yields
\[\lim_{i \to \infty}\int_{B_R(x_i)}f_ig_{i,j}d\upsilon_i=\int_{B_R(x_{\infty})}f_{\infty}g_{\infty, j}d\upsilon_{\infty}.\]
On the other hand, since the H$\ddot{\text{o}}$lder inequality yields
\[\lim_{i \to \infty}\left|\int_{B_R(x_i)}f_ig_{i,j}d\upsilon_i\right| \le \liminf_{i \to \infty}\left(||f_i||_{L^p}||g_{i,j}||_{L^q}\right) = \left(\liminf_{i \to \infty}||f_i||_{L^p}\right)||g_{\infty, j}||_{L^q},\]
by letting $j \to \infty$, we have the assertion.
\end{proof}
On the other hand, Propositions \ref{strong3} and \ref{weak2} yield:
\begin{proposition}\label{low}
Assume $p<\infty$.
If $f_i$ $L^p$-converges strongly to $f_{\infty}$ on $B_R(x_{\infty})$, then
$\lim_{i \to \infty}||f_i||_{L^p(B_R(x_i))}= ||f_{\infty}||_{L^p(B_R(x_{\infty}))}$.
\end{proposition}
Conversely, we have the following.
\begin{proposition}\label{222}
Assume $p<\infty$.
Then  $f_i$ $L^p$-converges strongly to $f_{\infty}$ on $B_R(x_{\infty})$ if and only if the following two conditions hold: 
\begin{enumerate}
\item $\limsup_{i \to \infty}||f_i||_{L^p(B_R(x_i))}\le ||f_{\infty}||_{L^p(B_R(x_{\infty}))}$.
\item $f_i$ converges weakly to $f_{\infty}$ on $B_R(x_{\infty})$.
\end{enumerate}
\end{proposition}
\begin{proof}
It suffices to check `if' part.
First assume $p<2$.
Let $\{f_{i, j}\}_{i, j}$ be an $L^p$-approximate sequence of $f_{\infty}$.
Then Clarkson's inequality \cite[Theorem $2$]{clark} for $p <2$ yields
\[2\left(||f_{i, j}||_{L^p}^p + ||f_i||_{L^p}^p\right)^{q-1} \ge ||f_{i, j}+f_i||_{L^p}^q+||f_{i, j}-f_i||_{L^p}^q.\]
Since $f_i+f_{i, j}$ converges weakly to $f_{\infty}+f_{\infty, j}$ on $B_R(m_{\infty})$ as $i \to \infty$, by letting $i \to \infty$ and $j \to \infty$, Proposition \ref{lower} yields
\[2\left(||f_{\infty}||_{L^p}^p + ||f_{\infty}||_{L^p}^p\right)^{q-1} \ge 2^q||f_{\infty}||_{L^p}^q + \limsup_{j \to \infty}\left( \limsup_{i \to \infty}||f_{i, j}-f_i||_{L^p}^q\right).\]
Thus we have $\lim_{j \to \infty}\left( \limsup_{i \to \infty}||f_{i, j}-f_i||_{L^p}^q\right)=0$, i.e., $f_i$ $L^p$-converges strongly to $f_{\infty}$ on $B_R(x_{\infty})$.
Similarly the assertion of the case $p \ge 2$ follows from Clarkson's inequality \cite[Theorem $2$]{clark} for $p \ge 2$.
\end{proof}
The following result is a compatibility result for the case of $L^{\infty}$-functions:
\begin{proposition}\label{compatibility}
Let $g_i \in L^{\infty}(B_R(x_i))$ for every $i \le \infty$ with $\sup_{i \le \infty}||g_i||_{L^{\infty}}<\infty$.
Then $g_i$ converges strongly to $g_{\infty}$ at a.e. $z_{\infty} \in B_R(x_{\infty})$ if and only if $g_i$ $L^{\hat{p}}$-converges strongly to $g_{\infty}$ on $B_R(x_{\infty})$ for some (or every) $1<\hat{p}<\infty$.
\end{proposition}
\begin{proof}
It suffices to check ``if'' part.
Assume that $g_i$ $L^{\hat{p}}$-converges strongly to $g_{\infty}$ on $B_R(x_{\infty})$ for some $1<\hat{p}<\infty$.
Since Proposition \ref{weak2} and Corollary \ref{rere} yield that
\begin{align*}
&\lim_{i \to \infty}\frac{1}{\upsilon_i(B_r(z_{i}))}\int_{B_r(z_{i})}\left|g_{i}-\frac{1}{\upsilon_i (B_r(z_{i}))}\int_{B_r(z_{i})}g_{i}d\upsilon_i\right|d\upsilon_i \\
&=\frac{1}{\upsilon_{\infty}(B_r(z_{\infty}))}\int_{B_r(z_{\infty})}\left|g_{\infty}-\frac{1}{\upsilon_{\infty} (B_r(z_{\infty}))}\int_{B_r(z_{\infty})}g_{\infty}d\upsilon_{\infty} \right|d\upsilon_{\infty}\\
\end{align*}
holds for every $r>0$, where $z_i \to z_{\infty}$,  we see that $g_i$ converges strongly to $g_{\infty}$ at every $z_{\infty} \in B_R(x_{\infty})$ satisfying 
\[\lim_{r \to \infty}\frac{1}{\upsilon_{\infty}(B_r(z_{\infty}))}\int_{B_r(z_{\infty})}\left|g_{\infty}-\frac{1}{\upsilon_{\infty} (B_r(z_{\infty}))}\int_{B_r(z_{\infty})}g_{\infty}d\upsilon_{\infty} \right|d\upsilon_{\infty}=0.\]
\end{proof}
On the other hand, we recall the definition of $L^p$-convergence of functions with respect to the Gromov-Hausdorff topology given by Kuwae-Shioya in \cite{KS, KS2}.
Note that without loss of generality we can assume that every $\epsilon_i$-isometry $\psi_i: B_{R_i}(x_i) \to B_{R_i}(x_{\infty})$ is a Borel map.
\begin{definition}\cite[Definition $3.14$]{KS2}
Assume $p<\infty$.
\begin{enumerate}
\item We say that \textit{$f_i$ $L^p$-converges to $f_{\infty}$ on $B_R(x_{\infty})$ in the sense of Kuwae-Shioya} if there exists $\{\phi_j\}_{j<\infty} \subset C^0(B_R(x_{\infty}))$ such that 
\[\lim_{j \to \infty}(\limsup_{i \to \infty}||f_i - \phi_j \circ \psi_i||_{L^p(B_R(x_i))})=\lim_{j \to \infty}||f_{\infty}-\phi_j||_{L^p(B_R(x_{\infty}))}=0.\]
\item We say that \textit{$f_i$ $L^p$-converges weakly to $f_{\infty}$ on $B_R(x_{\infty})$ in the sense of Kuwae-Shioya} if 
\[\lim_{i \to \infty}\int_{B_R(x_i)}f_ig_id\upsilon_i=\int_{B_R(x_{\infty})}f_{\infty}g_{\infty}d\upsilon_{\infty}\]
holds for every $L^q$-convergent sequence $g_i \to g_{\infty}$ on $B_R(x_{\infty})$ in the sense of Kuwae-Shioya.
\end{enumerate}
\end{definition}
We end this subsection by giving an equivalence between Kuwae-Shioya's formulation and our setting:
\begin{corollary}\label{KSequiv}
We have the following:
\begin{enumerate}
\item $f_i$ $L^p$-converges strongly to $f_{\infty}$ on $B_R(x_{\infty})$ if and only if $f_i$ $L^p$-converges to $f_{\infty}$ on $B_R(x_{\infty})$ in the sense of Kuwae-Shioya.
\item $f_i$ converges weakly to $f_{\infty}$ on $B_R(x_{\infty})$ if and only if $f_i$ $L^p$-converges weakly to $f_{\infty}$ on $B_R(x_{\infty})$ in the sense of Kuwae-Shioya.
\end{enumerate}
\end{corollary}
\begin{proof}
It is a direct consequence of Remark \ref{7890} and Proposition \ref{weak3}.
\end{proof}
\begin{remark}
In \cite{KS, KS2} Kuwae-Shioya proved a compactness result which corresponds to Proposition \ref{weak com} for the case $p=2$.
For instance they showed that $\{L^p(B_R(x_i))\}_i$ satisfies an \textit{asymptotic relation} (see \cite[Definition $3.1$, Theorem $3.27$]{KS2} for the precise definition and statement), and gave a compactness result for a sequence of $\mathrm{CAT}(0)$ spaces having an asymptotic relation.
However in general, $L^p(B_R(x_i))$ is NOT a $\mathrm{CAT}(0)$-space, more precisely, $L^p$-space is $\mathrm{CAT}(0)$ if and only if $p=2$.
Therefore we can not apply directly \cite[Lemma $5.5$]{KS2} to our setting for $p \neq 2$.
\end{remark}
\begin{remark}\label{normal1}
Assume $p<\infty$.
Then it is easy to check that if $(X_i, x_i, \upsilon_i)\equiv (X, x, \upsilon)$ and $\psi_i=id_X$ hold for every $i\le \infty$, then $f_i$ $L^p$-converges strongly to $f_{\infty}$ with respect to the convergence $(X, x, \upsilon) \stackrel{(id_X, \epsilon_i, R_i)}{\to} (X, x, \upsilon)$ if and only if $||f_i-f_{\infty}||_{L^p(B_R(x))} \to 0$.
\end{remark}
\begin{remark}
By Proposition \ref{strong3} and the same way as in Definition \ref{strongdeff},  we can give the definition of $L^1$-strong convergence of functions with respect to the Gromov-Hausdorff topology and also show that it is equivalent to that in the sense of Kuwae-Shioya given in \cite{KS2}.
\end{remark}
\subsubsection{Poincar\'e inequality and $L^p$-strong compactness.}
In this subsection we will prove a compactness result about $L^p$-strong convergence.
The following proposition is an essential tool to get it:
\begin{proposition}\label{mosco}
Let $(X, \upsilon)$ be a proper geodesic metric measure space, $x \in X$ with $\upsilon (B_1(x))=1$ and $1<p<\infty$.
Assume that the following two conditions hold:
\begin{enumerate}
\item (The doubling property on $(X, \upsilon)$ for $\kappa=\kappa (r)$). For every $r>0$ there exists $\kappa=\kappa(r) \ge 0$ such that $\upsilon(B_{2t}(y))\le 2^{\kappa}\upsilon (B_t(y))$ holds for every $y \in X$ and every $t<r$.
\item (The weak Poincar\'e inequality of type $(1, p)$ on $(X, \upsilon)$ for $\tau=\tau(r)$). For every $r>0$ there exists $\tau=\tau(r)>0$ such that for every $y \in X$, every $t\le r$ and every $f \in H_{1, p}(B_t(x))$, we have
\[\frac{1}{\upsilon (B_t(y))}\int_{B_t(y)}\left|f-\frac{1}{\upsilon (B_t(y))}\int_{B_t(y)}fd\upsilon \right| d\upsilon \le \tau t\left(\frac{1}{\upsilon (B_t(y))}\int_{B_t(y)}|g_f|^pd\upsilon\right)^{1/p},\]
where $g_f$ is the minimal generalized upper gradient for $f$ (see \cite[Theorem $2.10$]{ch1} for the precise definition of $g_f$). 
\end{enumerate}
Then for every $T>0$, every $R>0$, every $L \ge 1$ and every $f \in H_{1, p}(B_R(x))$ with $||f||_{H_{1, p}(B_R(x))}\le T$, there exists an $L$-Lipschitz function $f_L$ on $B_{R/5}(x)$ such that $||f_L||_{H_{1, p}(B_{R/5}(x))} \le C(\kappa (R), \tau(R), p, T, R)$ and $||f-f_L||_{L^p(B_{R/5}(x))}\le C(\kappa (R), \tau(R), p, T, R)L^{-\alpha}$, where $\alpha:=\alpha (\kappa(R), \tau(R), p, R)>0$. 
\end{proposition}
\begin{proof}
Define $f \equiv 0$ on $X \setminus B_R(x)$.
Let 
\[K_L:=\left\{ z \in B_R(x); \frac{1}{\upsilon (B_r(z))}\int_{B_r(z)}|df|^pd\upsilon \le L^p\, \mathrm{for\, every\,}r\le R\right\}.\]
By the proof of \cite[Lemma $3.1$]{Ho}, we have 
$\upsilon (B_R(x) \setminus K_L)\le  C(\kappa (R), T)L^{-p}$.
On the other hand, the Poincar\'e inequality yields that 
\[\frac{1}{\upsilon (B_r(w))}\int_{B_r(w)}\left| f- \frac{1}{\upsilon (B_r(w))}\int_{B_r(w)}fd\upsilon \right|d\upsilon \le \tau(R) rL\]
holds for every $w \in K_L$ and every $r\le R$ with $B_r(w) \subset B_R(x)$. 
An argument similar to the proof of \cite[Theorem $4.14$]{ch1} yields that there exists $\hat{K}_L\subset K_L \cap B_{R/5}(x)$ such that 
$\upsilon (K_L \cap B_{R/5}(x) \setminus \hat{K}_L)=0$ and that $f|_{\hat{K}_L}$ is $C(\kappa(R), \tau(R))L$-Lipschitz.
MacShane's lemma (c.f. $(8.2)$ or $(8.3)$ in \cite{ch1}) yields that there exists a $C(\kappa (R), \tau(R))L$-Lipschitz function $f_L$ on $X$ such that $f_L|_{\hat{K}_L}\equiv f|_{\hat{K}_L}$.
Then \cite[Corollary $2.25$]{ch1} yields
\begin{align*}
\int_{B_{R/5}(x)}|df_L|^pd\upsilon &=\int_{B_{R/5}(x)\setminus \hat{K}_L}|df_L|^pd\upsilon+\int_{\hat{K}_L}|df_L|^pd\upsilon \\
&\le \int_{B_{R/5}(x) \setminus \hat{K}_L}C(\kappa (R), \tau(R), p)L^pd\upsilon + \int_{\hat{K}_L}|df|^pd\upsilon \\
&\le C(\kappa (R), \tau(R), p)\upsilon (B_{R}(x)\setminus K_L)L^p+T^p \le C\left(\kappa (R), \tau(R), p, T\right).
\end{align*}
Fix $y_0 \in \hat{K}_{L} \cap B_{R/5}(x)$ with 
\[f_L(y_0)=f(y_0)=\lim_{r \to 0}\frac{1}{\upsilon (B_r(y_0))}\int_{B_r(y_0)}fd\upsilon.\]
Then for every $y \in B_R(x)$ we have 
$|f_L(y)|\le |f_L(y_0)|+C(\kappa(R), \tau(R))L.$
On the other hand, the Poincar\'e inequality and a `telescope argument' (see the proof of \cite[Theorem $4.14$]{ch1}) yield
\[\left|f_L(y_0)-\frac{1}{\upsilon (B_{R/5}(y_0))}\int_{B_{R/5}(y_0)}fd\upsilon \right| \le C(\kappa(R), \tau(R), R)L.\]
Therefore we have $||f_L||_{L^{\infty}(B_R(x))}\le C(\kappa(R), \tau(R), T, R)L$.
In particular we have 
\begin{align*}
\int_{B_{R/5}(x)}|f_L|d\upsilon&=\int_{B_{R/5}(x)\setminus \hat{K}_L}|f_L|d\upsilon + \int_{\hat{K}_L}|f|d\upsilon \\
&\le \upsilon (B_{R/5}(x) \setminus \hat{K}_L)C(\kappa(R), \tau(R), T, R)L +C(\kappa(R), p, T, R) \\
&\le C(\kappa(R), \tau(R), p, T, R).
\end{align*}
On the other hand, the Poincar\'e-Sobolev inequality \cite[Theorem $1$]{HK} yields that there exists $\hat{p}:=\hat{p}(\kappa(R), \tau(R), p)>p$ such that 
\begin{align*}
&\left(\frac{1}{\upsilon (B_{R/5}(x))}\int_{B_{R/5}(x)}\left|f_L-\frac{1}{\upsilon (B_{R/5}(x))}\int_{B_{R/5}(x)}f_Ld\upsilon \right|^{\hat{p}}d\upsilon\right)^{1/\hat{p}}\\
&\le C(\kappa(R), \tau(R), R, p)\left(\frac{1}{\upsilon (B_{R/5}(x))}\int_{B_{R/5}(x)}|df_L|^pd\upsilon \right)^{1/p} \le C\left(\kappa(R), \tau(R), p, T, R\right).
\end{align*}
In particular we have $||f_L||_{L^{\hat{p}}(B_{R/5}(x))} \le C(\kappa(R), \tau(R), p, T, R)$.
Similarly we have $||f||_{L^{\hat{p}}(B_R(x))}\le C(\kappa(R), \tau(R), p, T, R)$.
Therefore the H$\ddot{\text{o}}$lder inequality yields 
$||f-f_L||_{L^p(B_{R/5}(x))}\le (\upsilon (B_{R/5}(x) \setminus \hat{K}_L))^{1/\beta}||f-f_L||_{L^{\hat{p}}(B_{R/5}(x))} \le C(\kappa(R), \tau(R), p, T, R)L^{-p/\beta}$, 
where $\beta$ is the conjugate exponent of $\hat{p}/p >1$.
Therefore we have the assertion.
\end{proof}
We are now in a position to give a compactness result about $L^p$-strong convergence.
Compare with \cite[Theorem $4.5$]{KS2}:
\begin{proposition}\label{strong com}
Let $1<p<\infty$.
Assume that there exists $\{\tau=\tau(r)\}_{r>0} \subset \mathbf{R}$ such that the weak Poincar\'e inequality of type $(1, p)$ for $\tau$ holds on $(X_i, \upsilon_i)$ for every $i < \infty$.
Then for every sequence $\{f_i\}_{i<\infty}$ of $f_i \in H_{1, p}(B_R(x_i))$ with $\sup_{i<\infty}||f_i||_{H_{1, p}}<\infty$,
there exist $f_{\infty} \in L^{p}(B_R(x_{\infty}))$ and a subsequence $\{f_{i(j)}\}_j$ of $\{f_i\}_i$ such that $f_{i(j)}$ $L^p$-converges strongly to $f_{\infty}$ on $B_R(x_{\infty})$.
\end{proposition}
\begin{proof}
Let $T:=\sup_{i<\infty}||f_i||_{H_{1, p}}$.
Proposition \ref{weak com} yields that there exist $f_{\infty} \in L^{p}(B_R(x_{\infty}))$ and a subsequence $\{f_{i(j)}\}_j$ of $\{f_i\}_i$ such that $f_{i(j)}$ converges weakly to $f_{\infty}$ on $B_R(x_{\infty})$.
\begin{claim}\label{jjjj}
Let $r>0$ and $z_{\infty} \in B_R(x_{\infty})$ with $B_{5r}(z_{\infty}) \subset B_R(x_{\infty})$.
Then $f_{i(j)}$ $L^p$-converges strongly to $f_{\infty}$ on $B_r(z_{\infty})$. 
\end{claim}
The proof is as follows.
By Proposition \ref{mosco}, for every $L \ge 1$ and every $j<\infty$, there exists an $L$-Lipschitz function $(f_{i(j)})_{L}$ on $B_{r}(z_{i(j)})$ such that $||f_{i(j)}-(f_{i(j)})_{L}||_{L^p(B_{r}(z_{i(j)}))}\le  C(\kappa(R), \tau(R), p, T, R)L^{-\alpha}$, where $\alpha=\alpha (\kappa(R), \tau(R), p)>0$ and $z_{i(j)} \to z_{\infty}$.
By Proposition \ref{conti com}, without loss of generality we can assume that there exists an $L$-Lipschitz function $(f_{\infty})_{L}$ on $B_r(z_{\infty})$ such that 
$(f_{i(j)})_L \to (f_{\infty})_L$ on $B_r(z_{\infty})$.
Then Proposition \ref{lower} yields $||f_{\infty}-(f_{\infty})_L||_{L^p(B_r(z_{\infty}))}\le \liminf_{j \to \infty}||f_{i(j)}-(f_{i(j)})_L||_{L^p(B_r(z_{i(j)}))}
\le C(\kappa(R), \tau(R), p, T, R)L^{-\alpha}$.
Since $L$ is arbitrary, we have Claim \ref{jjjj}.

Let $\hat{p}, \beta$ as in the proof of Proposition \ref{mosco} and let $\{\overline{B}_{r_l}(w_l)\}_l$ be a countable pairwise disjoint collection with $B_{5r_l}(w_l) \subset B_R(x_{\infty})$ and $\upsilon_{\infty}(B_R(x_{\infty}) \setminus \bigcup_l B_{r_l}(w_l))=0$.
Fix $\epsilon >0$. 
Let $N_0$ with $\upsilon_{\infty}(B_R(x_{\infty}) \setminus \bigcup_{l=1}^{N_0} B_{r_l}(w_l))<\epsilon$.
By the proof of Proposition \ref{mosco} we have $||f_{i(j)}||_{L^{\hat{p}}(B_R(x_{i(j)}))}\le C(\kappa(R), \tau(R), p, T, R)$ for every $j<\infty$.
In particular the H$\ddot{\text{o}}$lder inequality yields 
\begin{align*}
&\left| \int_{B_R(x_{i(j)})}|f_{i(j)}|^pd\upsilon_{i(j)}- \sum_{l=1}^{N_0}\int_{B_{r_l}(w_{l, i(j)})}|f_{i(j)}|^pd\upsilon_{i(j)}\right|\\
&\le \left(\upsilon_{i(j)}\left(B_R(x_{i(j)}) \setminus \bigcup_{l=1}^{N_0} B_{r_l}(w_{l, i(j)})\right)\right)^{1/\beta}||f_{i(j)}||_{L^{\hat{p}}(B_R(x_{i(j)}))}^p<\Psi(\epsilon; \kappa(R), \tau(R), p, T, R)
\end{align*}
for every sufficiently large $j<\infty$, where $w_{l, i(j)} \to w_l$.
Thus Claim \ref{jjjj} yields 
\begin{align*}
\int_{B_R(x_{i(j)})}|f_{i(j)}|^pd\upsilon_{i(j)}&=\sum_{l=1}^{N_0}\int_{B_{r_l}(w_{l, i(j)})}|f_{i(j)}|^pd\upsilon_{i(j)} \pm \Psi(\epsilon; \kappa(R), \tau(R), p, T, R) \\
&=\sum_{l=1}^{N_0}\int_{B_{r_l}(w_l)}|f_{\infty}|^pd\upsilon_{\infty} \pm \Psi(\epsilon; \kappa(R), \tau(R), p, T, R)\\
& \le \int_{B_R(x_{\infty})}|f_{\infty}|^pd\upsilon_{\infty} + \Psi(\epsilon; \kappa(R), \tau(R), p, T, R)
\end{align*}
for every sufficiently large $j<\infty$.
Since $\epsilon$ is arbitrary, we have $\limsup_{j \to \infty}||f_{i(j)}||_{L^p(B_R(x_{i(j)}))} \le ||f_{\infty}||_{L^p(B_R(x_{\infty}))}$.
Thus Proposition \ref{222} yields that $f_{i(j)}$ $L^p$-converges strongly to $f_{\infty}$ on $B_R(x_{\infty})$.
\end{proof}
\subsection{Tensor fields.}
Throughout this subsection we will always consider the following setting:
\begin{enumerate}
\item $R>0$, $n \in \mathbf{N}, 1<p \le \infty$ and $r,s \in \mathbf{Z}_{\ge 0}$.
\item $\{(M_i, m_i)\}_{i <\infty}$ is a sequence of pointed $n$-dimensional complete Riemannian manifolds with $\mathrm{Ric}_{M_i} \ge -(n-1)$.
\item $(M_{\infty}, m_{\infty}, \upsilon)$ is the Ricci limit space of $\{(M_i, m_i, \underline{\mathrm{vol}})\}_{i <\infty}$ with $M_{\infty} \neq \{m_{\infty}\}$.
\end{enumerate}
Let $k:=\mathrm{dim} M_{\infty}$.
\subsubsection{$L^1_{w-\mathrm{loc}}, L^1_{\mathrm{loc}}$-tensor fields.}
\begin{definition}
We say that \textit{$T_{\infty} \in \Gamma_{\mathrm{Bor}}(T^r_sB_R(m_{\infty}))$ is a weakly locally $L^1$-tensor field} if $\langle T_{\infty},  \bigotimes_{j=1}^r\nabla r_{x_{j}} \otimes \bigotimes_{j=r+1}^{r+s}dr_{x_{j}}\rangle \in L^1_{\mathrm{loc}}(B_R(m_{\infty}))$ for every $\{x_j\}_{j} \subset M_{\infty}$.
We also say that \textit{$T_{\infty} \in \Gamma_{\mathrm{Bor}}(T^r_sB_R(m_{\infty}))$ is a locally $L^1$-tensor field} if $|T_{\infty}| \in L^1_{\mathrm{loc}}(B_R(m_{\infty}))$.
\end{definition}
Let us denote by $L^1_{w-\mathrm{loc}}(T^r_sB_R(m_{\infty}))$ the set of weakly locally $L^1$-tensor fields of type $(r, s)$ on $B_R(m_{\infty})$ and by $L^1_{\mathrm{loc}}(T^r_sB_R(m_{\infty}))$ the set of locally $L^1$-tensor fields of type $(r, s)$ on $B_R(m_{\infty})$.
We end this subsection by giving the definition of the pointwise convergence of $\textbf{(W)}$ for $L^1_{w-\mathrm{loc}}$-tensor fields as in Section $1$:
\begin{definition}\label{tensor weak conv}
Let $T_i \in L^1_{w-\mathrm{loc}}(T^r_sB_R(m_i))$ for every $i \le \infty$. 
We say that \textit{$T_i$ converges weakly to $T_{\infty}$ at $z_{\infty} \in B_R(m_{\infty})$} if
$\langle T_i, \bigotimes_{j=1}^r\nabla r_{x_{j, i}} \otimes \bigotimes_{j=r+1}^{r+s}dr_{x_{j, i}}\rangle$ converges weakly to 
$\langle T_{\infty}, \bigotimes_{j=1}^r\nabla r_{x_{j}} \otimes \bigotimes_{j=r+1}^{r+s}dr_{x_{j}}\rangle$ at $z_{\infty}$ for every $x_{l, i} \to x_l$ as $i \to \infty$.
\end{definition}
\subsubsection{$L^{\infty}$-tensor fields.}
Let $T_i \in L^{\infty}(T^r_sB_R(m_i))$ for every $i \le \infty$ with $\sup_{i \le \infty}||T_i||_{L^{\infty}}<\infty$.
\begin{definition}\label{inftyinfty}
We say that \textit{$T_i$ converges strongly to $T_{\infty}$ at $z_{\infty} \in B_R(m_{\infty})$} if the following two conditions hold:
\begin{enumerate}
\item $T_i$ converges weakly to $T_{\infty}$ at $z_{\infty}$.
\item $\{|T_i|^2\}_i$ is weakly upper semicontinuous at $z_{\infty}$. 
\end{enumerate}
\end{definition}
Compare with Definition \ref{8}, Proposition \ref{strong}, Remark \ref{77} and \cite[Definition $4.4$]{Ho}.
\begin{proposition}\label{dist}
We see that $\bigotimes_{i=1}^r\nabla r_{x_{i, j}} \otimes \bigotimes_{i=r+1}^{r+s}dr_{x_{i, j}}$ converges strongly to 
$\bigotimes_{i=1}^r\nabla r_{x_{i, \infty}} \otimes \bigotimes_{i=r+1}^{r+s}dr_{x_{i, \infty}}$ at every $z_{\infty} \in M_{\infty}$ as $j \to \infty$ for every $x_{i, j} \to x_{i, \infty}$.
\end{proposition}
\begin{proof}
For every $\hat{x}_{i, j} \to \hat{x}_{i, \infty}$, since 
\[\left\langle \bigotimes_{i=1}^r\nabla r_{x_{i, j}} \otimes \bigotimes_{i=r+1}^{r+s}dr_{x_{i, j}}, \bigotimes_{i=1}^r\nabla r_{\hat{x}_{i, j}} \otimes \bigotimes_{i=r+1}^{r+s}dr_{\hat{x}_{i, j}}\right\rangle (w) = \prod_{i=1}^{r+s} \cos \angle x_{i, j}w\hat{x}_{i, j}\]
holds for a.e. $w \in M_{j}$, 
the assertion follows from Propositions \ref{strong3}, \ref{angle2} and \ref{weak2}.
\end{proof}
The following is a main property of the strong convergence:
\begin{proposition}\label{strong2}
Let $\hat{r}, \hat{s} \in \mathbf{Z}_{\ge 0}$, $S_i \in L^{\infty}(T^{\hat{r}}_{\hat{s}}B_R(m_i))$ for every $i \le \infty$ with $\sup_{i \le \infty}||S_i||_{L^{\infty}}<\infty$ and $A \subset B_R(m_{\infty})$.
Assume that $S_i, T_i$ converge strongly to $S_{\infty}, T_{\infty}$ at a.e. $z_{\infty} \in A$, respectively.
Then we have the following:
\begin{enumerate}
\item If $(r, s)=(\hat{r}, \hat{s})$, then $\langle S_i, T_i \rangle$ converges strongly to $\langle S_{\infty}, T_{\infty} \rangle$ at a.e. $z_{\infty} \in A$. 
\item $S_i \otimes T_i (\in L^{\infty}(T^{r+\hat{r}}_{s+\hat{s}}B_R(m_i)))$  converges strongly  to $S_{\infty} \otimes T_{\infty}$ at a.e. $z_{\infty} \in A$.
\item If $(r, s)=(\hat{r}, \hat{s})$, then $S_i + T_i$ converges strongly  to $S_{\infty} + T_{\infty}$  at a.e. $z_{\infty} \in A$.
\item If $\hat{r} \le r$ and $\hat{s} \le s$, then $T_i(S_i) (\in L^{\infty}(T^{r-\hat{r}}_{s-\hat{s}}B_R(m_i)))$  converges strongly to $T_{\infty}(S_{\infty})$ at a.e. $z_{\infty} \in A$.
\end{enumerate}
\end{proposition}
\begin{proof}
Let $\Lambda:=\mathrm{Map}(\{1, \ldots, r+s\} \to \{1, \ldots, k\})$ and $\hat{L}:=\sup_i||S_i||_{L^{\infty}}$.
By an argument similar to the proof of \cite[Lemma $3. 15$]{Ho}, for a.e. $z_{\infty} \in B_R(m_{\infty})$ there exist $\{x_i\}_{1 \le i \le k} \subset M_{\infty}$ and $\{C(a)\}_{a \in \Lambda} \subset \mathbf{R}$ with $|C(a)|\le C(\hat{L}, n)$ such that
\[\lim_{r \to 0}\frac{1}{\upsilon (B_r(z_{\infty}))}\int_{B_r(z_{\infty})}\left|S_{\infty}-\sum_{a \in \Lambda} C(a)\bigotimes_{j=1}^r \nabla r_{x_{a(j)}} \otimes \bigotimes_{j=r+1}^{r+s}dr_{x_{a(j)}}\right|^2d\upsilon=0.\]
Thus $(1)$ follows from Proposition \ref{dist} and an argument similar to the proof of \cite[Theorem $4.4$]{Ho}.
On the other hand, $(2), (3), (4)$ follow directly from $(1)$, Propositions \ref{strong3} and \ref{strong}.
\end{proof}
\begin{corollary}
If $T_i$ converges strongly to $T_{\infty}$ at a.e. $z_{\infty} \in B_R(m_{\infty})$, then $T_i$ converges strongly to $T_{\infty}$ at every $z_{\infty} \in B_R(m_{\infty})$.
\end{corollary}
\begin{proof}
It follows from Propositions \ref{strong}, \ref{weak2}, Corollary \ref{098} and Proposition \ref{strong2}.
\end{proof}
We now recall a main result of \cite{Ho}:
\begin{proposition}\cite[Corollary $4.5$]{Ho}\label{harm5}
Let $f_i$ be a Lipschitz function on $B_R(m_i)$ for every $i \le \infty$ with $\sup_{i \le \infty}\mathbf{Lip}f_i<\infty$.
Assume that $f_i \in C^2(B_R(m_i))$ holds for every $i < \infty$ with $\sup_{i<\infty}||\Delta f_i||_{L^2(B_R(m_i))}<\infty$, and that $f_i \to f_{\infty}$ on $B_R(m_{\infty})$.
Then $df_i \to df_{\infty}$ on $B_R(m_{\infty})$ (which means that $df_i$ converges strongly to $df_{\infty}$ on $B_R(m_{\infty})$).
\end{proposition}
In Section $4$ we will prove that the assumption of uniform $L^2$-bounds as in Proposition \ref{harm5} can be replaced by \textit{uniform $L^1$-bounds}.
See Corollary \ref{l1}.
We will also easily check that this assumption of uniform $L^1$-bounds is sharp in some sense.
See Remarks \ref{qqqqq} and \ref{ppppp}. 
\subsubsection{$L^p$-tensor fields.}
Let $1<p \le \infty$, $T_i \in L^p(T^r_sB_R(m_i))$ for every $i \le \infty$ with $L:=\sup_{i \le \infty}||T_i||_{L^p}<\infty$ and let $q$ is the conjugate exponent of $p$.
\begin{proposition}\label{weak4}
Let $\hat{r}, \hat{s} \in \mathbf{Z}_{\ge 0}, S_i \in L^{\infty}(T^{\hat{r}}_{\hat{s}}B_R(m_i))$ for every $i \le \infty$ with $\sup_{i \le \infty}||S_i||_{L^{\infty}}<\infty$ and $A \subset B_R(m_{\infty})$.
Assume that $S_i$ converges strongly to $S_{\infty}$ at a.e. $z_{\infty} \in A$ and that $T_i$ converges weakly to $T_{\infty}$ at a.e. $z_{\infty} \in A$.
Then we have the following:
\begin{enumerate}
\item If $(r, s)=(\hat{r}, \hat{s})$, then $\langle S_i, T_i \rangle$ converges weakly to $\langle S_{\infty}, T_{\infty} \rangle$ at a.e. $z_{\infty} \in A$.
\item $S_i \otimes T_i$ converges weakly to $S_{\infty} \otimes T_{\infty}$ at a.e. $z_{\infty} \in A$.
\item If $\hat{r} \le r$ and $\hat{s} \le s$, then $T_i(S_i)$ converges weakly to $T_{\infty}(S_{\infty})$ at a.e. $z_{\infty} \in A$.
\item If $r \le \hat{r}$ and $s \le \hat{s}$, then $S_i(T_i)$ converges weakly to $T_{\infty}(S_{\infty})$ at a.e. $z_{\infty} \in A$.
\end{enumerate}
\end{proposition}
\begin{proof}
We first check $(1)$.
With the same notation as in the proof of Proposition \ref{strong2}, for a.e. $z_{\infty} \in B_R(m_{\infty})$ and every $\epsilon>0$
there exists $r>0$ such that
\[\frac{1}{\upsilon (B_t(z_{\infty}))}\int_{B_t(z_{\infty})}\left|S_{\infty}-\sum_{a \in \Lambda} C(a)\bigotimes_{j=1}^r \nabla r_{x_{a(j)}} \otimes \bigotimes_{j=r+1}^{r+s}dr_{x_{a(j)}}\right|^2d\upsilon<\epsilon\]
and
\begin{align*}
&\limsup_{i \to \infty}\Biggl| \frac{1}{\mathrm{vol}\,B_t(z_{l})}\int_{B_t(z_{l})}\left\langle T_i, \bigotimes_{j=1}^r \nabla r_{x_{a(j), i}} \otimes \bigotimes_{j=r+1}^{r+s}dr_{x_{a(j), i}} \right\rangle d\mathrm{vol} \\
&- \frac{1}{\upsilon (B_t(z_{\infty}))}\int_{B_t(z_{\infty})}\left\langle T_{\infty}, \bigotimes_{j=1}^r \nabla r_{x_{a(j)}} \otimes \bigotimes_{j=r+1}^{r+s}dr_{x_{a(j)}} \right\rangle d\upsilon \Biggl| <\epsilon
\end{align*}
hold for every $t<r$ and every $a \in \Lambda$.
By an argument similar to the proof of \cite[Theorem $4.4$]{Ho}, we see that 
\[\frac{1}{\mathrm{vol}\,B_t(z_{l})}\int_{B_t(z_{l})}\left|S_{l}-\sum_{a \in \Lambda} C(a)\bigotimes_{j=1}^r \nabla r_{x_{a(j), l}} \otimes \bigotimes_{j=r+1}^{r+s}dr_{x_{a(j), l}}\right|^2d\mathrm{vol}<\Psi(\epsilon ;n, \hat{L})\]
holds for every sufficiently large $l$, where $x_{a(j), l} \to x_{a(j)}$ and $z_l \to z_{\infty}$.
Let
\[K_{t, l}:=\left\{w \in B_t(z_l); \left|S_{l}-\sum_{a \in \Lambda} C(a)\bigotimes_{j=1}^r \nabla r_{x_{a(j), l}} \otimes \bigotimes_{j=r+1}^{r+s}dr_{x_{a(j), l}}\right|(w) \le \Psi(\epsilon ;n, \hat{L}) \right\}.\]
Then we have $\mathrm{vol}\,K_{t, l}/\mathrm{vol}\,B_t(z_l) \ge 1-\Psi(\epsilon; n, \hat{L}, R)$.
In particular we see that 
\[\frac{1}{\mathrm{vol}\,B_t(z_{l})}\int_{B_t(z_{l})}\left|S_{l}-\sum_{a \in \Lambda} C(a)\bigotimes_{j=1}^r \nabla r_{x_{a(j), l}} \otimes \bigotimes_{j=r+1}^{r+s}dr_{x_{a(j), l}}\right|^qd\mathrm{vol}<\Psi(\epsilon ;n, \hat{L}, p)\]
holds for every sufficiently large $l\le \infty$.
Therefore the H$\ddot{\text{o}}$lder inequality and Proposition \ref{dist} yield that
\begin{align*}
&\frac{1}{\mathrm{vol}\,B_t(z_{l})}\int_{B_t(z_{l})}\langle S_l, T_l \rangle d\mathrm{vol}\\
&=\frac{1}{\mathrm{vol}\,B_t(z_{l})}\int_{B_t(z_{l})}\left\langle \sum_{a \in \Lambda} C(a)\bigotimes_{j=1}^r \nabla r_{x_{a(j), l}} \otimes \bigotimes_{j=r+1}^{r+s}dr_{x_{a(j), l}}, T_l \right\rangle d\mathrm{vol} \pm \Psi(\epsilon; n, p, L, \hat{L})\\
&=\frac{1}{\upsilon (B_t(z_{\infty}))}\int_{B_t(z_{\infty})}\left\langle \sum_{a \in \Lambda} C(a)\bigotimes_{j=1}^r \nabla r_{x_{a(j)}} \otimes \bigotimes_{j=r+1}^{r+s}dr_{x_{a(j)}}, T_{\infty} \right\rangle d\mathrm{\upsilon} \pm \Psi(\epsilon; n, p, L, \hat{L})\\
&=\frac{1}{\upsilon (B_t(z_{\infty}))}\int_{B_t(z_{\infty})}\langle S_{\infty}, T_{\infty} \rangle d\upsilon \pm \Psi(\epsilon; n, p, L, \hat{L})\\
\end{align*}
holds for every sufficiently large $l$.
Therefore we have $(1)$.
On the other hand, $(2)$ follows directly from Proposition \ref{linear} and $(1)$.
If $\hat{r} \le r$ and $\hat{s} \le s$, then since 
\[\left\langle T_j(S_j), \bigotimes_{i=1}^{r-\hat{r}}\nabla r_{x_{i, j}} \otimes \bigotimes_{i=r-\hat{r}+1}^{r-\hat{r}+s-\hat{s}}dr_{x_{i, j}} \right\rangle=\left\langle T_j, S_j \otimes \left(\bigotimes_{i=1}^{r-\hat{r}}\nabla r_{x_{i, j}} \otimes \bigotimes_{i=r-\hat{r}+1}^{r-\hat{r}+s-\hat{s}}dr_{x_{i, j}}\right)\right\rangle,\]
$(3)$ follows directly from $(2)$.
Similarly we have $(4)$.
\end{proof}
The following is a direct consequence of Corollary \ref{098}, Propositions \ref{dist} and \ref{weak4}:
\begin{corollary}\label{weakkk}
If $T_i$ converges weakly to $T_{\infty}$ at a.e. $z_{\infty} \in B_R(m_{\infty})$, then $T_i$ converges weakly to $T_{\infty}$ at every $z_{\infty} \in B_R(m_{\infty})$.
\end{corollary}
We now give a compactness result for the weak convergence of $L^p$-tensor fields similar to Proposition \ref{weak com}:
\begin{proposition}\label{tensor com}
Let $S_i \in L^p(T^r_sB_R(m_i))$ for every $i<\infty$ with $\sup_{i < \infty}||S_i||_{L^p} < \infty$.
Then there exist $S_{\infty} \in L^p(T^r_sB_R(m_{\infty}))$ and a subsequence $\{S_{i(j)}\}_j$ of $\{S_i\}_i$ such that 
$S_{i(j)}$ converges weakly to $S_{\infty}$ on $B_R(m_{\infty})$.
\end{proposition}
\begin{proof}
We only give a proof of the case $p<\infty$ because the proof of the case $p=\infty$ is similar.
Let $A:=\{x_i\}_{i \in \mathbf{N}}$ be a dense subset of $M_{\infty}$ and $\Lambda:=\mathrm{Map}(\{1, \ldots, r+s\} \to \mathbf{N})$. 
Proposition \ref{weak com} yields that there exist a subsequence $\{i(j)\}_j$ and $\{F_{a}\}_{a \in \Lambda} \subset L^p(B_R(x_{\infty}))$ such that
$\langle S_{i(j)}, \bigotimes_{l=1}^r\nabla r_{x_{a(l), i(j)}} \otimes \bigotimes_{l=r+1}^{r+s}dr_{x_{a(l), i(j)}} \rangle$ converges weakly to $F_{a}$ on $B_R(x_{\infty})$ for every $x_{k, i(j)} \to x_{k}$ as $j \to \infty$.
By the assumption and \cite[Lemma $3.1$]{Ho}, there exists $K_{\infty} \subset B_R(m_{\infty})$ such that $\upsilon (B_R(m_{\infty}) \setminus K_{\infty})=0$, $K_{\infty} \subset M_{\infty} \setminus \bigcup_i C_{x_i}$, 
\[F_{a}(w_{\infty})=\lim_{r \to 0}\frac{1}{\upsilon (B_r(w_{\infty}))}\int_{B_r(w_{\infty})}F_{a}d\upsilon\]
holds for every $a \in \Lambda$ and every $w_{\infty} \in K_{\infty}$, 
and that for every $z_{\infty} \in K_{\infty}$, 
\[\limsup_{r \to 0}\left( \limsup_{j \to \infty}\frac{1}{\mathrm{vol}\,B_r(z_{i(j)})}\int_{B_r(z_{i(j)})}|S_{i(j)}|^pd\mathrm{vol}\right)<\infty\]
holds for some $z_{i(j)} \to z_{\infty}$ (see also the proof of Proposition \ref{contr}).
\begin{claim}\label{honhon}
Let $z_{\infty} \in K_{\infty}$ and $\{C_i(a)\}_{a \in \Lambda, i=1,2} \subset \mathbf{R}$. 
Assume that there exists a finite subset $\hat{\Lambda} \subset \Lambda$ such that $C_i(a) = 0$ holds for every $a \in \Lambda \setminus \hat{\Lambda}$ and every $i=1, 2$, and that 
\begin{align*}
&\sum_{a, b \in \Lambda} (C_1(a)-C_2(b))\prod_{l=1}^{r+s} \cos \angle x_{a(l)}z_{\infty}x_{b(l)}=0
\end{align*}
holds. Then $\sum_{a}C_1(a)F_{a}(z_{\infty})=\sum_{a}C_2(a)F_{a}(z_{\infty})$.
\end{claim}
The proof is as follows.
By the proof of \cite[Theorem $4.3$]{Ho3}, we see that
\begin{align*}
&\left|\frac{1}{\mathrm{vol}\,B_r(z_{i(j)})}\int_{B_{r}(z_{i(j)})}\left\langle S_{i(j)}, X_{i(j)}^1\right\rangle d\mathrm{vol}-\frac{1}{\mathrm{vol}\,B_r(z_{i(j)})}\int_{B_{r}(z_{i(j)})}\left\langle S_{i(j)}, X_{i(j)}^2\right\rangle d\mathrm{vol}  \right|\\
&\le \left( \frac{1}{\mathrm{vol}\,B_r(z_{i(j)})}\int_{B_r(z_{i(j)})}|S_{i(j)}|^pd\mathrm{vol}\right)^{1/p}\left( \frac{1}{\mathrm{vol}\,B_r(z_{i(j)})}\int_{B_r(z_{i(j)})}|X_{i(j)}^1-X^2_{i(j)}|^qd\mathrm{vol}\right)^{1/q} \to 0
\end{align*}
holds as $j \to \infty$ and $r \to 0$, where $X_{i(j)}^m=\sum_{a \in \Lambda} C_m(a)\bigotimes_{l=1}^{r}\nabla r_{x_{a(l), i(j)}}\otimes \bigotimes_{l=r+1}^{r+s}dr_{x_{a(l), i(j)}} \in L^{\infty}(T^r_sM_{i(j)})$.
Therefore we have Claim \ref{honhon}.

Theorem \ref{dist3}, \cite[Remark $4.3$]{Ho} and Claim \ref{honhon} yield that there exists a unique $S_{\infty} \in \Gamma_{\mathrm{Bor}}(T^r_sB_R(m_{\infty}))$ such that for  every $a \in \Lambda$ we see that
$\langle S_{\infty},  \bigotimes_{l=1}^r\nabla r_{x_{a(l)}} \otimes \bigotimes_{l=r+1}^{r+s}dr_{x_{a(l)}}\rangle (w_{\infty})=F_{a}(w_{\infty})$ holds for a.e. $w_{\infty} \in B_R(m_{\infty})$.
In particular $S_{\infty} \in L^1_{w-\mathrm{loc}}(T^r_sB_R(m_{\infty}))$.
\begin{claim}\label{lp}
$S_{\infty} \in L^p(T^r_sB_R(m_{\infty}))$.
\end{claim}
The proof is as follows.
By Theorem \ref{dist3}, Remark \ref{radial} and Lebesgue's differentiation theorem, without loss of generality we can assume that there exist sequences $\{K_j\}_{j < \infty}$ of $K_j \subset K_{\infty}$ and $\{I_j\}_{j<\infty}$ of $I_j \subset \mathbf{N}$ such that $\sharp I_j=k  (=\mathrm{dim}\,M_{\infty})$, $\mathrm{Leb}\,K_j=K_j$, $K_{\infty}=\bigcup_jK_j$
 and that 
for every $\epsilon>0$ and every $z_{\infty} \in K_{\infty}$, there exists $j:=j(z_{\infty}, \epsilon)$ such that $z_{\infty} \in K_j$ and that $\{\bigotimes_{l=1}^{r}\nabla r_{x_{a(l)}}\otimes \bigotimes_{l=r+1}^{r+s}dr_{x_{a(l)}}(w_{\infty})\}_{a \in \Lambda_j}$ is an $\epsilon$-orthogonal basis on $(T^r_{s}M_{\infty})_{w_{\infty}}$ for every $w_{\infty} \in K_j$, where $\Lambda_j:=\mathrm{Map}(\{1, \ldots, r+s\} \to I_j)$.
Thus Proposition \ref{ep} yields that 
\begin{align*}
|S_{\infty}|^2(w_{\infty})&\le (1+\Psi(\epsilon; r, s, n))\sum_{a \in \Lambda_j}\langle S_{\infty}, \bigotimes_{l=1}^{r}\nabla r_{x_{a(l)}}\otimes \bigotimes_{l=r+1}^{r+s}dr_{x_{a(l)}}\rangle^2 (w_{\infty})\\
&=(1+\Psi(\epsilon; r, s, n))\sum_{a \in \Lambda_j}(F_{a}(w_{\infty}))^2
\end{align*}
holds for every $w_{\infty} \in K_{j}$, in particular, 
$|S_{\infty}|^p(w_{\infty})\le C(n, p, r, s)\sum_{a \in \Lambda_j}|F_{a}(w_{\infty})|^p$.

Fix a sufficiently small $\epsilon>0$.
A standard covering argument yields that there exists a countable pairwise disjoint collection $\{\overline{B}_{r_i}(w_i)\}_i$ such that 
$w_i \in K_{j(i)}$ where $j(i):=j(w_i, \epsilon)$, $\overline{B}_{5r_i}(w_i) \subset B_R(m_{\infty})$, $\upsilon (B_{r_i}(w_i) \cap K_{j(i)})/\upsilon (B_{r_i}(w_i)) \ge 1 - \epsilon$
and that $K_{\infty} \setminus \bigcup_{i=1}^N \overline{B}_{r_i}(w_i) \subset \bigcup_{i=N+1}^{\infty}\overline{B}_{5r_i}(w_i) \subset B_R(m_{\infty})$ holds for every $N$.
Fix $N_0$ with $\sum_{i=N_0+1}^{\infty}\upsilon (B_{5r_i}(w_i))<\epsilon$.
Let $K^{\epsilon}_{\infty}:=\bigcup_{i=1}^{N_0}(\overline{B}_{r_i}(w_i) \cap K_{j(i)})$. 
Note 
$\upsilon (B_R(m_{\infty}) \setminus K_{\infty}^{\epsilon})\le (1 + \Psi(\epsilon; n))\upsilon \left(K_{\infty} \setminus \bigcup_{i=1}^{N_0}\overline{B}_{r_i}(w_i)\right)<\Psi(\epsilon ; n)$.
Then Proposition \ref{lower} yields that
\begin{align*}
&\int_{K_{\infty}^{\epsilon}}|S_{\infty}|^pd\upsilon \\
&= \sum_{l=1}^{N_0}\int_{B_{r_l}(w_l) \cap K_{j(l)}}|S_{\infty}|^pd\upsilon\\
&\le C(n, p, r, s)\sum_{l=1}^{N_0}\sum_{a \in \Lambda_{j(l)}}\int_{B_{r_l}(w_l) \cap K_{j(l)}}|F_{a}|^pd\upsilon + \Psi(\epsilon;n, R)\\
&\le C(n, p, r, s)\sum_{l=1}^{N_0} \sum_{a \in \Lambda_{j(l)}}\int_{B_{r_l}(w_{l, i(m)})}\left|\left\langle S_{i(m)}, \bigotimes_{t=1}^r\nabla r_{x_{a(t), i(m)}} \otimes \bigotimes_{t=r+1}^{r+s}dr_{x_{a(t), i(m)}} \right\rangle \right|^pd\underline{\mathrm{vol}}+\Psi(\epsilon;n, R) \\
&\le C(n, p, r, s)\sum_{l=1}^{N_0}\int_{B_{r_l}(w_{l, i(m)})}|S_{i(m)}|^pd\underline{\mathrm{vol}}+\Psi(\epsilon;n, R) \\
&\le C(n, p, r, s)\int_{B_R(m_{i(m)})}|S_{i(m)}|^pd\underline{\mathrm{vol}}+\Psi(\epsilon;n, R)
\end{align*}
holds for every sufficiently large $m$, where $w_{l, i(m)} \to w_{l}$ as $m \to \infty$.
Thus by the dominated convergence theorem, we have Claim \ref{lp}.

Thus we have the assertion.
\end{proof}
The next corollary follows from Proposition \ref{tensor com} and an argument similar to the proof of Corollary \ref{1}:
\begin{corollary}
Let $S_i \in L^p(T^r_sB_R(m_i))$ for every $i < \infty$ with $\sup_{i<\infty}||S_i||_{L^p}<\infty$,  and $S_{\infty} \in L^1_{w-\mathrm{loc}}(T^r_sB_R(m_{\infty}))$.
Assume that $S_i$ converges weakly to 
$S_{\infty}$ at a.e. $z_{\infty} \in B_R(m_{\infty})$.
Then $S_{\infty} \in L^p(T^r_sB_R(m_{\infty}))$.
\end{corollary}
\begin{remark}\label{000}
Similarly, we can prove the following:
Let $\{(C_{j}, \phi_{j})\}_{j}$ be a rectifiable coordinate system of $(M_{\infty}, \upsilon)$ associated with $\{(M_i, m_i)\}_i$, $S_i \in L^p(T^r_sB_R(m_i))$ for every $i < \infty$ with $\sup_{i<\infty}||S_i||_{L^p}<\infty$, and $S_{\infty} \in \Gamma_{\mathrm{Bor}}(T^r_sB_R(m_{\infty}))$.
Assume that the following hold:
\begin{enumerate}
\item $\langle S_{\infty}, \bigotimes _{l=1}^r\nabla \phi_{j, a(l), \infty} \otimes \bigotimes _{l=r+1}^{r+s}d\phi_{j, a(l), \infty}\rangle \in L^1_{\mathrm{loc}}$ holds for every $j$ and every $a \in \mathrm{Map}(\{1, \ldots, r+s\} \to \{1, \ldots, k\})$.
\item $\langle S_i, \bigotimes _{l=1}^r\nabla \phi_{j, a(l), i} \otimes \bigotimes _{l=r+1}^{r+s}d\phi_{j, a(l), i}\rangle$ converges weakly to $\langle S_{\infty}, \bigotimes _{l=1}^r\nabla \phi_{j, a(l), \infty} \otimes \bigotimes _{l=r+1}^{r+s}d\phi_{j, a(l), \infty}\rangle$ at a.e. $z_{\infty} \in C_{j}$ for every $j$ and every $a \in \mathrm{Map}(\{1, \ldots, r+s\} \to \{1, \ldots, k\})$.
\end{enumerate}
Then $S_{\infty} \in L^p(T^r_sB_R(m_{\infty}))$.
Moreover, in Proposition \ref{compat} we will prove that  $S_i$ converges weakly to $S_{\infty}$ on $B_R(m_{\infty})$.
\end{remark}
\begin{definition}
Assume $p<\infty$.
Let $S_{\infty} \in L^p(T^r_sB_R(m_{\infty}))$ and $S_{i, j} \in L^{\infty}(T^r_sB_R(m_{i}))$ for every  $i\le \infty$ and every $j< \infty$.
We say that \textit{$\{S_{i, j}\}_{i, j}$ is an $L^p$-approximate sequence of $S_{\infty}$} if the following three conditions hold:
\begin{enumerate}
\item $\sup_{i \le \infty}||S_{i,j}||_{L^{\infty}}<\infty$ for every $j$.
\item $S_{i, j}$ converges strongly to $S_{\infty, j}$ on $B_R(m_{\infty})$ as $i \to \infty$ for every $j$.
\item $||S_{\infty}-S_{\infty, j}||_{L^p} \to 0$ as $j \to \infty$.
\end{enumerate}
\end{definition}
\begin{proposition}\label{app8}
Assume $p<\infty$.
For every $S_{\infty} \in L^p(T^r_sB_R(m_{\infty}))$ there exists an $L^p$-approximate sequence of $S_{\infty}$.
\end{proposition}
\begin{proof}
It is easy to check that $L^{\infty}(T^r_sB_R(m_{\infty}))$ is dense in $L^p(T^r_sB_R(m_{\infty}))$.
Thus without loss of generality we can assume that $S_{\infty} \in L^{\infty}$.
Let $\Lambda:=\mathrm{Map}(\{1, \ldots, r+s\} \to \{1, \ldots, k\})$.
By an argument similar to the proof of Proposition \ref{weak4}, there exists $K_{\infty} \subset B_R(m_{\infty})$ such that $\upsilon (B_R(m_{\infty}) \setminus K_{\infty})=0$ and that for every $z_{\infty} \in K_{\infty}$ there exist $\{C(a, z_{\infty})\}_{a \in \Lambda} \subset \mathbf{R}$ and $\{x_{j}(z_{\infty})\}_{1 \le j \le k} \subset M_{\infty}$ such that for every $\epsilon>0$, there exists $r(z_{\infty}, \epsilon)>0$ such that for every $t<r(z_{\infty}, \epsilon)$, we see that
\[\frac{1}{\mathrm{vol}\,B_t(z_{l})}\int_{B_t(z_{l})}\left|S_{l}-\sum_{a \in \Lambda} C(a, z_{\infty})\bigotimes_{j=1}^r \nabla r_{x_{a(j), l}(z_{\infty})} \otimes \bigotimes_{j=r+1}^{r+s}dr_{x_{a(j), l}(z_{\infty})}\right|^pd\mathrm{vol}<\epsilon\]
holds for every sufficiently large $l \le \infty$, where $z_l \to z_{\infty}$ and $x_{j, l}(z_{\infty}) \to x_j(z_{\infty})$ as $l \to \infty$. 
Fix $j \in \mathbf{N}$.
A standard covering argument yields that there exists a finite pairwise disjoint collection $\{\overline{B}_{r_i}(w_i)\}_{1 \le i \le N}$ such that 
$w_i \in K_{\infty}$, $r_i<r(w_i, j^{-1})$, $\overline{B}_{r_i}(w_i) \subset B_R(m_{\infty})$, and $\upsilon \left(K_{\infty} \setminus \bigcup_{i=1}^N\overline{B}_{r_i}(w_i)\right)<j^{-1}$.
Let $S_{l, j} := \sum_{a \in \Lambda, 1 \le i \le N} C(a, w_{i})1_{B_{r_i}(w_{i, l})}\bigotimes_{j=1}^r \nabla r_{x_{a(j), l}(w_{i})} \otimes \bigotimes_{j=r+1}^{r+s}dr_{x_{a(j), l}(w_{i})}$, where $w_{i, l} \to w_i$ as $l \to \infty$.
Then Proposition \ref{dist} yields that $\{S_{l, j}\}_{l, j}$ is an $L^p$-approximate sequence of $S_{\infty}$.
\end{proof}
Proposition \ref{strong2} yields:
\begin{proposition}
Assume $p<\infty$.
Let $S_{\infty} \in L^p(T^r_sB_R(m_i))$ and let $\{S_{i, j}\}_{i, j}, \{\hat{S}_{i, j}\}_{i, j}$ be $L^p$-approximate sequences of $S_{\infty}$.
Then 
\[\lim_{j \to \infty}\left( \limsup_{i \to \infty}||S_{i, j}-\hat{S}_{i, j}||_{L^p}\right)=0.\]
\end{proposition}
We now are in a position to give the definition of \textbf{(S)} for tensor fields as in Section $1$:
\begin{definition}\label{strongdef}
Assume $p<\infty$.
We say that \textit{$T_i$ $L^p$-converges strongly to $T_{\infty}$ on $B_R(m_{\infty})$} if 
\[\lim_{j \to \infty}\left( \limsup_{i \to \infty}||T_i - T_{i, j}||_{L^p}\right)=0\]
for every (or some) $L^p$-approximate sequence $\{T_{i, j}\}_{i, j}$ of $T_{\infty}$. 
\end{definition}
By Proposition \ref{strong2} and an argument similar to the proof of Corollary \ref{61746174} we have the following.
See also Proposition \ref{strong83}.
\begin{corollary}\label{617461746174}
Let $S_i \in L^p(T^r_sB_R(m_i))$ for every $i \le \infty$ with $\sup_{i \le \infty}||S_i||_{L^p}<\infty$.
Assume that $S_i, T_i$ $L^p$-converge strongly to $S_{\infty}, T_{\infty}$ on $B_R(m_{\infty})$, respectively.
Then $S_i + T_i$ $L^p$-converges strongly to $S_{\infty}+T_{\infty}$ on $B_R(m_{\infty})$.
\end{corollary}
\begin{remark}\label{83}
The H$\ddot{\text{o}}$lder inequality yields that if $T_i$ $L^p$ converges strongly to $T_{\infty}$ on $B_R(m_{\infty})$, then $T_i$ $L^{\hat{p}}$ converges strongly to $T_{\infty}$ on $B_R(m_{\infty})$ for every $1<\hat{p}\le p$.
\end{remark}
By Proposition \ref{strong2} and an argument similar to the proof of Proposition \ref{weak3}, we have the following:
\begin{proposition}\label{conv678}
Let $S_i \in L^q(T^r_sB_R(m_i))$ for every $i \le \infty$ with $\sup_{i \le \infty}||S_i||_{L^q}<\infty$.
Assume that $p<\infty$, $T_{i}$ converges weakly to $T_{\infty}$ on $B_R(m_{\infty})$ and that $S_i$ $L^q$-converges strongly to $S_{\infty}$ on $B_R(m_{\infty})$.
Then 
\[\lim_{i \to \infty}\int_{B_R(m_i)}\langle S_i, T_i \rangle d\underline{\mathrm{vol}}=\int_{B_R(m_{\infty})}\langle S_{\infty}, T_{\infty}\rangle d\upsilon.\] 
\end{proposition}
The following is a direct consequence of Proposition \ref{conv678} and triangle inequality:
\begin{proposition}\label{normconv}
Assume that $p<\infty$ and that
$T_i$ $L^p$-converges strongly to $T_{\infty}$ on $B_R(m_{\infty})$.
Then we have the following:
\begin{enumerate}
\item $T_i$ converges weakly to $T_{\infty}$ on $B_R(m_{\infty})$.
\item $|T_i|$ $L^p$-converges strongly to $|T_{\infty}|$ on $B_R(m_{\infty})$.
\end{enumerate}
\end{proposition}
As a corollary of Propositions \ref{low} and \ref{normconv}, we have the following:
\begin{corollary}\label{stnorm}
Assume that $p<\infty$ and
that $T_i$ $L^p$-converges strongly to $T_{\infty}$ on $B_R(m_{\infty})$.
Then $\lim_{i \to \infty}||T_i||_{L^p}=||T_{\infty}||_{L^p}$.
\end{corollary}
We give a lower semicontinuity of $L^p$-norms with respect to the weak convergence:
\begin{proposition}\label{lower semi}
If $T_i$ converges weakly to $T_{\infty}$ on $B_{R}(m_{\infty})$,
then $\liminf_{i \to \infty}||T_i||_{L^p} \ge ||T_{\infty}||_{L^p}$.
\end{proposition}
\begin{proof}
First assume $p<\infty$.
Recall $T_{\infty}^{(p-1)}(x):= |T_{\infty}(x)|^{p-2}T_{\infty}(x)$.
Since $|T^{(p-1)}_{\infty}(x)|=|T_{\infty}(x)|^{p-1}$, 
we see that $T^{(p-1)}_{\infty} \in L^q(T^r_sB_R(m_{\infty}))$ and 
\[\int_{B_R(m_{\infty})}\langle T_{\infty}, T^{(p-1)}_{\infty} \rangle d\upsilon=||T_{\infty}||_{L^p}^p=||T^{(p-1)}_{\infty}||_{L^q}^q\]
hold.
Let $\{\hat{T}_{i, j}\}_{i, j}$ be an $L^q$-approximate sequence of $T^{(p-1)}_{\infty}$.
Then Proposition \ref{weak4} and the H$\ddot{\text{o}}$lder inequality yield that 
\[\lim_{i \to \infty}\int_{B_R(m_i)}\langle T_i, \hat{T}_{i, j}\rangle d\underline{\mathrm{vol}}=\int_{B_R(m_{\infty})}\langle T_{\infty}, \hat{T}_{\infty, j}\rangle d\upsilon,\]
\[\lim_{j \to \infty}\int_{B_R(m_{\infty})}\langle T_{\infty}, \hat{T}_{\infty, j}\rangle d\upsilon=\int_{B_R(m_{\infty})}\langle T_{\infty}, T^{(p-1)}_{\infty} \rangle d\upsilon=||T_{\infty}||_{L^p}^p,\]
\[\int_{B_R(m_i)}\langle T_i, \hat{T}_{i, j}\rangle d\underline{\mathrm{vol}}\le ||T_i||_{L^p}||\hat{T}_{i, j}||_{L^q} \]
hold,
$||\hat{T}_{i, j}||_{L^q} \to ||\hat{T}_{\infty, j}||_{L^q}$ holds as $i \to \infty$, and that
$||\hat{T}_{\infty, j}||_{L^q} \to ||T^{(p-1)}_{\infty}||_{L^q}=||T_{\infty}||_{L^p}^{p/q}$ holds as $j \to \infty$. 
Therefore we have 
\[||T_{\infty}||_{L^p}^{p/q} \liminf_{i \to \infty}||T_i||_{L^p} \ge ||T_{\infty}||_{L^p}^p.\]
Therefore we have the assertion for the case $p < \infty$.

Next assume $p = \infty$.
By considering rescaled metrics $R^{-2}g_{M_i}$, without loss of generality we can assume that $\underline{\mathrm{vol}}\,B_R(m_{i})=1$ holds for every $i \le \infty$. 
Then since $||T_i||_{L^{\infty}} \ge ||T_{i}||_{L^{\hat{p}}}$ holds for every $\hat{p}<\infty$, we see that $\liminf_{i \to \infty}||T_i||_{L^{\infty}} \ge ||T_{\infty}||_{L^{\hat{p}}}$ holds for every $\hat{p}<\infty$.
Letting $\hat{p} \to \infty$ gives the assertion for the case $p = \infty$.
\end{proof}
\begin{remark}\label{normal2}
It is easy to check that for $p<\infty$, if $(M_i, m_i, \underline{\mathrm{vol}}) \equiv (M_{\infty}, m_{\infty}, \upsilon)$ and $\psi_i\equiv  id_{M_{\infty}}$ hold for every $i<\infty$, then $T_i$ $L^p$-converges strongly to $T_{\infty}$ on $B_R(m_{\infty})$ with respect to the convergence $(M_{\infty}, m_{\infty}, \upsilon) \stackrel{(id_{M_{\infty}}, \epsilon_i, R_i)}{\to} (M_{\infty}, m_{\infty}, \upsilon)$ if and only if $||T_i-T_{\infty}||_{L^p(B_R(m_{\infty}))} \to 0$.
Compare with Remark \ref{normal1}.
\end{remark}
\begin{proposition}\label{hon}
Assume $p<\infty$.
Then $T_i$ $L^p$-converges strongly to $T_{\infty}$ on $B_R(m_{\infty})$ if and only if the following two conditions hold:
\begin{enumerate}
\item $\limsup_{i \to \infty}||T_i||_{L^p} \le ||T_{\infty}||_{L^p}$.
\item $T_i$ converges weakly to $T_{\infty}$ on $B_R(m_{\infty})$.
\end{enumerate}
\end{proposition}
\begin{proof}
First we recall Clarkson's inequalities:
\begin{claim}\label{1223}
Let $v, u \in \mathbf{R}^l$.
\begin{enumerate}
\item If $p<2$, then $|u+v|^q+|u-v|^q\le 2(|u|^p+|v|^p)^{q-1}$.
\item If $p \ge 2$, then $|u+v|^p +|u-v|^p \le 2^{p-1}(|u|^p+|v|^p)$.
\end{enumerate}
\end{claim}
The proof is as follows.
It is known that Claim \ref{1223} holds for $l=2$.
See for instance the proof of Clarkson's inequalities \cite[Theorem $2$]{clark}.
Since there exists an isometric embedding linear map from $\mathrm{span}\{u, v\}$ to $\mathbf{R}^2$, we have Claim \ref{1223}.

The next claim follows from Claim \ref{1223} and an argument similar to the proof of Clarkson's inequalities \cite[Theorem $2$]{clark}:
\begin{claim}\label{1225}
Let $T, S \in L^p(T^r_sB_R(m_{\infty}))$.
\begin{enumerate}
\item If $p<2$, then $||T+S||_{L^p}^q+||T-S||_{L^p}^q\le 2(||T||_{L^p}^p+||S||_{L^p}^p)^{q-1}$.
\item If $p \ge 2$, then $||T+S||_{L^p}^q+||T-S||_{L^p}^q\le 2^{p-1}(||T||_{L^p}^p+||S||_{L^p}^p)$.
\end{enumerate}
\end{claim}
Then Proposition \ref{hon} follows from Claim \ref{1225} and an argument similar to the proof of Proposition \ref{222}.
\end{proof}
\begin{proposition}\label{compati7}
Let $S_i \in L^{\infty}(T^r_sB_R(m_i))$ for every $i \le \infty$ with  $\sup_{i \le \infty} ||S_i||_{L^{\infty}}<\infty$.
Then $S_i$ converges strongly to $S_{\infty}$ on $B_R(m_{\infty})$ if and only if $S_i$ $L^{\hat{p}}$-converges strongly to $S_{\infty}$ on $B_R(m_{\infty})$ for some (or every) $1<\hat{p}<\infty$.
\end{proposition}
\begin{proof}
It suffices to check `if' part.
Assume that $S_i$ $L^{\hat{p}}$-converges strongly to $S_{\infty}$ on $B_R(m_{\infty})$ for some $\hat{p}$.
Then Propositions \ref{compatibility} and \ref{normconv} yield that $S_i$ converges weakly to $S_{\infty}$ on $B_R(m_{\infty})$ and that $|S_i|$ converges strongly to $|S_{\infty}|$ at a.e. $x_{\infty} \in B_R(m_{\infty})$.
Thus Propositions \ref{strong3} gives that $|S_i|^2$ converges strongly to $|S_{\infty}|^2$ at a.e. $x_{\infty} \in B_R(m_{\infty})$.
Therefore the assertion follows from Proposition \ref{weak2}.
\end{proof}
We now prove Theorem \ref{metric}:

\textit{A proof of Theorem \ref{metric}.}

This is a direct consequence of Propositions \ref{dist} and \ref{compati7}. $\,\,\,\,\,\,\,\,\,\,\,\,\Box$.
\begin{proposition}\label{strong83}
Let $\hat{r}, \hat{s} \in \mathbf{Z}_{\ge 0}$ and  $S_i \in L^{\infty}(T^{\hat{r}}_{\hat{s}}B_R(m_i))$ for every $i \le \infty$ with $\sup_{i\le \infty}||S_i||_{L^{\infty}}<\infty$.
Assume that $p<\infty$, $T_i$ $L^p$-converges strongly to $T_{\infty}$ on $B_R(m_{\infty})$ and that $S_i$ converges strongly $S_{\infty}$ at every $z_{\infty} \in B_R(m_{\infty})$.
Then we have the following:
\begin{enumerate}
\item If $r=\hat{r}$ and $s=\hat{s}$, then $\langle S_i, T_i \rangle$ $L^p$-converges strongly to $\langle S_{\infty}, T_{\infty} \rangle$ on $B_R(m_{\infty})$.
\item $S_i \otimes T_i$ $L^p$-converges strongly to $S_{\infty} \otimes T_{\infty}$ on $B_R(m_{\infty})$.
\item If $\hat{r}\le r$ and $\hat{s} \le s$, then $T_i(S_i)$ $L^p$-converges strongly to $T_{\infty}(S_{\infty})$ on $B_R(m_{\infty})$.
\item If $r \le \hat{r}$ and $s \le \hat{s}$, then $S_i(T_i)$ $L^p$-converges strongly to $S_{\infty}(T_{\infty})$ on $B_R(m_{\infty})$.
\end{enumerate}
\end{proposition}
\begin{proof}
This is a direct consequence of Proposition \ref{strong2}.
\end{proof}
We end this subsection by giving the following compatibility result which performs a crucial role in the next section.
\begin{proposition}\label{compat}
Let $\{(C_{j}, \phi_{j})\}_{j}$ be a rectifiable coordinate system of $(M_{\infty}, \upsilon)$ associated with $\{(M_i, m_i)\}_i$, and $A \subset B_R(m_{\infty})$.
Assume that  $\langle T_i, \bigotimes _{l=1}^r\nabla \phi_{j, a(l), i} \otimes \bigotimes _{l=r+1}^{r+s}d\phi_{j, a(l), i}\rangle$ converges weakly to $\langle T_{\infty}, \bigotimes _{l=1}^r\nabla \phi_{j, a(l), \infty} \otimes \bigotimes _{l=r+1}^{r+s}d\phi_{j, a(l), \infty}\rangle$ at a.e. $z_{\infty} \in C_{j} \cap A$ for every $j$ and every $a \in \mathrm{Map}(\{1, \ldots, r+s\} \to \{1, \ldots, k\})$.
Then $T_i$ converges weakly to $T_{\infty}$ at a.e. $z_{\infty} \in A$.
\end{proposition}
\begin{proof}
Proposition \ref{strong2} yields that $\bigotimes _{l=1}^r\nabla \phi_{j, a(l), i} \otimes \bigotimes _{l=r+1}^{r+s}d\phi_{j, a(l), i}$ converges strongly to $\bigotimes _{l=1}^r\nabla \phi_{j, a(l), \infty} \otimes \bigotimes _{l=r+1}^{r+s}d\phi_{j, a(l), \infty}$ at a.e. $z_{\infty} \in C_{j} \cap A$ for every $j$ and every $a$.
Let $\hat{K}_{\hat{L}}$ be the set of $z_{\infty} \in B_R(m_{\infty})$ satisfying that there exists  $z_i \to z_{\infty}$ such that 
\[\frac{1}{\mathrm{vol}\,B_t(z_i)}\int_{B_t(z_{i})}|T_i|^pd\mathrm{vol} \le \hat{L}\]
holds for every $t\le 1$ and every $i \le \infty$.
See also the proof of Proposition \ref{contr}.
Note that by the proof of Proposition \ref{tensor com} (or Proposition \ref{contr}), we see that $\upsilon (B_R(m_{\infty}) \setminus \hat{K}_{\hat{L}}) \to 0$ holds as $\hat{L} \to \infty$.
Let $\{x_i\}_{1 \le i \le r+s} \subset M_{\infty}$.
By an argument similar to the proof of Proposition \ref{strong2}, without loss of generality we can assume that for every $j$ and a.e. $z_{\infty} \in C_j \cap A$ there exists $\{C(a, j, z_{\infty})\}_{a} \subset \mathbf{R}$ with $|C(a, j, z_{\infty})| \le C(n)$ such that for every $\epsilon >0$ there exists $r:=r(z_{\infty}, j, \epsilon)>0$ such that 
\[\frac{1}{\upsilon (B_t(z_{\infty}))}\int_{B_t(z_{\infty})}\left| \bigotimes_{l=1}^r \nabla r_{x_l} \otimes \bigotimes_{l=r+1}^{r+s}dr_{x_l}-\sum_{a}C(a, j, z_{\infty})\bigotimes_{l=1}^r \nabla \phi_{j, a(l), \infty} \otimes \bigotimes_{l=r+1}^{r+s}d\phi_{j, a(l), \infty}\right|^qd\upsilon < \epsilon\]
and
\begin{align*}
&\limsup_{i \to \infty}\Biggl| \frac{1}{\mathrm{vol}\,B_t(z_i)}\int_{B_t(z_{i})}\left\langle T_i, \bigotimes_{l=1}^r\nabla \phi_{j, a(l), i} \otimes \bigotimes_{l=r+1}^{r+s}d \phi_{j, a(l), i} \right\rangle d\mathrm{vol} \\
&-\frac{1}{\upsilon (B_t(z_{\infty}))}\int_{B_t(z_{\infty})}\left\langle T_{\infty}, \bigotimes_{l=1}^r\nabla \phi_{j, a(l), \infty} \otimes \bigotimes_{l=r+1}^{r+s}d\phi_{j, a(l), \infty} \right\rangle d\upsilon \Biggl|<\epsilon
\end{align*}
hold for every $t<r$ and every $a$.
Moreover by Proposition \ref{strong2} without loss of generality we can assume that for every $t<r$
\[\frac{1}{\mathrm{vol}\,B_t(z_{i})}\int_{B_t(z_{i})}\left| \bigotimes_{l=1}^r \nabla r_{x_{l, i}} \otimes \bigotimes_{l=r+1}^{r+s}dr_{x_{l, i}}-\sum_{a}C(a, j, z_{\infty})\bigotimes_{l=1}^r \nabla \phi_{j, a(l), i} \otimes \bigotimes_{l=r+1}^{r+s}d\phi_{j, a(l), i}\right|^qd\mathrm{vol} < \epsilon\]
holds for every sufficiently large $i$.
Then for every $j$, a.e. $z_{\infty} \in C_j \cap A \cap \hat{K}_{\hat{L}}$ and every $t<r(z_{\infty}, j, \epsilon)$, the H$\ddot{\text{o}}$lder inequality yields that 
\begin{align*}
&\frac{1}{\mathrm{vol}\,B_t(z_{i})}\int_{B_t(z_{i})}\left\langle T_i, \bigotimes_{l=1}^r \nabla r_{x_{l, i}} \otimes \bigotimes_{l=r+1}^{r+s}dr_{x_{l, i}} \right\rangle d\mathrm{vol}\\
&=\frac{1}{\mathrm{vol}\,B_t(z_{i})}\int_{B_t(z_{i})}\left\langle T_i, \sum_{a}C(a, j, z_{\infty})\bigotimes_{l=1}^r \nabla \phi_{j, a(l), i} \otimes \bigotimes_{l=r+1}^{r+s}d\phi_{j, a(l), i}\right\rangle d\mathrm{vol} \pm \Psi(\epsilon; n, \hat{L})\\
&=\frac{1}{\upsilon (B_t(z_{\infty}))}\int_{B_t(z_{\infty})}\left\langle T_{\infty}, \sum_{a}C(a, j, z_{\infty})\bigotimes_{l=1}^r \nabla \phi_{j, a(l), \infty} \otimes \bigotimes_{l=r+1}^{r+s}d\phi_{j, a(l), \infty}\right\rangle d\upsilon \pm \Psi(\epsilon; n, \hat{L})\\
&=\frac{1}{\upsilon (B_t(z_{\infty}))}\int_{B_t(z_{\infty})}\left\langle T_{\infty},  \bigotimes_{l=1}^r \nabla r_{x_l} \otimes \bigotimes_{l=r+1}^{r+s}dr_{x_l}\right\rangle d\upsilon \pm \Psi(\epsilon; n, \hat{L})
\end{align*}
holds for every sufficiently large $i$, where $x_{l, i} \to x_l$.
Thus $T_i$ converges weakly to $T_{\infty}$ at a.e. $z_{\infty} \in C_j \cap A \cap \hat{K}_{\hat{L}}$.
Therefore we have the assertion.
\end{proof}
\subsubsection{Contraction.}
Let $T_i \in L^p(T^r_sB_R(m_i))$ for every $i \le \infty$ with $L:=\sup_{i \le \infty}||T_i||_{L^p}<\infty$.
Assume $p<\infty$.

In \cite{ch-co1} Cheeger-Colding showed that the following four conditions (called \textit{noncollapsing conditions}) are equivalent:
\begin{enumerate}
\item $\mathcal{R}_n \neq \emptyset$.
\item $\mathcal{R}_i = \emptyset $ for every $i<n$.
\item There exists $\nu>0$ such that $\mathrm{vol}\,B_1(m_i) \ge \nu$ holds for every $i<\infty$.
\item $\mathrm{dim}_HM_{\infty}=\mathrm{dim}\,M_{\infty}=n$, where $\mathrm{dim}_H$ is the Hausdorff dimension.
\end{enumerate}
Note that if a condition above holds, then $(M_i, m_i, \mathrm{vol}) \to (M_{\infty}, m_{\infty}, H^n)$, where $H^n$ is the $n$-dimensional spherical Hausdorff measure.
See \cite[Theorems $5.9$ and $5.11$]{ch-co1} for the details.

The following is an essential tool to prove $(1)$ of Theorem \ref{laplacian}: 
\begin{proposition}\label{contr}
Let $A \subset B_R(m_{\infty})$.
Assume that $(M_{\infty}, m_{\infty})$ is the noncollapsed limit space of $\{(M_i, m_i)\}_i$ and that $T_i$ converges weakly to $T_{\infty}$ at a.e. $z_{\infty} \in A$.
Then $C_b^aT_i$ converges weakly to $C^a_bT_{\infty}$ at a.e. $z_{\infty} \in A$ for every $0 \le a \le b$. 
\end{proposition}
\begin{proof}
We will give a proof of the case for $s=0, a=1, b=r=2$ only because the proof of the other case is similar.
Let $T_i \equiv 0$ on $M_i \setminus B_R(m_i)$.
For every $\hat{L} \ge 1$, let $K_{\hat{L}, i}$ be the set of $z_i \in B_R(m_i)$ such that 
\[\frac{1}{\mathrm{vol}\,B_t(z_i)}\int_{B_t(z_i)}|T_i|^pd\mathrm{vol}\le \hat{L}\]
holds for every $t \le 1$.
Then \cite[Lemma $3.1$]{Ho} yields $\mathrm{vol}\,K_{\hat{L}, i}/\mathrm{vol}\,B_R(m_i) \ge 1-\Psi(\hat{L}^{-1}; n, R, L)$.
Let $\hat{K}_{\hat{L}, i}$ be a compact subset of $K_{\hat{L}, i}$ with $\mathrm{vol}\,\hat{K}_{\hat{L}, i}/\mathrm{vol}\,B_R(m_i) \ge 1-\Psi(\hat{L}^{-1}; n, R, L)$.
Without loss of generality we can assume that there exists $\lim_{i \to \infty}\hat{K}_{\hat{L}, i} \subset \overline{B}_R(m_{\infty})$.
Note that by \cite[Proposition $2.3$]{Ho}, we have $\upsilon (\lim_{i \to \infty}\hat{K}_{\hat{L}, i})/\upsilon (B_R(m_{\infty})) \ge 1-\Psi (\hat{L}^{-1};n, R, L)$.

On the other hand, by Theorem \ref{dist3}, there exist a sequence $\{K_j\}_{j \in \mathbf{N}}$ of $K_j \subset M_{\infty}$ and $\{x_i^j\}_{1 \le i \le n, j \in \mathbf{N}} \subset M_{\infty}$ with $\mathrm{Leb}\,K_j=K_j$ and $K_j \subset M_{\infty} \setminus \bigcup_{i=1}^nC_{x_i^j}$ such that $\upsilon (M_{\infty} \setminus \bigcup_jK_j)=0$ and that for every $z_{\infty} \in \bigcup_jK_j$ and every $\epsilon >0$, there exists $j:=j(z_{\infty}, \epsilon)$ such that $z_{\infty} \in K_j$ and that $\{dr_{x_i^j}(w_{\infty})\}_i$ is an $\epsilon$-orthogonal basis on $T_{w_{\infty}}^*M_{\infty}$ for every $w_{\infty} \in K_j$. 
Let $K_{\hat{L}}:=\lim_{i \to \infty}\hat{K}_{\hat{L}, i}\cap \hat{K}_{\hat{L}, \infty} \cap \bigcup_{j}K_j$.

Fix $\hat{L} \ge 1, z_{\infty} \in K_{\hat{L}} \cap A$ and $\epsilon >0$.
Let $j:=j(z_{\infty}, \epsilon)$, $\{x_i:=x_i^j\}_i$ as above.
Then there exists $r>0$ such that for every $t<r$ we see that
\[\frac{1}{\upsilon (B_t(z_{\infty}))}\int_{B_t(z_{\infty})}\left|\langle dr_{x_i}, dr_{x_m}\rangle - \delta_{im} \right|d\upsilon< \Psi(\epsilon; n)\]
holds for every $i, m$, and that
\begin{align*}
&\left| \frac{1}{\upsilon (B_t(z_{\infty}))}\int_{B_t(z_{\infty})}\sum_i T_{\infty}(dr_{x_i}, dr_{x_i})d\upsilon -\frac{1}{\mathrm{vol}\,B_t(z_{l})}\int_{B_t(z_{l})}\sum_i T_{l}(dr_{x_{i, l}}, dr_{x_{i, l}})d\mathrm{vol} \right|\\
&\le \epsilon
\end{align*}
holds for every sufficiently large $l$, where $z_l (\in K_{\hat{L}, l}) \to z_{\infty}$ and $x_{i, l} \to x_{i}$.
Fix $t>0$ with $t<r$.
By Proposition \ref{dist} we see that
\[\frac{1}{\mathrm{vol}\,B_t(z_{l})}\int_{B_t(z_{l})}\left|\langle dr_{x_{i, l}}, dr_{x_{m, l}}\rangle - \delta_{im} \right|d\mathrm{vol} < \Psi(\epsilon; n)\]
holds for every sufficiently large $l \le \infty$.
Let $A_{t, l}:=\{y \in B_t(z_l); |\langle dr_{x_{i, l}}, dr_{x_{j, l}}\rangle(y) - \delta_{ij}|<\Psi(\epsilon; n)$ holds for every $i, j.\}$.
Then we see that $\mathrm{vol}\,A_{t, l}/\mathrm{vol}\,B_t(z_l) \ge 1-\Psi(\epsilon; n)$ holds for every sufficiently large $l \le \infty$.
The H$\ddot{\text{o}}$lder inequality yields that 
\begin{align*}
&\left| \frac{1}{\mathrm{vol}\,B_t(z_{l})}\int_{B_t(z_{l})}\sum_i T_{l}(dr_{x_{i, l}}, dr_{x_{i, l}})d\mathrm{vol} -\frac{1}{\mathrm{vol}\,B_t(z_{l})}\int_{A_{t,l}}\sum_i T_{l}(dr_{x_{i, l}}, dr_{x_{i, l}})d\mathrm{vol} \right|\\
&\le \Psi(\epsilon; \hat{L}, n, p)
\end{align*}
and 
\begin{align*}
\left| \frac{1}{\mathrm{vol}\,B_t(z_{l})}\int_{B_t(z_{l})}\mathrm{Tr}\,T_{l}d\mathrm{vol} -\frac{1}{\mathrm{vol}\,B_t(z_{l})}\int_{A_{t,l}}\mathrm{Tr}\,T_{l}d\mathrm{vol} \right| \le \Psi(\epsilon; \hat{L}, n, p)
\end{align*}
hold for every sufficiently large $l \le \infty$.
On the other hand, $(1)$ of Proposition \ref{ep} yields
\begin{align*}
\left| \frac{1}{\mathrm{vol}\,B_t(z_{l})}\int_{A_{t,l}}\sum_i T_{l}(dr_{x_{i, l}}, dr_{x_{i, l}})d\mathrm{vol} -\frac{1}{\mathrm{vol}\,B_t(z_{l})}\int_{A_{t,l}}\mathrm{Tr}\,T_{l}d\mathrm{vol} \right| \le \Psi(\epsilon; \hat{L}, n, p).
\end{align*}
Since $\hat{L}$ and $\epsilon$ are arbitrary, we have the assertion.
\end{proof}
\begin{remark}
Note that Proposition \ref{contr} does NOT hold for collapsing case.
For instance if $M_{\infty}$ is collapsed, then for every $z_{\infty} \in M_{\infty}$,  $\mathrm{Tr}\,g_{M_i} (\equiv n)$ dose not $L^p$-converge weakly to $\mathrm{Tr}\,g_{M_{\infty}} (\equiv k)$ at $z_{\infty}$.
\end{remark}
On the other hand, for $L^p$-strong convergence we have the following without the noncollapsing assumption:
\begin{proposition}\label{contr2}
Assume that $T_i$ $L^p$-converges strongly to $T_{\infty}$ on $B_R(m_{\infty})$.
Then $C^a_bT_i$ $L^p$-converges strongly to $C_b^aT_{\infty}$ on $B_R(m_{\infty})$ for every $0 \le a \le b$. 
\end{proposition}
\begin{proof}
We will give a proof of the case for $s=0, a=1, b=r=2$ only.
By Definition \ref{strongdef} and Proposition \ref{compati7}, without loss of generality we can assume that $T_i \in L^{\infty}(B_R(m_i))$ for every $i \le \infty$ with $L_1:=\sup_{i \le \infty}||T_i||_{L^{\infty}}<\infty$ and that $T_i$ converges strongly to $T_{\infty}$ on $B_R(m_{\infty})$.

Let $k:=\mathrm{dim}\,M_{\infty}$.
Theorem \ref{dist3} yields that there exist a sequence $\{K_j\}_{j \in \mathbf{N}}$ of $K_j \subset M_{\infty}$ and $\{x_i^j\}_{1 \le i \le k, j \in \mathbf{N}} \subset M_{\infty}$ such that $\mathrm{Leb}\,K_j=K_j$,  $K_j \subset M_{\infty} \setminus \bigcup_{i=1}^kC_{x_i^j}$, $\upsilon (M_{\infty} \setminus \bigcup_jK_j)=0$ and that for every $z_{\infty} \in \bigcup_jK_j$ and every $\epsilon >0$ there exists $j=j(z_{\infty}, \epsilon)$ such that $z_{\infty} \in K_j$ and that $\{dr_{x_i^j}(w_{\infty})\}_i$ is an $\epsilon$-orthogonal basis on $T_{w_{\infty}}^*M_{\infty}$ for every $w_{\infty} \in K_j$. 

Fix $\epsilon >0$ and $z_{\infty} \in \bigcup_jK_j$.
Let $j:=j(z_{\infty}, \epsilon)$ as above.
Then an argument similar to the proof of Proposition \ref{contr} yields that there exists $r>0$ such that 
\[\frac{1}{\upsilon (B_t(z_{\infty}))}\int_{B_t(z_{\infty})}\left| \sum_{s, t}^k(T_{\infty}(dr_{x_s^j}, dr_{x_t^j}))^2 - |T_{\infty}|^2\right|d\upsilon<\Psi(\epsilon; n, L_1)\]
holds for every $t<r$.
Fix $t>0$ with $t<r$.
Then Proposition \ref{strong2} yields that 
\[\frac{1}{\mathrm{vol}\,B_t(z_l)}\int_{B_t(z_l)}\left| \sum_{s, t}^k(T_l(dr_{x_{s, l}^j}, dr_{x_{t, l}^j}))^2 - |T_l|^2\right|d\mathrm{vol}<\Psi(\epsilon; n, L_1)\]
holds for every sufficiently large $l \le \infty$, where $x_{s, l}^i \to x_s^i$.
Thus $(2)$ of Proposition \ref{ep} and an argument similar to the proof of Proposition \ref{contr} yield that
\[\frac{1}{\mathrm{vol}\,B_t(z_l)}\int_{B_t(z_l)}\left| \sum_{i=1}^kT_l(dr_{x_{s, l}^j}, dr_{x_{s, l}^j})-\mathrm{tr}\, T_l \right| d\mathrm{vol}  <\Psi(\epsilon; n, L_1)\]
holds for every sufficiently large $l \le \infty$.
Therefore we see that $\mathrm{Tr}\, T_l$ converges strongly to $\mathrm{Tr}\,T_{\infty}$ at $z_{\infty}$.
Thus we have the assertion.
\end{proof}
\begin{remark}
Let $R>0$ and let $A_i$ be a Borel subset of $B_R(m_i)$ for every $i \le \infty$ satisfying that $1_{A_i}$ converges strongly to $1_{A_{\infty}}$ at a.e. $z_{\infty} \in B_R(m_{\infty})$.
For a sequence $\{S_i\}_{i \le \infty}$ of $S_i \in L^p(T^r_sA_i)$, we say that \textit{$S_i$  $L^p$-converges strongly to $S_{\infty}$ on $A_{\infty}$} if $1_{A_i}S_i$ $L^p$-converges strongly to $1_{A_{\infty}}S_{\infty}$ on $B_R(m_{\infty})$.
Then we can get several properties for this convergence similar to that given in this section.
\end{remark}
\begin{remark}
Let $n \in \mathbf{N}$, $K \in \mathbf{R}$ and let $\{(Y_i, y_i, \upsilon_i)\}_{i \le \infty}$ be a sequence of $(n, K)$-Ricci limit spaces with $(Y_i, y_i, \upsilon_i) \to (Y_{\infty}, y_{\infty}, \upsilon_{\infty})$, $R>0$ and $T_i \in L^p(T^r_sB_R(y_i))$ for every $i \le \infty$ with $\sup_{i \le \infty}||T_i||_{L^p}<\infty$.
Then similarly, we can also consider $L^p$-weak or $L^p$-strong convergence $T_i \to T_{\infty}$ and show several properties similar to that given in this section.
\end{remark}
\section{Applications.}
\subsection{Convergence of Sobolev functions.}
In this subsection we consider the same setting in the previous section again:
\begin{enumerate}
\item $(M_{\infty}, m_{\infty}, \upsilon)$ is the $(n, -1)$-Ricci limit space of $\{(M_i, m_i, \underline{\mathrm{vol}})\}_{i<\infty}$ with $M_{\infty} \neq \{m_{\infty}\}$.
\item $R>0$, $1<p\le \infty$.
\end{enumerate}
Let $k:=\mathrm{dim}\,M_{\infty}$.
A main result of this subsection is Theorem \ref{7}.
\begin{theorem}\label{green}
Let $X_i \in L^{p}(TB_R(m_i))$ for $i \le \infty$ with $\sup_{i \le \infty}||X_i||_{L^p(B_R(m_i))}<\infty$. Assume that $X_i \in \mathcal{D}^p(\mathrm{div}^{\underline{\mathrm{vol}}}, B_R(m_i))$ for every $i<\infty$ with $\sup_{i<\infty}||\mathrm{div}^{\underline{\mathrm{vol}}}X_i||_{L^p}<\infty$, and that $X_i$ converges weakly to $X_{\infty}$ on $B_R(m_{\infty})$.
Then we see that  $X_{\infty} \in \mathcal{D}^p(\mathrm{div}^{\upsilon}, B_R(m_{\infty}))$ and that $\mathrm{div}^{\underline{\mathrm{vol}}}X_i$ converges weakly to $\mathrm{div}^{\upsilon}X_{\infty}$ on $B_R(m_{\infty})$.
\end{theorem}
\begin{proof}
Proposition \ref{weak com} yields that there exist a subsequence $\{i(j)\}_j$ and $h_{\infty} \in L^p(B_R(m_{\infty}))$ such that $\mathrm{div}^{\underline{\mathrm{vol}}}X_{i(j)}$ converges weakly to $h_{\infty}$ on $B_R(m_{\infty})$.
Let $f_{\infty}$ be a Lipschitz function on $B_R(m_{\infty})$ with compact support.
By \cite[Theorem $4.2$]{Ho}, without loss of generality we can assume that there exists 
a sequence $\{f_{i(j)}\}_{j< \infty}$ of Lipschitz functions $f_{i(j)}$ on $B_R(m_{i(j)})$ such that $\sup_{j}\mathbf{Lip}f_{i(j)}<\infty$, every $f_{i(j)}$ has compact support and that $(f_{i(j)}, df_{i(j)}) \to (f_{\infty}, df_{\infty})$ on $B_R(m_{\infty})$.
Since 
\[\int_{B_R(m_{i(j)})}\langle \nabla f_{i(j)}, X_{i(j)} \rangle d\underline{\mathrm{vol}}=-\int_{B_R(m_{i(j)})}f_{i(j)}\mathrm{div}^{\underline{\mathrm{vol}}}X_{i(j)}d\underline{\mathrm{vol}},\]
by letting $j \to \infty$, we have 
\[\int_{B_R(m_{\infty})}\langle \nabla f_{\infty}, X_{\infty}\rangle d\upsilon=-\int_{B_R(m_{\infty})}f_{\infty}h_{\infty}d\upsilon.\]
In particular we have $X_{\infty} \in \mathcal{D}^p(\mathrm{div}^{\upsilon}, B_R(m_{\infty}))$ and $h_{\infty}=\mathrm{div}^{\upsilon}X_{\infty}$.
The uniqueness of $\mathrm{div}^{\upsilon}X_{\infty}$ yields that $\mathrm{div}^{\underline{\mathrm{vol}}}X_i$ converges weakly to $\mathrm{div}^{\upsilon}X_{\infty}$ on $B_R(m_{\infty})$.
\end{proof}
\begin{remark}\label{apoiuy}
In general, $L^p$-strong convergence $X_i \to X_{\infty}$ does NOT imply $L^p$-strong convergence $\mathrm{div}^{\underline{\mathrm{vol}}}X_i \to \mathrm{div}^{\upsilon}X_{\infty}$.
We give a simple example:
Let $g_n$ be as in Remark \ref{zigzag} and put
\[f_n(t):=\int_0^t \int_0^sg_n(x)dxds\]
on $(0, 1)$. 
Then for every $1<p<\infty$ we see that $\nabla f_n$ $L^p$-converges strongly to $0$ on $(0, 1)$ and that $\Delta f_n$ does not $L^p$-converge strongly to $0$ on $(0, 1)$.
\end{remark}
\begin{remark}\label{diverge}
Let $U$ be an open subset of $M_{\infty}$, $X \in \mathcal{D}^1_{\mathrm{loc}}(\mathrm{div}^{\upsilon}, U)$ and $f$ a Lipschitz function on $U$ with compact support.
Then we see that $fX \in \mathcal{D}^1_{\mathrm{loc}}(\mathrm{div}^{\upsilon}, U)$ and that $\mathrm{div}^{\upsilon}(fX)=f\mathrm{div}^{\upsilon}X+g_{M_{\infty}}(\nabla f, X)$ holds.
Compare with Proposition \ref{hess2}.
It is easy to check that if $X$ has compact support, then 
\[\int_{U}\langle X, \nabla g \rangle d\upsilon=-\int_Ug\mathrm{div}^{\upsilon}Xd\upsilon\]
holds for every locally Lipschitz function $g$ on $U$.
\end{remark}
The following theorem is a key result to prove Theorem \ref{main}:
\begin{theorem}\label{conv2}
Let $X_i \in L^{p}(TB_R(m_i))$ for every $i\le \infty$ with $\sup_{i \le \infty}||X_i||_{L^p} < \infty$, $1< \hat{p}<\infty$ and $h_i \in H_{1, \hat{p}}(B_R(m_i))$ for every $i \le \infty$ with $\sup_{i \le \infty}||h_i||_{H_{1, \hat{p}}}<\infty$.
Assume that the following hold:
\begin{enumerate}
\item $\hat{q}<p$, where $\hat{q}$ is the conjugate exponent of $\hat{p}$.
\item $X_i$ converges weakly to $X_{\infty}$ on $B_R(m_{\infty})$.
\item $X_i \in \mathcal{D}^p(\mathrm{div}^{\underline{\mathrm{vol}}}, B_R(m_i))$ holds for every $i<\infty$ with $\sup_{i <\infty}||\mathrm{div}^{\underline{\mathrm{vol}}}X_i||_{L^p}<\infty$.
\item $h_i$ converges weakly to $h_{\infty}$ on $B_R(m_{\infty})$.
\end{enumerate}
Then 
\[\lim_{i \to \infty}\int_{B_R(m_i)}\langle X_i, \nabla h_i \rangle d\underline{\mathrm{vol}}=\int_{B_R(m_{\infty})}\langle X_{\infty}, \nabla h_{\infty}\rangle d\upsilon.\]
\end{theorem}
\begin{proof}
Let $\epsilon>0$ and $L := \sup_{i}(||X_i||_{L^p}+||h_i||_{H_{1, \hat{p}}})$.
Then \cite[Theorem $6.33$]{ch-co} yields that for every $i <\infty$ there exists a $C^{\infty}$-function $\phi_i^{\epsilon}$ on $M_i$ such that 
$0 \le \phi_i^{\epsilon} \le 1$, $\mathrm{supp} (\phi_i^{\epsilon}) \subset B_{R - \epsilon}(m_i)$, $\phi_i^{\epsilon}|_{B_{R-2\epsilon}(m_i)}\equiv 1$
and $\mathbf{Lip} \phi_i^{\epsilon} + ||\Delta \phi_i^{\epsilon}||_{L^{\infty}} \le C(n, R, \epsilon)$.
By Proposition \ref{conti com}, without loss of generality we can assume that there exists a Lipschitz function $\phi_{\infty}^{\epsilon}$ on $M_{\infty}$ such that $\phi_i^{\epsilon} \to \phi_{\infty}^{\epsilon}$ on $M_{\infty}$.
Note that Proposition \ref{strong com} yields that $h_i$ $L^{\hat{p}}$-converges strongly to $h_{\infty}$ on $B_R(m_{\infty})$.
Corollary \ref{61746174}, Proposition \ref{weak3} and Theorem \ref{green} yield
\begin{align*}
\int_{B_R(m_i)}\langle X_i, \nabla(\phi_i^{\epsilon}h_i)\rangle d\underline{\mathrm{vol}}&=-\int_{B_R(m_i)}\mathrm{div}^{\underline{\mathrm{vol}}} X_i \phi_i^{\epsilon}h_id\underline{\mathrm{vol}}\\
&\rightarrow -\int_{B_R(m_{\infty})} \mathrm{div}^{\upsilon} X_{\infty} \phi_{\infty}^{\epsilon}h_{\infty}d\upsilon = \int_{B_R(m_{\infty})}\langle X_{\infty}, \nabla(\phi_{\infty}^{\epsilon}h_{\infty})\rangle d\upsilon.
\end{align*}
On the other hand, Propositions \ref{weak3}, \ref{weak4} and \ref{conv678} yield
\[\int_{B_R(m_i)}\langle X_i, \nabla \phi_i^{\epsilon}\rangle h_id\underline{\mathrm{vol}} \rightarrow \int_{B_R(m_{\infty})}\langle X_{\infty}, \nabla \phi_{\infty}^{\epsilon}\rangle h_{\infty}d\upsilon.\]
Since
\[\int_{B_R(m_i)}\langle X_i, \nabla(\phi_i^{\epsilon}h_i)\rangle d\underline{\mathrm{vol}}=\int_{B_R(m_i)}\langle X_i, \nabla \phi_i^{\epsilon}\rangle h_id\underline{\mathrm{vol}} + \int_{B_R(m_i)}\langle X_i, \nabla h_i\rangle \phi_i^{\epsilon}d\underline{\mathrm{vol}}\]
holds for every $i \le \infty$, we have
\[\int_{B_R(m_i)}\langle X_i, \nabla h_i \rangle \phi_i^{\epsilon}d\underline{\mathrm{vol}} \rightarrow \int_{B_R(m_{\infty})}
\langle X_{\infty}, \nabla h_{\infty}\rangle \phi_{\infty}^{\epsilon} d\upsilon.\]
The H$\ddot{\text{o}}$lder inequality yields that 
\begin{align*}
&\left|\int_{B_R(m_i)}\phi_i^{\epsilon}\langle X_i, \nabla h_i \rangle d\underline{\mathrm{vol}} - \int_{B_R(m_{\infty})}\langle X_i, \nabla h_i \rangle d\underline{\mathrm{vol}} \right|\\
&\le \left( \int_{B_R(m_i)} |1-\phi_i^{\epsilon}|^{\hat{q}}|X_i|^{\hat{q}}d\underline{\mathrm{vol}}\right)^{1/\hat{q}}\left( \int_{B_R(m_i)}|\nabla h_i|^{\hat{p}}d\underline{\mathrm{vol}} \right)^{1/\hat{p}} \\
&\le  \left( \int_{B_R(m_i)} |1-\phi_i^{\epsilon}|^{\hat{q}\alpha}d\underline{\mathrm{vol}}\right)^{1/(\hat{q}\alpha)}\left( \int_{B_R(m_i)}|X_i|^pd\upsilon \right)^{1/p}L \\
&\le L^2\left( \underline{\mathrm{vol}}(B_R(m_i) \setminus B_{R-2\epsilon}(m_i)) \right)^{1/(\hat{q}\alpha)} <\Psi(\epsilon;n, L, R, p, \hat{p})
\end{align*}
holds for every $i \le \infty$, where $\alpha$ is the conjugate exponent of $p/\hat{q}>1$.
Therefore since $\epsilon$ is arbitrary, we have the assertion.
\end{proof}
\begin{corollary}\label{conv1}
Assume $p<\infty$.
Let $h_i \in H_{1, p}(B_R(m_i))$ for every $i\le\infty$ with $\sup_{i \le \infty}||h_i||_{H_{1, p}(B_R(m_i))}<\infty$.
Assume that $h_i$ converges weakly to $h_{\infty}$ on $B_R(m_{\infty})$.
Then $dh_{i}$ converges weakly to $dh_{\infty}$ on $B_R(m_{\infty})$.
In particular, $\liminf_{i \to \infty}||dh_{i}||_{L^p(B_R(m_i))}\ge ||dh_{\infty}||_{L^p(B_R(m_{\infty}))}$.
\end{corollary}
\begin{proof}
Let $x_{\infty} \in B_R(m_{\infty})$, $r>0$ with $\overline{B}_r(x_{\infty}) \subset B_R(m_{\infty})$, and let $f_i$ be a harmonic function on $B_r(x_i)$ for every $i < \infty$ with $\sup_{i < \infty}\mathbf{Lip}f_i <\infty$, where $x_i \to x_{\infty}$, and $f_{\infty}$ a Lipschitz function on $B_r(x_{\infty})$ with $f_i \to f_{\infty}$ on $B_r(x_{\infty})$.
Then Theorem \ref{conv2} and Proposition \ref{harm5} yield
\[\lim_{i \to \infty}\int_{B_r(x_i)}\langle df_i, dh_i \rangle d\underline{\mathrm{vol}}=\int_{B_r(x_{\infty})}\langle df_{\infty}, dh_{\infty}\rangle d\upsilon.\]
Thus  Theorem \ref{harm3} and Proposition \ref{compat} yield that $dh_{i}$ converges weakly to $dh_{\infty}$ on $B_R(m_{\infty})$.
\end{proof}
The following corollary is a refinement of a main theorem of \cite{Ho}.
Compare with \cite[Definition $4.4$ and Corollary $4.5$]{Ho}.
\begin{corollary}\label{l1}
Let $f_i$ be a Lipschitz function on $B_R(m_i)$ for every $i \le \infty$ with $\sup_{i \le \infty}\mathbf{Lip}f_i<\infty$ and $f_i \to f_{\infty}$ on $B_R(m_{\infty})$.
Then we have the following:
\begin{enumerate}
\item If $\{|df_i|^2\}_{i\le \infty}$ is upper semicontinuous at $z_{\infty} \in B_R(m_{\infty})$, then $df_i$ converges strongly to $df_{\infty}$ at $z_{\infty}$.
\item If $f_i \in C^2(B_R(m_i))$ holds for every $i<\infty$ with $\sup_{i<\infty}||\Delta f_i||_{L^1(B_R(m_i))}<\infty$, then $df_i$ converges strongly to $df_{\infty}$ on $B_R(m_{\infty})$.
\end{enumerate}
\end{corollary}
\begin{proof}
$(1)$ is a direct consequence of Corollary \ref{conv1}.
$(2)$ follows from $(1)$ and \cite[Proposition $4.9$]{Ho}.
\end{proof}
The assumption of uniform $L^1$-bounds as in $(2)$ of Corollary \ref{l1} is \textit{sharp} in the following sense:  
\begin{remark}\label{qqqqq}
Let $g_n$ be a smooth function on $\mathbf{R}$ as in Remark \ref{zigzag} and 
\[G_n(t):=\int_0^tg_n(s)ds\]
on $(0, 1)$.
Then it is easy to check the following:
\begin{enumerate}
\item $\sup_{n \in \mathbf{N}}\mathbf{Lip}G_n <\infty$.
\item $G_n \to 0$ on $(0, 1)$.
\item $||\Delta G_n||_{L^1((0,1))} \to \infty$.
\item $dG_n$ converges weakly to $0$ on $(0, 1)$.
\item For every $t \in (0, 1)$, $dG_n$ dose NOT converges strongly to $0$ at $t$.
\end{enumerate}
\end{remark}
\begin{remark}\label{ppppp}
Let $h_n$ be a smooth function on $[0, 1]$ satisfying that $h_n|_{[1/n+1/n^2, 1]}\equiv 0$, $|h_n|\le 1$, $|\nabla h_n| \le n$  and that $h_n(s)=-ns+1$ holds for every $s \in [0, 1/n]$.
Put
\[H_n(t):=\int_0^th_n(s)ds\]
on $(0, 1)$.
Then it is easy to check the following:
\begin{enumerate}
\item $\sup_{n< \infty}(\mathbf{Lip}H_n+||H_n||_{L^{\infty}} + ||\Delta H_n||_{L^1((0, 1))})<\infty$
\item $H_n \to 0$ on $(0, 1)$.
\item $||\Delta H_n||_{L^p((0, 1))} \to \infty$ for every $p>1$.
\item $dH_n$ converges strongly to $0$ on $(0, 1)$. 
\end{enumerate}
\end{remark}
We now give a compactness result about Sobolev functions with respect to the Gromov-Hausdorff topology.
Compare with Proposition \ref{strong com}.
\begin{theorem}\label{7}
Assume $p<\infty$.
Let $h_i \in H_{1, p}(B_R(m_i))$ for every $i<\infty$ with $\sup_{i<\infty}||h_i||_{H_{1,p}}<\infty$.
Then there exist $h_{\infty} \in H_{1, p}(B_R(m_{\infty}))$ and a subsequence $\{h_{i(j)}\}_j$ of $\{h_i\}_i$ such that 
$h_{i(j)}$ $L^p$-converges strongly to $h_{\infty}$ on $B_R(m_{\infty})$ and that $dh_{i(j)}$ converges weakly to $dh_{\infty}$ on $B_R(m_{\infty})$. 
In particular $\liminf_{j \to \infty}||dh_{i(j)}||_{L^p} \ge ||dh_{\infty}||_{L^p}$.
\end{theorem}
\begin{proof}
Let $T:=\sup_{i<\infty}||h_i||_{H_{1, p}}$.
Proposition \ref{strong com} yields that there exist $h_{\infty} \in L^p(B_R(m_{\infty}))$ and a subsequence $\{h_{i(j)}\}_j$ of $\{h_i\}_i$ such that $h_{i(j)}$ $L^p$-converges strongly to $h_{\infty}$ on $B_R(m_{\infty})$.
\begin{claim}\label{mmmm}
Let $r>0$ and $z_{\infty} \in B_R(m_{\infty})$ with $B_{5r}(z_{\infty}) \subset B_R(m_{\infty})$.
Then we see that $h_{\infty}|_{B_r(z_{\infty})} \in H_{1, p}(B_r(z_{\infty}))$ and that $dh_{i(j)}$ converges weakly to $dh_{\infty}$ on $B_r(z_{\infty})$.
\end{claim}
The proof is as follows.
Proposition \ref{mosco} yields that for every $j<\infty$ and every $L \ge 1$, there exists an $L$-Lipschitz function $(h_{i(j)})_L$ on $B_r(z_{i(j)})$ such that $||h_{i(j)}-(h_{i(j)})_L||_{L^{p}(B_r(z_{i(j)}))}\le \Psi(L^{-1}; n, R, T)$ and $||(h_{i(j)})_L||_{H_{1, p}}\le C(n, R, T)$, where $z_{i(j)} \to z_{\infty}$.
By Proposition \ref{conti com}, without loss of generality we can assume that there exists an $L$-Lipschitz function $(h_{\infty})_L$ on $B_r(z_{\infty})$ such that $(h_{i(j)})_L \to (h_{\infty})_L$ on $B_r(z_{\infty})$.
Corollary \ref{conv1} yields that $||(h_{\infty})_L||_{H_{1,p}}\le \liminf_{j \to \infty}||(h_{i(j)})_L||_{H_{1, p}}\le C(n, R, T)$ and $||h_{\infty}-(h_{\infty})_L||_{L^p(B_r(z_{\infty}))}= \lim_{j \to \infty}||h_{i(j)}-(h_{i(j)})_L||_{L^p(B_r(z_{i(j)}))}\le \Psi(L^{-1}; n, R, T)$ hold.
Thus by letting $L \to \infty$, we have $h_{\infty} \in H_{1, p}(B_r(z_{\infty}))$.
Then Claim \ref{mmmm} follows directly from Corollary \ref{conv1}. 

By Claim \ref{mmmm}, for every $r<R$ it is easy to check that $h_{\infty}|_{B_r(m_{\infty})} \in H_{1, p}(B_r(m_{\infty}))$ and that $dh_{i(j)}$ converges to $dh_{\infty}$ on $B_r(m_{\infty})$.
On the other hand, Corollary \ref{conv1} yields $||dh_{\infty}||_{L^{p}(B_r(m_{\infty}))}\le \liminf_{j \to \infty}||dh_{i(j)}||_{L^{p}(B_R(m_{\infty}))}$ for every $r<R$.
In particular, we have $||dh_{\infty}||_{L^p(B_R(m_{\infty}))}<\infty$.
Note that 
\[\frac{1}{\upsilon (B_t(x))}\int_{B_t(x)}\left| h_{\infty}-\frac{1}{\upsilon (B_t(x))}\int_{B_t(x)}h_{\infty}d\upsilon\right| d\upsilon \le tC(n, R)\left(\frac{1}{\upsilon (B_t(x))}\int_{B_t(x)}|dh_{\infty}|^pd\upsilon\right)^{1/p}\]
holds for every $\overline{B}_t(x) \subset B_R(m_{\infty})$.
Thus (the proof of) \cite[Theorem $1.1$]{HKT} yields $h_{\infty} \in H_{1, p}(B_R(m_{\infty}))$.
Therefore we have the assertion.
\end{proof}
\begin{remark}\label{rmk}
As a corollary of Theorem \ref{7} we have the following:
Define $E^p_i:L^p(B_R(m_i)) \to \mathbf{R}_{\ge 0} \cup \{\infty\}$ by
\[E^p_i(f):=\int_{B_R(m_i)}|df|^pd\underline{\mathrm{vol}}\]
if $f \in H_{1, p}(B_R(m_i))$, $E^p_i(f) \equiv \infty$ if otherwise.
Then $E^p_i$ \textit{compactly converges to} $E^p_{\infty}$ in the sense of Kuwae-Shioya (see \cite[Definition $4.5$]{KS2} for the precise definition).
Kuwae-Shioya proved this result for the case $p=2$.
See \cite[Corollary $6.3$]{KS2}.
\end{remark}
We end this subsection by giving the following `Sobolev-embedding' type theorem:
\begin{theorem}\label{sobolev emb}
Assume $p<\infty$.
Let $f_i \in H_{1, p}(B_R(m_i))$ for every $i \le \infty$ with $\sup_i ||f_i||_{H_{1, p}(B_R(m_i))}<\infty$.
If $f_i, df_i$ $L^p$-converge strongly to $f_{\infty}, df_{\infty}$ on $B_R(m_{\infty})$, respectively, then $f_i$ $L^{\hat{p}}$-converges strongly to $f_{\infty}$ on $B_R(m_{\infty})$ for some $\hat{p}:=\hat{p}(n, R, p) > p$.
\end{theorem}
\begin{proof}
Let $r<R$ and $L:=\sup_i ||f_i||_{H_{1, p}(B_R(m_i))}$.
Since the space of locally Lipschitz functions on $B_R(m_{\infty})$ in $H_{1, p}(B_R(m_{\infty}))$ is dense in $H_{1, p}(B_R(m_{\infty}))$, by \cite[Theorem $4.2$]{Ho},
without loss of generality we can assume that there exists a sequence $\{\hat{f}_{i}\}_{i<\infty}$ of Lipschitz functions $\hat{f}_{i}$ on $B_r(m_i)$ such that $\hat{f}_i, d\hat{f}_i$ $L^p$-converge strongly to $f_{\infty}|_{B_r(m_{\infty})}, d(f_{\infty}|_{B_r(m_{\infty})})$ on $B_r(m_{\infty})$, respectively.
Then the Poincar\'e inequality of type $(1, p)$ yields that
\begin{align*}
&\frac{1}{\mathrm{vol}\,B_t(x_i)}\int_{B_t(x_i)}\left| f_i-\hat{f}_{i}-\frac{1}{\mathrm{vol}\,B_t(x_i)}\int_{B_t(x_i)}(f_i-\hat{f}_{i})d\mathrm{vol} \right|d\mathrm{vol} \\
&\le tC(n, R)\left(\frac{1}{\mathrm{vol}\,B_t(x_i)}\int_{B_t(x_i)}|\nabla f_i-\nabla \hat{f}_{i}|^{p}d\mathrm{vol}\right)^{1/p}
\end{align*}
holds for every $i$, every $x_i \in B_r(m_i)$ and every $t>0$ with $B_t(x_i) \subset B_r(m_i)$.
Thus the Poincar\'e-Sobolev inequality \cite[Theorem $1$]{HK} yields 
\begin{align*}
&\left(\frac{1}{\mathrm{vol}\,B_r(m_i)}\int_{B_r(m_i)}\left| f_i-\hat{f}_{i}-\frac{1}{\mathrm{vol}\,B_r(m_i)}\int_{B_r(m_i)}(f_i-\hat{f}_{i})d\mathrm{vol} \right|^{\tilde{p}}d\mathrm{vol}\right)^{1/\tilde{p}} \\
&\le C(n, R, p)\left(\frac{1}{\mathrm{vol}\,B_r(m_i)}\int_{B_r(m_i)}|\nabla f_i-\nabla \hat{f}_{i}|^{p}d\mathrm{vol}\right)^{1/p}
\end{align*}
for some $\tilde{p}:=\tilde{p}(n, R, p)>p$.
Thus by letting $i \to \infty$, we see that $f_i$ $L^{\tilde{p}}$-converges strongly to $f_{\infty}$ on $B_r(m_{\infty})$.
Note that an argument similar to that above yields $\sup_{i}||f_i||_{L^{\tilde{p}}(B_R(m_{i}))}\le C(n, p, R, L)$.
Let $\hat{p}:=\hat{p}(n, R, p)$ with $p<\hat{p}<\tilde{p}$.
Then the H$\ddot{\text{o}}$lder inequality yields $||f_i||_{L^{\hat{p}}(B_R(m_i))}=||f_i||_{L^{\hat{p}}(B_r(m_i))} \pm \Psi(R-r;n, R, L, p)$ for every $i$.
Thus we have $\lim_{i \to \infty}||f_i||_{L^{\hat{p}}(B_R(m_i))}=||f_{\infty}||_{L^{\hat{p}}(B_R(m_{\infty}))}$.
Therefore Proposition \ref{222} yields the assertion.
\end{proof}
\subsection{$p$-Laplacian.}
In this subsection we discuss a convergence of $p$-Laplacians with respect to the Gromov-Hausdorff topology.
We will always consider the following setting similar to that in subsection $4.1$: $(M_{\infty}, m_{\infty}, \upsilon)$ is the Ricci limit space of $\{(M_i, m_i, \underline{\mathrm{vol}})\}_i$ with $M_{\infty} \neq \{m_{\infty}\}$,  $R>0$, $1<p<\infty$ and $q$ is the conjugate exponent of $p$.
\begin{proposition}\label{plap}
Let $r, s \in \mathbf{Z}_{\ge 0}$ and $T_i \in L^p(T^r_sB_R(m_i))$ for every $i \le \infty$ with $\sup_{i \le \infty}||T_{i}||_{L^p}<\infty$.
Assume that $T_i$ $L^p$-converges strongly to $T_{\infty}$ on $B_R(m_{\infty})$.
Then $T_i^{(p-1)}$ $L^q$-converges strongly to $T_{\infty}^{(p-1)}$ on $B_R(m_{\infty})$. 
\end{proposition}
\begin{proof}
For $\hat{L}>0$, let $K_{\hat{L}}$ be as in the proof of Proposition \ref{contr}. 
Let $x_{\infty} \in K_{\hat{L}}$ satisfying that
\[\lim_{r \to 0}\frac{1}{\upsilon (B_r(x_{\infty}))}\int_{B_r(x_{\infty})}\left||T_{\infty}|^{p-1}- (|T_{\infty}|(x_{\infty}))^{p-1}\right|d\upsilon=0\]
holds.
Let $c:=|T_{\infty}|(x_{\infty})$.
For every $\epsilon>0$ there exists $r_0=r_0(\epsilon)>0$ such that
\[\frac{1}{\upsilon (B_t(x_{\infty}))}\int_{B_t(x_{\infty})}\left||T_{\infty}|^{p-1}- c^{p-1}\right|d\upsilon<\epsilon\]
holds for every $t<r_0$.
Fix $\epsilon>0$ and $t<r_0(\epsilon)$.
Since $|T_i|^{p-1}$ $L^{\hat{p}}$-converges strongly to $|T_{\infty}|^{p-1}$ on $B_R(m_{\infty})$, where $\hat{p}:= p/(p-1)>1$,
we see that
\[\frac{1}{\mathrm{vol}\,B_t(x_{i})}\int_{B_t(x_{i})}\left||T_{i}|^{p-1}- c^{p-1}\right|d\mathrm{vol}<\epsilon\]
holds for every sufficiently large $i$, where $x_i \to x_{\infty}$.
Let $\{x_{l, i}\}_{1\le l \le r+s, i \le \infty}$ be a collection of $x_{l, i} \in M_i$ with $x_{l, i} \to x_{l, \infty}$.
If $c=0$, then 
\[\left| \frac{1}{\mathrm{vol}\,B_t(x_{i})}\int_{B_t(x_i)}\left\langle T_i^{(p-1)},  \bigotimes_{l=1}^r \nabla r_{x_{l, i}} \otimes \bigotimes _{l=r+1}^{r+s}dr_{x_{l, i}}\right\rangle d\mathrm{vol} \right| \le \frac{1}{\mathrm{vol}\,B_t(x_{i})}\int_{B_t(x_i)}|T_i|^{p-1}d\mathrm{vol}<\epsilon\]
holds for every sufficiently large $i \le \infty$.
In particular $T_i^{(p-1)}$ converges weakly to $T_{\infty}^{(p-1)}$ at $x_{\infty}$.

Next we assume $c \neq 0$.
Let $A(t, i):=\{y \in B_t(x_i); ||T_{i}(y)|^{p-1}- c^{p-1}|<\epsilon^{1/2}\}$.
Then we see that $\mathrm{vol}\,A(t, i)/\mathrm{vol}\,B_t(x_i) \ge 1-\epsilon^{1/2}$ holds for every sufficiently large $i \le \infty$.
On the other hand, the H$\ddot{\text{o}}$lder inequality yields that
\begin{align*}
&\frac{1}{\mathrm{vol}\, B_t(x_i)}\int_{B_t(x_i)}\left\langle T_i^{(p-1)}, \bigotimes_{l=1}^r \nabla r_{x_{l, i}} \otimes \bigotimes _{l=r+1}^{r+s}dr_{x_{l, i}}\right\rangle d\mathrm{vol}\\
&= \frac{1}{\mathrm{vol}\, B_t(x_i)}\int_{A(t, i)}\left \langle T_i^{(p-1)}, \bigotimes_{l=1}^r \nabla r_{x_{l, i}} \otimes \bigotimes _{l=r+1}^{r+s}dr_{x_{l, i}}\right\rangle d\mathrm{vol} \\
& \, \, \,\,\,\,\,\, \pm \left( \frac{\mathrm{vol} (B_t(x_i) \setminus A(t, i))}{\mathrm{vol}\,B_t(x_i)}\right)^{1/\hat{q}}\left( \frac{1}{\mathrm{vol}\,B_t(x_i)}\int_{B_t(x_i)}|T_i|^pd\mathrm{vol}\right)^{1/\hat{p}}\\
&=\frac{1}{\mathrm{vol}\, B_t(x_i)}\int_{A(t, i)}\left \langle T_i^{(p-1)}, \bigotimes_{l=1}^r \nabla r_{x_{l, i}} \otimes \bigotimes _{l=r+1}^{r+s}dr_{x_{l, i}}\right\rangle d\mathrm{vol} \pm \Psi(\epsilon; n, \hat{L}, c, p)\\
&=\frac{1}{\mathrm{vol}\, B_t(x_i)}\int_{A(t, i)}c^{p-2}\left\langle T_i, \bigotimes_{l=1}^r \nabla r_{x_{l, i}} \otimes \bigotimes _{l=r+1}^{r+s}dr_{x_{l, i}}\right\rangle d\mathrm{vol} \pm \Psi(\epsilon; n, \hat{L}, c, p) \\
&=\frac{1}{\mathrm{vol}\, B_t(x_i)}\int_{B_t(x_i)}c^{p-2}\left\langle T_i, \bigotimes_{l=1}^r \nabla r_{x_{l, i}} \otimes \bigotimes _{l=r+1}^{r+s}dr_{x_{l, i}}\right\rangle d\mathrm{vol} \pm \Psi(\epsilon; n, \hat{L}, c, p) \\
&=\frac{1}{\upsilon (B_t(x_{\infty}))}\int_{B_t(x_{\infty})}c^{p-2}\left\langle T_i, \bigotimes_{l=1}^r \nabla r_{x_{l, \infty}} \otimes \bigotimes _{l=r+1}^{r+s}dr_{x_{l, \infty}}\right\rangle d\upsilon \pm \Psi(\epsilon; n, \hat{L}, c, p) \\
&=\frac{1}{\upsilon (B_t(x_{\infty}))}\int_{A(t, \infty)}c^{p-2}\left\langle T_{\infty}, \bigotimes_{l=1}^r \nabla r_{x_{l, \infty}} \otimes \bigotimes _{l=r+1}^{r+s}dr_{x_{l, \infty}}\right\rangle d\upsilon \pm \Psi(\epsilon; n, \hat{L}, c, p) \\
&=\frac{1}{\upsilon (B_t(x_{\infty}))}\int_{B_t(x_{\infty})}\left\langle T_{\infty}^{(p-1)}, \bigotimes_{l=1}^r \nabla r_{x_{l, \infty}} \otimes \bigotimes _{l=r+1}^{r+s}dr_{x_{l, \infty}}\right\rangle d\upsilon \pm \Psi(\epsilon; n, \hat{L}, c, p) \\
\end{align*}
holds for every sufficiently large $i < \infty$,
where $\hat{q}$ is the conjugate exponent of $\hat{p}$. 
Thus we see that $T_i^{(p-1)}$ converges weakly to $T_{\infty}^{(p-1)}$ at $x_{\infty}$.
Therefore Corollary \ref{weakkk} yields that  $T_i^{(p-1)}$ converges weakly to $T_{\infty}^{(p-1)}$ on $B_R(m_{\infty})$.
On the other hand, since 
\begin{align*}
\lim_{i \to \infty}\int_{B_R(m_i)}|T_i^{(p-1)}|^qd\underline{\mathrm{vol}}&=\lim_{i \to \infty}\int_{B_R(m_i)}|T_i|^{p}d\underline{\mathrm{vol}} \\
&=\int_{B_R(m_{\infty})}|T_{\infty}|^{p}d\upsilon=\int_{B_R(m_{\infty})}|T_{\infty}^{(p-1)}|^{q}d\upsilon,
\end{align*}
the assertion follows from Propositions \ref{hon}.
\end{proof}
\begin{remark}
Note that in general the $L^p$-weak convergence $T_i \to T_{\infty}$ on $B_R(m_{\infty})$ does NOT imply the $L^q$-weak convergence $T_i^{(p-1)} \to T_{\infty}^{(p-1)}$ on $B_R(m_{\infty})$. 
For example let $g_n$ be a smooth function on $\mathbf{R}$ as in Remark \ref{zigzag} and $\hat{g}_n:=g_n+1$.
Then since $\hat{g}_n^{(2)} = (g_n)^2 +2g_n +1$, 
Remark \ref{zigzag} yields that $\hat{g}_n^{(2)}$ converges weakly to $1/3+1=4/3 (\neq 1)$ on $\mathbf{R}$.
\end{remark}
For $f \in H_{1, p}(B_R(m_{\infty}))$ with $(\nabla f)^{(p-1)} \in \mathcal{D}^q(\mathrm{div}^{\upsilon}, B_R(m_{\infty}))$, let $\Delta_p^{\upsilon}f:=-\mathrm{div}^{\upsilon}(\nabla f)^{(p-1)} \in L^q(B_R(m_{\infty}))$. 
\begin{theorem}\label{plap1}
Let $f_i \in H_{1, p}(B_R(m_i))$ for every $i \le \infty$ with $\sup_{i\le \infty}||f_i||_{H_{1, p}}< \infty$.
Assume that $(\nabla f_i)^{(p-1)} \in \mathcal{D}^q(\mathrm{div}^{\upsilon}, B_R(m_i))$ holds for every $i<\infty$ with 
$\sup_{i<\infty}||\Delta_p^{\underline{\mathrm{vol}}}f_i||_{L^{q}(B_R(m_i))}<\infty$
and that $f_i, \nabla f_i$ $L^p$-converge strongly to $f_{\infty}, \nabla f_{\infty}$ on $B_R(m_{\infty})$, respectively.
Then we see that $(\nabla f_{\infty})^{(p-1)} \in \mathcal{D}^q(\mathrm{div}^{\upsilon}, B_R(m_{\infty}))$ and that $\Delta_p^{\underline{\mathrm{vol}}}f_i$ converges weakly to $\Delta_p^{\upsilon}f_{\infty}$ on $B_R(m_{\infty})$.
\end{theorem}
\begin{proof}
The assertion follows from Theorem \ref{green} and Proposition \ref{plap}.
\end{proof}
We now are in a position to prove Theorem \ref{p-eigen}:

\textit{A proof of Theorem \ref{p-eigen}.}

We start by introducing the following claim proved by Wu-Wang-Zheng in \cite{WWZ}:
\begin{claim}\cite[Lemma $2.2$]{WWZ}\label{char}
Let $(X, \nu) \in  \overline{\mathcal{M}(n, d, K)}$ and $f \in H_{1, p}(X)$.
Then for $t \in \mathbf{R}$ the following two conditions are equivalent:
\begin{enumerate}
\item $||f+t||_{L^p}=\min _{s \in \mathbf{R}}||f+s||_{L^p}$.
\item \[\int_X(f+t)^{(p-1)}d\nu=0.\]
\end{enumerate}
Moreover there exists a unique $s_0 \in \mathbf{R}$ such that $||f+s_0||_{L^p}=\min _{s \in \mathbf{R}}||f+s||_{L^p}$.
\end{claim}
Denote $s_0$ as above by $s(f, X)$.
As a corollary, we see that the set 
\[\left\{f \in H_{1, p}(X); f \not \equiv 0, \int_X f^{(p-1)}d\nu=0 \right\}\]
is not empty if $X$ is not a single point.
\begin{claim}\label{existence}
Let $(X, \nu) \in  \overline{\mathcal{M}(n, d, K)}$.
Assume that $X$ is not a single point.
Then there exists $f \in H_{1, p}(X)$ such that $||f||_{L^p}=1$, $||\nabla f||_{L^p}^p=\lambda_{1, p}(X)>0$ and
\[\int_X f^{(p-1)}d\nu=0.\]
\end{claim}
The proof is as follows.
Let $\{f_i\}_{i<\infty} \subset H_{1, p}(X)$ satisfying that $||f_i||_{L^p}\equiv 1$, $||\nabla f_i||_{L^p}^p \to \lambda_{1, p}(X)$ and 
\[\int_{X}f_i^{(p-1)}d\nu \equiv 0.\]
By Theorem \ref{7}, without loss of generality we can assume that there exists $f \in H_{1, p}(X)$ such that $f_i$ $L^p$-converges strongly to $f$ on $X$ and that $\nabla f_i$ converges weakly to $\nabla f$ on $X$.
Thus we have $||f||_{L^p}=\lim_{i \to \infty}||f_i||_{L^p}=1$.
Proposition \ref{plap} yields that $f_i^{(p-1)}$ $L^q$-converges strongly to $f^{(p-1)}$ on $X$.
In particular we have 
\[\int_Xf^{(p-1)}d\nu=\lim_{i \to \infty}\int_Xf_i^{(p-1)}d\nu=0.\]
Thus by the definition of $\lambda_{1, p}(X)$, we have $||\nabla f||_{L^p}^p \ge \lambda_{1, p}(X)$.
On the other hand, Proposition \ref{lower semi} yields $\liminf_{i \to \infty}||\nabla f_i||_{L^p} \ge ||\nabla f||_{L^p}$.
Thus we have $||\nabla f||_{L^p}^p = \lambda_{1, p}(X)$.
Finally we will prove $\lambda_{1, p}(X)>0$.
Assume $\lambda_{1, p}(X)=0$. Then since $||\nabla f||_{L^p}=0$, the Poincar\'e inequality yields that $f$ is a constant function.
However since
\[||f||_{L^p}=1\,\,\mathrm{and}\,\,\int_Xf^{(p-1)}d\nu=0,\]
this is a contradiction. 
Therefore we have Claim \ref{existence}.

Let $\{(X_i, \nu_i)\}_{i \le \infty} \subset \overline{\mathcal{M}(n, d, K)}$.
Assume that $(X_i, \nu_i) \in \mathcal{M}(n, d, K)$ for every $i<\infty$ and that $(X_i, \nu_i)$ Gromov-Hausdorff converges to $(X_{\infty}, \nu_{\infty})$.
It suffices to check $\lim_{i \to \infty} \lambda_{1, p}(X_i)=\lambda_{1, p}(X_{\infty})$.

If $X_{\infty}$ is a single point, then Naber-Valtorta's sharp estimates \cite[Theorem $1.4$]{NV} (or the proof of Corollary \ref{asympto}) yields $\lambda_{1, p}(X_i) \to \infty = \lambda_{1, p}(X_{\infty})$.
Thus we assume that $X_{\infty}$ is not a single point.

First we will check $\limsup_{i \to \infty}\lambda_{1, p}(X_i)\le \lambda_{1, p}(X_{\infty})$.
Let $f_{\infty} \in H_{1, p}(X_{\infty})$ satisfying that $||f_{\infty}||_{L^p}=1$, $||\nabla f_{\infty}||_{L^p}^p=\lambda_{1, p}(X_{\infty})$ and
\[\int_{X_{\infty}}f_{\infty}^{(p-1)}d\nu_{\infty}=0.\]
Since the space of Lipschitz functions on $X_{\infty}$ is dense in $H_{1, p}(X_{\infty})$, by \cite[Theorem $4.2$]{Ho} without loss of generality, we can assume that there exists a sequence $\{f_i\}_{i<\infty}$ of Lipschitz functions $f_i$ on $X_i$ such that $f_i, \nabla f_i$ $L^p$-converge strongly to $f_{\infty}, \nabla f_{\infty}$ on $X_{\infty}$, respectively. 
Then we have $|s(f_i, X_i)| = ||s(f_i, X_i) + f_i -f_i||_{L^p}\le ||s(f_i, X_i) + f_i||_{L^p}+ ||f_i||_{L^p} \le 2||f_i||_{L^p} \le 3$ for every sufficiently large $i<\infty$.
Therefore without loss of generality we can assume that there exists $s_{\infty} \in \mathbf{R}$ such that $s(f_i, X_i) \to s_{\infty}$.
\begin{claim}\label{22334455}
$s_{\infty}=0$.
\end{claim}
The proof is as follows.
Since $s(f_i, X_i)$ is the minimizer, we have $||f_i + s(f_i, X_i)||_{L^p}\le ||f_i||_{L^p}$.
Letting $i \to \infty$ gives $||f_{\infty}+s_{\infty}||_{L^p}\le ||f_{\infty}||_{L^p}$.
Since  
\[\int_{X_{\infty}}f_{\infty}^{(p-1)}d\nu_{\infty}=0,\]
Claim \ref{char} yields $||f_{\infty}||_{L^p}\le ||f_{\infty}+s_{\infty}||_{L^p}$.
Thus by the uniqueness of the minimizer as in Claim \ref{char}, we have Claim \ref{22334455}.

By Claim \ref{22334455}, we have
\begin{align*}
\limsup_{i \to \infty}\lambda_{1, p}(X_i)\le \lim_{i \to \infty}\frac{||\nabla (f_i+s(f_i, X_i))||^p_{L^p}}{||f_i +s(f_i, X_i)||_{L^p}^p}=\frac{||\nabla f_{\infty}||_{L^p}^p}{||f_{\infty}||_{L^p}^p}=\lambda_{1, p}(X_{\infty}).
\end{align*}

Next we will show $\liminf_{i \to \infty}\lambda_{1, p}(X_i) \ge \lambda_{1, p}(X_{\infty})$.
For every $i<\infty$, let $g_i \in H_{1, p}(X_i)$ satisfying that $||g_i||_{L^p}=1$, $||\nabla g_i||_{L^p}^p=\lambda_{1, p}(X_i)$ and
\[\int_{X_i}g_i^{(p-1)}d\nu_i=0.\]
Since  $\limsup_{i \to \infty}\lambda_{1, p}(X_i)\le \lambda_{1, p}(X_{\infty})<\infty$, by Theorem \ref{7} without loss of generality we can assume that there exists $g_{\infty} \in H_{1, p}(X_{\infty})$ such that $g_i$ $L^p$-converges strongly to $g_{\infty}$ on $X_{\infty}$ and that $\nabla g_i$ converges weakly to $\nabla g_{\infty}$ on $X_{\infty}$.
Thus we have $||g_{\infty}||_{L^p}=\lim_{i \to \infty}||g_{i}||_{L^p}=1$.
Proposition \ref{plap} yields that $g_i^{(p-1)}$ $L^q$-converges strongly to $g_{\infty}^{(p-1)}$ on $X_{\infty}$.
In particular we have 
\[\int_{X_{\infty}}g_{\infty}^{(p-1)}d\nu_{\infty}=\lim_{i \to \infty}\int_{X_i}g_i^{(p-1)}d\nu_i=0.\]
Thus we have $||\nabla g_{\infty}||_{L^p}^p\ge \lambda_{1, p}(X_{\infty})$.
On the other hand, Proposision \ref{lower semi} yields
\[\liminf_{i \to \infty}\lambda_{1, p}(X_i)=\liminf_{i \to \infty}||\nabla g_i||^p_{L^p} \ge ||\nabla g_{\infty}||_{L^p}^p \ge \lambda_{1, p}(X_{\infty}).\] 
Therefore we have the assertion.   $\,\,\,\,\,\,\,\,\,\,\,\Box$

We end this subsection by giving the following two-sided bounds for $\lambda_{1, p}$.
It is worth pointing out that in \cite{Ma, NV, V}, Matei, Naber-Valtorta, and Valtorta give the sharp lower bounds on $\lambda_{1, p}$.
See also \cite{KN, LF, Tak, V2, WWZ, Zha}.
\begin{corollary}\label{asympto}
Let $K \in \mathbf{R}$ and let $M$ be an $n$-dimensional compact Riemannian manifold with
\[(\mathrm{diam}\,M)^2\mathrm{Ric}_M \ge (n-1)K.\]
Then 
\[0<C_1(n, p, K) \le \frac{\lambda_{1, p}(M)}{(\mathrm{diam}\,M)^p} \le C_2(n, p, K)<\infty.\]
\end{corollary}
\begin{proof}
Let
\[C_1(n, p, K):= \min_{X \in \overline{\mathcal{M}(n, 1, K)\setminus \mathcal{M}(n, 1/2, K)}} \lambda_{1, p}(X)\, \,\, \mathrm{and}\,\,\,C_2(n, p, \kappa):= \max_{X \in \overline{\mathcal{M}(n, 1, K)\setminus \mathcal{M}(n, 1/2, K)}} \lambda_{1, p}(X).\]
For a rescaled metric $(\mathrm{diam}\,M)^{-2}g_M$, since $(M, (\mathrm{diam}\,M)^{-2}g_M, \underline{\mathrm{vol}}) \in \mathcal{M}(n, 1, K) \setminus  \mathcal{M}(n, 1/2, K)$, we have
$C_1(n, p, K)\le \lambda_{1, p}(M, (\mathrm{diam}\,M)^{-2}g_M) \le C_2(n, p, K)$.
Since $\lambda_{1, p}(M, (\mathrm{diam}\,M)^{-2}g_M)=(\mathrm{diam}\,M)^{-p}\lambda_{1, p}(M, g_M)$, we have the assertion.
\end{proof}
\subsection{Convergence of Hessians.}
In this subsection we will prove Theorem \ref{main}.
Throughout this subsection we will always consider the following setting:
\begin{enumerate}
\item $(M_{\infty}, m_{\infty}, \upsilon)$ is the Ricci limit space of $\{(M_i, m_i, \underline{\mathrm{vol}})\}_i$ with $M_{\infty} \neq \{m_{\infty}\}$, and $R>0$.
\item $\mathcal{A}_{2nd}$ is a weakly second order differential structure on $(M_{\infty}, \upsilon)$ associated with  $\{(M_i, m_i, \underline{\mathrm{vol}})\}_i$.
\end{enumerate}
The following is a generalization of \cite[Corollary $4.6$]{Ho}.
Compare with Theorem \ref{plap1}:
\begin{proposition}\label{laplap}
Let $1<p<\infty$, $f_i \in H_{1, p}(B_R(m_i))$ for every $i<\infty$ with $\sup_{i <\infty}||f_i||_{H_{1, p}}<\infty$, and $f_{\infty} \in L^p(B_R(m_{\infty}))$.
Assume that $f_i$ converges weakly to $f_{\infty}$ on $B_R(m_{\infty})$ and that $\nabla f_i \in \mathcal{D}^p(\mathrm{div}^{\underline{\mathrm{vol}}}, B_R(m_i))$ holds for every $i<\infty$ with $\sup_{i<\infty}||\mathrm{div}^{\underline{\mathrm{vol}}}\nabla f_i||_{L^p}<\infty$.
Then we have the following:
\begin{enumerate}
\item $f_i$ $L^p$-converges strongly to $f_{\infty}$ on $B_R(m_{\infty})$.
\item $f_{\infty} \in H_{1, p}(B_R(m_{\infty}))$ and $\nabla f_{\infty} \in \mathcal{D}^p(\mathrm{div}^{\upsilon}, B_R(m_{\infty}))$.
\item $\nabla f_i, \mathrm{div}^{\underline{\mathrm{vol}}}\nabla f_i$ converge weakly to $\nabla f_{\infty}, \mathrm{div}^{\upsilon}\nabla f_{\infty}$ on $B_R(m_{\infty})$, respectively.
\end{enumerate}
In particular if $p\ge 2$, then $f_{\infty} \in \mathcal{D}^2(\Delta^{\upsilon}, B_R(m_{\infty}))$ and  $\Delta^{\upsilon}f_{\infty}=-\mathrm{div}^{\upsilon}\nabla f_{\infty}$.
\end{proposition}
\begin{proof}
It follows directly from Theorems \ref{green} and \ref{7}.
\end{proof}
\begin{remark}\label{eigen conti}
We recall a continuity of eigenfunctions with respect to the Gromov-Hausdorff topology.
Let $\phi_{\infty}$ be an eigenfunction associated with the eigenvalue $\lambda_{\infty}$ with respect to the Dirichlet problem on $B_R(m_{\infty})$. 
Then there exist $\{\lambda_i\}_i \subset \mathbf{R}_{>0}$ and a sequence $\{\phi_i\}_{i<\infty}$ of eigenfunctions $\phi_i \in C^{\infty}(B_R(m_i))$ associated with the eigenvalues $\lambda_{i}$ with respect to the Dirichlet problems on $B_R(m_i)$ such that $\lambda_i \to \lambda_{\infty}$ and that $\phi_i$ $L^2$-converges strongly to $\phi_{\infty}$ on $B_R(m_{\infty})$.
In particular we see that $\Delta \phi_i$ $L^2$-converges strongly to $\Delta^{\upsilon} \phi_{\infty}$ on $B_R(m_{\infty})$.
See \cite[Theorem $7.11$]{ch-co3}, \cite[Lemma $5.17$]{di2} and \cite[Lemma $5.8$]{KS} for the details.
Note that in \cite{Ho3} these results played crucial roles to get Theorem \ref{2nd}.
See also \cite{di2, Kasue1, Kasue2, ka-ku1, ka-ku2} for a convergence of heat kernels.
\end{remark}
We now are in a position to prove Theorem \ref{main}:

\textit{A proof of Theorem \ref{main}.}

Without loss of generality we can assume $K=-1$.
Proposition \ref{laplap} yields (\ref{a}), $f_{\infty} \in \mathcal{D}^2(\Delta^{\upsilon}, B_R(m_{\infty}))$, (\ref{f}) and that $\nabla f_i$ converges weakly to $\nabla f_{\infty}$ on $B_R(m_{\infty})$. 
\begin{claim}\label{7000}
Let $R>0$, $L>0$, and let $(M, m)$ be a pointed $n$-dimensional complete Riemannian manifold with $\mathrm{Ric}_M \ge -(n-1)$, and $f$ a $C^2$-function on $B_R(m)$ with
$||f||_{H_{1, 2}(B_R(m))}+||\Delta f||_{L^2(B_R(m))} \le L$.
Then for every $r<R$ we see that $||\nabla f||_{L^{2p_0}(B_r(m))} +||\mathrm{Hess}_f||_{L^2(B_r(m))}\le C(L, n, r, R)$ holds for some $p_0:=p_0(n, R)>1$.   
\end{claim}
The proof is as follows. 
Let $r>0$ with $r<R$. 
Since $|\nabla|\nabla f|^2|\le 2|\mathrm{Hess}_f||\nabla f| \le |\mathrm{Hess}_f|^2+|\nabla f|^2$,
the Poincar\'e inequality of type $(1, 1)$ (see for instance \cite[Theorem $2.15$]{ch-co3}) yields that 
\begin{align*}
&\frac{1}{\mathrm{vol}\,B_t(x)}\int_{B_t(x)}\left| |\nabla f|^2-\frac{1}{\mathrm{vol}\,B_t(x)}\int_{B_t(x)}|\nabla f|^2d\mathrm{vol} \right|d\mathrm{vol} \\
&\le tC(n, R)\frac{1}{\mathrm{vol}\,B_t(x)}\int_{B_t(x)}\left(|\mathrm{Hess}_f|^2 + |\nabla f|^2\right)d\mathrm{vol}
\end{align*}
holds for every $B_t(x) \subset B_r(m)$.
Thus the Poincar\'e-Sobolev inequality \cite[Theorem $1$]{HK} yields that there exists $p_0:=p_0(n, R)>1$ such that $p_0<2$ and 
\begin{align*}
&\left(\frac{1}{\mathrm{vol}\,B_r(m)}\int_{B_r(m)}\left|  |\nabla f|^2-\frac{1}{\mathrm{vol}\,B_r(m)}\int_{B_r(m)}|\nabla f|^2d\mathrm{vol}\right|^{p_0}d\mathrm{vol}\right)^{1/p_0} \\
&\le rC(n, R)\frac{1}{\mathrm{vol}\,B_r(m)}\int_{B_r(m)}\left(|\mathrm{Hess}_f|^2 + |\nabla f|^2\right)d\mathrm{vol}
\end{align*}
hold.
On the other hand, by an argument similar to that in \cite[Remark $4.2$]{Ho}, we have $||\mathrm{Hess}_f||_{L^2(B_r(m))}\le C(L, n, r, R)$.
Thus we have Claim \ref{7000}.
\begin{claim}\label{7001}
Let $f$ be as in Claim \ref{7000}, $\epsilon>0$, $r<R$ and $\phi \in C^0(B_R(m))$ satisfying that $0 \le \phi \le 1$, $\phi|_{B_{r-\epsilon}(m)}\equiv 1$ and $\mathrm{supp}(\phi) \subset B_r(m)$.
Then we have 
\[\left|\int_{B_r(m)}\phi |\nabla f|^2d\underline{\mathrm{vol}}-\int_{B_r(m)}|\nabla f|^2d\underline{\mathrm{vol}}\right|<\Psi(\epsilon; L, n, r, R).\]
\end{claim}
Because the H$\ddot{\text{o}}$lder inequality and Claim \ref{7000} yield 
\begin{align*}
\left|\int_{B_r(m)}\phi |\nabla f|^2d\underline{\mathrm{vol}}-\int_{B_r(m)}|\nabla f|^2d\underline{\mathrm{vol}}\right| \le ||1-\phi ||_{L^{q_0}(B_r(m))}||\nabla f||_{L^{2p_0}(B_r(m))}^{2} \le \Psi(\epsilon; L, n, r, R),
\end{align*}
where $q_0$ is the conjugate exponent of $p_0$.

Let $\epsilon>0$, $r<R$, $L:=\sup_{i < \infty}(||f_i||_{H_{1,2}(B_R(m_i))}+||\Delta f_i||_{L^2(B_R(m_i))})$ and let $\phi_i$ be a $C(\epsilon, r)$-Lipschitz function on $M_i$ for every $i \le \infty$ satisfying that $0 \le \phi_i \le 1$, $\phi_i|_{B_{r-\epsilon}(m_i)} \equiv 1$, $\mathrm{supp}(\phi_i) \subset B_r(m_i)$ and $\phi_i \to \phi_{\infty}$ on $M_{\infty}$.
\begin{claim}\label{7002}
We have 
\[\lim_{i \to \infty}\int_{B_r(m_i)}\phi_i|\nabla f_i|^2d\underline{\mathrm{vol}}=\int_{B_r(m_{\infty})}\phi_{\infty}|\nabla f_{\infty}|^2d\upsilon.\]
\end{claim}
The proof is as follows.
Since $\phi_i\nabla f_i$ converges weakly to $\phi_{\infty}\nabla f_{\infty}$ on $B_r(m_{\infty})$, Proposition \ref{weak3}, Theorem \ref{green} and Remark \ref{diverge} yield 
\begin{align*}
\int_{B_r(m_i)}\phi_i|\nabla f_i|^2d\underline{\mathrm{vol}}=-\int_{B_r(m_i)}f_i\mathrm{div}^{\underline{\mathrm{vol}}}(\phi_i\nabla f_i)d\underline{\mathrm{vol}} &\to -\int_{B_r(m_{\infty})}f_{\infty}\mathrm{div}^{\upsilon}(\phi_{\infty}\nabla f_{\infty})d\upsilon \\
&=\int_{B_r(m_{\infty})}\phi_{\infty}|\nabla f_{\infty}|^2d\upsilon.
\end{align*}

For every sufficiently large $i<\infty$, Claims \ref{7001} and \ref{7002} yield
\begin{align*}
\int_{B_r(m_i)}|\nabla f_i|^2d\underline{\mathrm{vol}}  &\le \int_{B_r(m_i)}\phi_i|\nabla f_i|^2d\underline{\mathrm{vol}} + \Psi(\epsilon; L, n, r, R)\\
&=\int_{B_r(m_{\infty})}\phi_{\infty}|\nabla f_{\infty}|^2d\upsilon + \Psi \le \int_{B_r(m_{\infty})}|\nabla f_{\infty}|^2d\upsilon + \Psi.
\end{align*}
Since $\epsilon$ is arbitrary, we have $\limsup_{i \to \infty}||\nabla f_i||_{L^2(B_r(m_i))} \le ||\nabla f_i||_{L^2(B_r(m_{\infty}))}$.
Therefore Proposition \ref{hon} yields that $\nabla f_i$ $L^2$-converges strongly to $\nabla f_{\infty}$ on $B_r(m_{\infty})$.
In particular, $|\nabla f_i|^2$ converges weakly to $|\nabla f_{\infty}|^2$ on $B_r(m_{\infty})$.

Let $p_1:=p_1(n, R)>1$ with $p_1<p_0$ and $p_1/(2-p_1)<p_0$.
Then Young's inequality yields 
\[|\nabla |\nabla f_i|^2|^{p_1}\le 2^{p_1}|\mathrm{Hess}_{f_i}|^{p_1}|\nabla f_i|^{p_1} \le 2^{p_1}\left( \frac{p_1}{2}|\mathrm{Hess}_{f_i}|^2 + \frac{2-p_1}{2}|\nabla f_i|^{2p_1/(2-p_1)}\right).\]
Thus by Claim \ref{7000} we have $||\nabla |\nabla f_i|^2||_{L^{p_1}(B_r(m_i))} \le C(n, r, R, L)$ for every $i <\infty$.
Therefore Theorem \ref{7} yields (\ref{d}), (\ref{e}) and that $|\nabla f_i|^2$ $L^{p_1}$-converges strongly to $|\nabla f_{\infty}|^2$ on $B_r(m_{\infty})$.
In particular $||\nabla f_i||_{L^{2p_1}(B_r(m_i))} \to ||\nabla f_{\infty}||_{L^{2p_1}(B_r(m_{\infty}))}$.
Therefore by Proposition \ref{hon} we see that $\nabla f_i$ $L^{2p_1}$-converges strongly to $\nabla f_{\infty}$ on $B_r(m_{\infty})$.
Thus Theorems \ref{7} and \ref{sobolev emb} yield (\ref{b}) and (\ref{c}).   

Let $w_{\infty} \in B_R(m_{\infty})$, $t>0$ with $\overline{B}_t(w_{\infty}) \subset B_R(m_{\infty})$, $w_i \to w_{\infty}$ and let $g_i, h_i$ be Lipschitz functions on $B_t(w_i)$ for every $i\le \infty$ with $\sup_{i \le \infty}(\mathbf{Lip}g_i + \mathbf{Lip}h_i)<\infty$ satisfying that $g_i, h_i \in C^2(B_t(w_i))$ hold for every $i<\infty$ with $\sup_{i<\infty}(||\Delta g_i||_{L^2(B_t(w_i))}+||\Delta h_i||_{L^2(B_t(w_i))})<\infty$ and that $g_i, h_i \to g_{\infty}, h_{\infty}$ on $B_t(w_{\infty})$, respectively.
(\ref{d}) yields $\langle df_{\infty}, dg_{\infty}\rangle \in H_{1, p_1}(B_s(w_{\infty}))$ for every $s<t$.
In particular, $\langle df_{\infty}, dg_{\infty}\rangle$ is weakly Lipschitz.
Thus $f_{\infty}$ is weakly twice differentiable on $B_R(m_{\infty})$ with respect to $\mathcal{A}_{2nd}$, i.e., we have (\ref{g}).
Claim \ref{7000} yields $\sup_{i<\infty}||\langle \nabla g_i, \nabla h_i \rangle||_{H_{1, 2}(B_s(w_i))}<\infty$ for every $s<t$.
Thus (\ref{c}), (\ref{e}) and Proposition \ref{conv678} give
\[\lim_{i \to \infty}\int_{B_s(w_i)}\left\langle \nabla f_i, \nabla \langle \nabla g_i, \nabla h_i \rangle \right\rangle d\underline{\mathrm{vol}}=\int_{B_s(w_{\infty})}\left\langle \nabla f_{\infty}, \nabla \langle \nabla g_{\infty}, \nabla h_{\infty} \rangle \right\rangle d\upsilon\]
for every $s<t$.
Similarly, (\ref{e}), Propositions \ref{harm5} and \ref{conv678} yield that
\[\lim_{i \to \infty}\int_{B_s(w_i)}\left\langle \nabla g_i, \nabla \langle \nabla f_i, \nabla h_i \rangle \right\rangle d\underline{\mathrm{vol}}=\int_{B_s(w_{\infty})}\left\langle \nabla g_{\infty}, \nabla \langle \nabla f_{\infty}, \nabla h_{\infty} \rangle \right\rangle d\upsilon\]
and
\[\lim_{i \to \infty}\int_{B_s(w_i)}\left\langle \nabla h_i, \nabla \langle \nabla g_i, \nabla f_i \rangle \right\rangle d\underline{\mathrm{vol}}=\int_{B_s(w_{\infty})}\left\langle \nabla h_{\infty}, \nabla \langle \nabla g_{\infty}, \nabla f_{\infty} \rangle \right\rangle d\upsilon\]
hold for every $s<t$.
Since 
\[2\langle \nabla_{\nabla g_i} \nabla f_i, \nabla h_i \rangle =\langle \nabla g_i, \nabla \langle \nabla f_i, \nabla h_i \rangle \rangle + 
\langle \nabla h_i, \nabla \langle \nabla g_i, \nabla f_i \rangle \rangle -\langle \nabla f_i, \nabla \langle \nabla g_i, \nabla h_i \rangle \rangle,\]
we see that $\mathrm{Hess}_{f_i}(\nabla g_i, \nabla h_i)$ converges weakly to $\mathrm{Hess}_{f_{\infty}}(\nabla g_{\infty}, \nabla h_{\infty})$ on $B_s(w_{\infty})$ for every $s<t$.
Thus (\ref{h}) and (\ref{i}) follow from Remark \ref{000} and Proposition \ref{compat}. $\,\,\,\,\,\,\,\,\,\,\,\Box$
\begin{remark}\label{bounds}
We now give a remark about $L^2$-bounds of the gradient of a $C^2$-function on a manifold:
Let $L>0, R>0$, and let $(M, m)$ be a pointed $n$-dimensional complete Riemannian manifold with $\mathrm{Ric}_M \ge -(n-1)$, and $f$ a $C^2$-function on $B_R(m)$ with $||f||_{L^2(B_R(m))}+||\Delta f||_{L^2(B_R(m))}\le L$.
Then for every $r<R$ we have $||\nabla f||_{L^2(B_r(m))} \le C(n, L, r, R)$.

The proof is as follows.
By \cite[Theorem $6.33$]{ch-co}, there exists a smooth function $\phi$ on $M$ such that $0 \le \phi \le 1$, $\phi |_{B_r(m)}\equiv 1$, $\mathrm{supp} (\phi) \subset B_R(m)$ and $|\Delta \phi|+|\nabla \phi| \le C(n, r, R)$ hold.
Then we have 
\begin{align*}
\int_{B_r(m)}|df|^2d\underline{\mathrm{vol}} \le \int_{B_R(m)}|d(\phi f)|^2d\underline{\mathrm{vol}} &=\int_{B_R(m)}\phi f\Delta(\phi f)d\underline{\mathrm{vol}} \\
&=\int_{B_R(m)}\left(-\frac{1}{2}\langle d\phi^2, df^2\rangle + f^2 \phi \Delta \phi + \phi^2 f\Delta f \right) d\underline{\mathrm{vol}}\\
&\le  -\frac{1}{2}\int_{B_R(m)}f^2\Delta \phi^2d\underline{\mathrm{vol}} + C(n, L, r, R) \le C(n, L, r, R).
\end{align*}
\end{remark}
By Theorem \ref{main} it is not difficult to check the following:
\begin{corollary}\label{tyu}
Let $f_i^1, f_i^2 \in L^2(B_R(m_i))$ for every $i \le \infty$. 
Assume that $f_i^1, f_i^2 \in C^2(B_R(m_i))$ hold for every $i<\infty$ with 
$\sup_{j \in \{1,2\}, i<\infty}(||f_i^j||_{L^2(B_R(m_i))}+||\nabla f_i^j||_{L^{\infty}(B_R(m_{i}))}+||\Delta f_i^j||_{L^2(B_R(m_i))}) < \infty$ and that $f_i^j$ converges weakly to $f_{\infty}^j$ on $B_R(m_{\infty})$ for every $j=1, 2$.
Then we have the following:
\begin{enumerate}
\item $\sup_{i \le \infty}(||[\nabla f_i^1, \nabla f_i^2]||_{L^{2}(B_r(m_i))}+ ||\nabla_{\nabla f_i^1}\nabla f^2_i||_{L^{2}(B_r(m_i))})<\infty$ holds for every $r<R$.
\item $\nabla_{\nabla f_i^1}\nabla f_i^2$, $[\nabla f^1_i, \nabla f^2_i]$ converge weakly to $\nabla_{\nabla f^1_{\infty}}\nabla f^2_i$, $[\nabla f^1_{\infty}, \nabla f^2_{\infty}]$ on $B_r(m_{\infty})$ for every $r<R$, respectively. 
\end{enumerate}
\end{corollary}
The following corollary is about an existence of a good cutoff function on $M_{\infty}$ which is a generalization of \cite[Theorem $6.33$]{ch-co} to limit spaces:
\begin{corollary}\label{cut off}
Let $r>0$ with $r<R$.
Then there exists a $C(n, r, R)$-Lipschitz function $\hat{\phi}_{\infty}$ on $M_{\infty}$ such that $0 \le \hat{\phi}_{\infty} \le 1$, $\hat{\phi}_{\infty}|_{B_r(m_{\infty})}\equiv 1$, $\mathrm{supp} (\hat{\phi}_{\infty}) \subset B_R(m_{\infty})$, $\hat{\phi}_{\infty} \in \mathcal{D}(\Delta^{\upsilon}, M_{\infty})$ and that $\hat{\phi}_{\infty}$ is weakly twice differentiable on $M_{\infty}$
with $||\mathrm{Hess}_{\hat{\phi}_{\infty}}||_{L^2(M_{\infty})}+||\Delta^{\upsilon}\hat{\phi}_{\infty}||_{L^{\infty}(M_{\infty})}\le C(n, r, R)$.
\end{corollary}
\begin{proof}
\cite[Theorem $6.33$]{ch-co} and \cite[Remark $4.2$]{Ho} yield that for every $i<\infty$ there exists a smooth $C(n, r, R)$-Lipschitz function $\hat{\phi}_{i}$ on $M_{i}$ such that $0 \le \hat{\phi}_{i} \le 1$, $\hat{\phi}_{i}|_{B_r(m_{i})}\equiv 1$, $\mathrm{supp} (\hat{\phi}_{i}) \subset B_R(m_{i})$, $\hat{\phi}_{i} \in \mathcal{D}(\Delta^{\upsilon}, M_{i})$ and  $||\mathrm{Hess}_{\hat{\phi}_{i}}||_{L^2(M_{i})}+||\Delta \hat{\phi}_{i}||_{L^{\infty}(M_{i})}\le C(n, r, R)$.
By Proposition \ref{conti com}, without loss of generality we can assume that there exists a  $C(n, r, R)$-Lipschitz function $\hat{\phi}_{\infty}$ on $M_{\infty}$ such that $\hat{\phi}_i \to \hat{\phi}_{\infty}$ on $M_{\infty}$.
Then Theorem \ref{main}, Propositions \ref{lower semi} and \ref{laplap} yield that $\hat{\phi}_{\infty}$ satisfies the desired conditions.
\end{proof}
\subsection{A Bochner-type inequality for general case.}
We now give a proof of Theorem \ref{bochner1}:

\textit{A proof of Theorem \ref{bochner1}.}

Let $\tau>0$ with $\mathrm{supp} (\phi_{\infty}) \subset B_{R-2\tau}(m_{\infty})$, $L:=\mathbf{Lip}\phi_{\infty}+||\phi_{\infty}||_{L^{\infty}}$, and 
$\epsilon>0$ with $\epsilon <<\tau$.
\cite[Theorem $4.2$]{Ho} yields that for every $i\le \infty$ and every $j<\infty$, there exist a $C(n,  L)$-Lipschitz function $\phi_{i, j}$ on $B_R(m_{i})$ and an open subset $\Omega_j \subset B_R(m_{\infty})$ such that $\mathrm{supp} (\phi_{i, j}) \subset B_{R-\tau}(m_i)$, $(\phi_{i, j}, d\phi_{i, j}) \to (\phi_{\infty, j}, d\phi_{\infty, j})$ on $\Omega_j$ as $i \to \infty$, $\upsilon (B_R(m_{\infty}) \setminus \Omega_j)<j^{-1}$ and that $||\phi_{\infty, j}-\phi_{\infty}||_{L^{\infty}(B_R(m_{\infty}))} + ||d\phi_{\infty, j}-d\phi_{\infty}||_{L^2(B_R(m_{\infty}))} \to 0$ as $j \to \infty$.
\cite[Theorem $6.33$]{ch-co} yields that for every $j< \infty$ there exists a smooth $C(n, \tau)$-Lipschitz function $\hat{\phi}_{j}$ on $M_j$ such that $||\Delta \hat{\phi}_j||_{L^{\infty}} \le C(n, \tau, R)$, $0 \le \hat{\phi}_j \le 1$, $\hat{\phi}_j|_{B_{R-2\tau}(m_j)}\equiv 1$ and $\mathrm{supp} (\hat{\phi}_j) \subset B_{R-\tau}(m_j)$.
Note that there exists $\epsilon_j \to 0$ such that for every $j$,  there exists $j_0$ such that $\phi_{i, j} + \epsilon_ j\hat{\phi}_{i} \ge 0$ holds  on $B_R(m_i)$ for every $i \ge j_0$.
Let $g_{i, j}:=\phi_{i, j} + \epsilon_ j\hat{\phi}_{i}$.
Then Propositions \ref{weak com} and \ref{tensor com} yield that there exists a subsequence $\{i(j)\}_j$ such that $g_{i(j), j} \ge 0$ and that $(g_{i(j), j}, dg_{i(j), j}) \to (\phi_{\infty}, d\phi_{\infty})$ on $B_R(m_{\infty})$.
Let $\phi_{i(j)} :=g_{i(j), j}$.
Then  Bochner's formula yields 
\begin{align*}
-\frac{1}{2}\int_{B_R(m_{i(j)})}\langle d\phi_{i(j)}, d|df_{i(j)}|^2 \rangle d\underline{\mathrm{vol}} &\ge \int_{B_R(m_{i(j)})}\phi_{i(j)}|\mathrm{Hess}_{f_{i(j)}}|^2d\underline{\mathrm{vol}} \\
&+ \int_{B_R(m_{i(j)})}\left(-\phi_{i(j)}(\Delta f_{i(j)})^2+\Delta f_{i(j)}\langle d\phi_{i(j)}, df_{i(j)}\rangle \right)d\underline{\mathrm{vol}} \\
&+K(n-1)\int_{B_R(m_{i(j)})}\phi_{i(j)}|df_{i(j)}|^2d\underline{\mathrm{vol}}.
\end{align*}
(\ref{e}) of Theorem \ref{main} and Proposition \ref{conv678} yield
\[\lim_{j \to \infty}\int_{B_R(m_{i(j)})}\langle d\phi_{i(j)}, d|df_{i(j)}|^2 \rangle d\underline{\mathrm{vol}}=\int_{B_R(m_{\infty})}\langle d\phi_{\infty}, d|df_{\infty}|^2 \rangle d\upsilon.\]
(\ref{i}) of Theorem \ref{main}, Propositions \ref{strong3} and \ref{weak4} yield that $(\phi_{i(j)})^{1/2}\mathrm{Hess}_{f_{i(j)}}$ converges weakly to $(\phi_{\infty})^{1/2}\mathrm{Hess}_{f_{\infty}}$ on $B_R(m_{\infty})$.
In particular by Proposition \ref{lower semi} we have  
\[\liminf_{j \to \infty}\int_{B_R(m_{i(j)})}\phi_{i(j)}|\mathrm{Hess}_{f_{i(j)}}|^2d\underline{\mathrm{vol}} \ge \int_{B_R(m_{\infty})}\phi_{\infty}|\mathrm{Hess}_{f_{\infty}}|^2d\upsilon.\]
On the other hand, Proposition \ref{strong3} and  Corollary \ref{61746174} yield that $(\phi_{i(j)})^{1/2}\Delta f_{i(j)}$ $L^2$-converges strongly to $(\phi_{\infty})^{1/2}\Delta^{\upsilon} f_{\infty}$ on $B_R(m_{\infty})$.
In particular we have
\[\lim_{j \to \infty}\int_{B_R(m_{i(j)})}\phi_{i(j)}(\Delta f_{i(j)})^2d\underline{\mathrm{vol}}= \int_{B_R(m_{\infty})}\phi_{\infty}(\Delta^{\upsilon}f_{\infty})^2d\upsilon. \]
(\ref{c}) of Theorem \ref{main} and Proposition \ref{weak3} yield that 
\[\lim_{j \to \infty}\int_{B_R(m_{i(j)})}\Delta f_{i(j)}\langle d\phi_{i(j)}, df_{i(j)}\rangle d\underline{\mathrm{vol}}=\int_{B_R(m_{\infty})}\Delta^{\upsilon}f_{\infty}\langle d\phi_{\infty}, df_{\infty}\rangle d\upsilon\]
and
\[\lim_{j \to \infty}\int_{B_R(m_{i(j)})}\phi_{i(j)}|df_{i(j)}|^2d\underline{\mathrm{vol}}=\int_{B_R(m_{\infty})}\phi_{\infty}|df_{\infty}|^2d\upsilon\]
hold.
Thus we have the assertion. \,\,\,\,\,\,\, $\Box$
\begin{remark}\label{gamma2}
Under the same assumption as in Theorem \ref{bochner1}, 
since $|\mathrm{Hess}_{f_i}|^2 \ge (\Delta f_i)^2 /n$ for every $i <\infty$, and $|\mathrm{Hess}_{f_{\infty}}|^2 \ge (\Delta^{g_{M_{\infty}}}f_{\infty})^2/k$, where $k := \mathrm{dim}\,M_{\infty}$, by an argument similar to the proof of Theorem \ref{bochner1} we see that
\begin{align*}
-\frac{1}{2}\int_{B_R(m_{\infty})}\langle d\phi_{\infty}, d|df_{\infty}|^2 \rangle d\upsilon &\ge \int_{B_R(m_{\infty})}\phi_{\infty}\frac{(\Delta^{\upsilon}f_{\infty})^2}{n}d\upsilon \\
&+ \int_{B_R(m_{\infty})}\left(-\phi_{\infty}(\Delta^{\upsilon}f_{\infty})^2+\Delta^{\upsilon}f_{\infty}\langle d\phi_{\infty}, df_{\infty}\rangle \right)d\upsilon \\
&+K(n-1)\int_{B_R(m_{\infty})}\phi_{\infty}|df_{\infty}|^2d\upsilon
\end{align*} 
and
\begin{align*}
-\frac{1}{2}\int_{B_R(m_{\infty})}\langle d\phi_{\infty}, d|df_{\infty}|^2 \rangle d\upsilon &\ge \int_{B_R(m_{\infty})}\phi_{\infty}\frac{(\Delta^{g_{M_{\infty}}}f_{\infty})^2}{k}d\upsilon \\
&+ \int_{B_R(m_{\infty})}\left(-\phi_{\infty}(\Delta^{\upsilon}f_{\infty})^2+\Delta^{\upsilon}f_{\infty}\langle d\phi_{\infty}, df_{\infty}\rangle \right)d\upsilon \\
&+K(n-1)\int_{B_R(m_{\infty})}\phi_{\infty}|df_{\infty}|^2d\upsilon.
\end{align*} 
On the other hand, Proposition \ref{green} and Remark \ref{diverge} yield
\[\int_{B_R(m_{\infty})}\left(- \phi_{\infty}(\Delta^{\upsilon}f_{\infty})^2+\Delta^{\upsilon}f_{\infty}\langle d\phi_{\infty}, df_{\infty}\rangle \right)d\upsilon=\int_{B_R(m_{\infty})}\Delta^{\upsilon}f_{\infty}\mathrm{div}^{\upsilon}(\phi_{\infty}\nabla f_{\infty})d\upsilon.\]
Moreover if $f_i \in C^3(B_R(m_i))$ holds for every $i<\infty$ with $\sup_{i<\infty}||\nabla \Delta f_i||_{L^2(B_R(m_i))}<\infty$, then Theorem \ref{7} yields that $\Delta^{\upsilon}f_{\infty} \in H_{1, 2}(B_R(m_{\infty}))$ and that $\nabla \Delta f_i$ converges weakly to $\nabla \Delta^{\upsilon} f_{\infty}$ on $B_R(m_{\infty})$.
Therefore we have 
\[\int_{B_R(m_{\infty})}\Delta^{\upsilon}f_{\infty}\mathrm{div}^{\upsilon}(\phi_{\infty}\nabla f_{\infty})d\upsilon=-\int_{B_R(m_{\infty})}\phi_{\infty}\langle \nabla \Delta^{\upsilon}f_{\infty}, \nabla f_{\infty}\rangle d\upsilon.\]
In particular for $r<R$ and $\phi_{\infty}$ as in Corollary \ref{cut off}, we have the following \textit{$\Gamma_2$-condition}:
\begin{align*}
-\frac{1}{2}\int_{B_R(m_{\infty})}\Delta^{\upsilon} \phi_{\infty} |df_{\infty}|^2 d\upsilon &\ge - \int_{B_R(m_{\infty})}\phi_{\infty}\langle \nabla \Delta^{\upsilon}f_{\infty}, \nabla f_{\infty}\rangle d\upsilon +K(n-1)\int_{B_R(m_{\infty})}\phi_{\infty}|df_{\infty}|^2d\upsilon.
\end{align*}
See also \cite{bac, gigli, gko, led}. 
\end{remark}
\subsection{Noncollapsing case.}
We now prove Theorem \ref{laplacian}:

\textit{A proof of Theorem \ref{laplacian}.}

First assume that $(M_{\infty}, m_{\infty})$ is the noncollapsed limit space of $\{(M_i, m_i)\}_i$.
Then
(\ref{i}) of Theorem \ref{main} and Proposition \ref{contr} yield that $\mathrm{tr} (\mathrm{Hess}_{f_i})$ converges weakly to $\mathrm{tr} (\mathrm{Hess}_{f_{\infty}})$ on $B_R(m_{\infty})$.
Therefore (\ref{c}) of Theorem \ref{main} yields  $-\mathrm{tr} (\mathrm{Hess}_{f_{\infty}})=\Delta^{\upsilon}f_{\infty}$ on $B_R(m_{\infty})$.
Thus we have $(1)$.
Similarly, it is easy to check $(2)$ by Proposition \ref{contr2}. \,\,\,\,\, $\Box$

We give a Bochner-type inequality for noncollapsed limit spaces:
\begin{corollary}\label{bochner2}
Let $(M_{\infty}, m_{\infty})$ be the  noncollapsed $(n, K)$-Ricci limit space of a sequence $\{(M_i, m_i)\}_{i<\infty}$.
Then with the same assumption as in Theorem \ref{bochner1}, we have 
\begin{align*}
-\frac{1}{2}\int_{B_R(m_{\infty})}\langle d\phi_{\infty}, d|df_{\infty}|^2 \rangle dH^n &\ge \int_{B_R(m_{\infty})}\phi_{\infty}|\mathrm{Hess}_{f_{\infty}}|^2dH^n \\
&+ \int_{B_R(m_{\infty})}\Delta^{g_{M_{\infty}}} f_{\infty} \mathrm{div}^{g_{M_{\infty}}}(\phi_{\infty}\nabla f_{\infty})dH^n \\
&+K(n-1)\int_{B_R(m_{\infty})}\phi_{\infty}|df_{\infty}|^2dH^n.
\end{align*}
\end{corollary}
\begin{proof}
It follows from Theorems \ref{bochner1}, \ref{laplacian} and Proposition \ref{hess2}.
\end{proof}
\begin{remark}\label{exex}
We calculate the Hessians of important warping functions given by Cheeger-Colding in \cite{ch-co}.
Let $(M_i, m_i, \underline{\mathrm{vol}}) \stackrel{(\psi_i, \epsilon_i, R_i)}{\to} (M_{\infty}, m_{\infty}, \upsilon)$ and $k:=\mathrm{dim}\,M_{\infty}$.

\textit{Splitting.} Assume that $\mathrm{Ric}_{M_i} \ge -\delta_i$ with $\delta_i \to 0$, and that there exists a line $l: \mathbf{R} \to M_{\infty}$ with $l(0)=m_{\infty}$ which means that $l$ is an isometric embedding.
Let $R>0$, $r_i<R_i$ with $r_i \to \infty$ and $z_i \in B_{R_i}(m_i)$ with $\overline{\psi_i(z_i), l(r_i)} \to 0$.
Define the harmonic function $\mathbf{b}_i$ on $B_{100R}(m_i)$ by $\mathbf{b}_i|_{\partial B_{100R}(m_i)}\equiv r_{z_i}- r_{z_i}(m_i)$.
Then we see that $\sup_i \mathbf{Lip}(\mathbf{b}_i|_{B_R(m_i)})<\infty$, 
$\mathbf{b}_i$ converges to the Busemann function $\mathbf{b}_{\infty}$ of $l$ on $B_R(m_{\infty})$, $d\mathbf{b}_i \to d\mathbf{b}_{\infty}$ on $B_R(m_{\infty})$ and that
\[\lim_{i \to 0}\int_{B_R(m_i)}|\mathrm{Hess}_{\mathbf{b}_i}|^2d\underline{\mathrm{vol}}=0\]
holds.
See \cite[Theorem $6.64$]{ch-co} for the proof.
Therefore (\ref{i}) of Theorem \ref{main} and Proposition \ref{hon} yield that $\mathrm{Hess}_{\mathbf{b}_{\infty}}\equiv 0$ and that $\mathrm{Hess}_{\mathbf{b}_i}$ $L^2$-converges strongly to $\mathrm{Hess}_{\mathbf{b}_{\infty}}$ on $B_R(m_{\infty})$.
In particular $(2)$ of Theorem \ref{laplacian} yields $\Delta^{\upsilon}\mathbf{b}_{\infty}=\Delta^{g_{M_{\infty}}}\mathbf{b}_{\infty}\equiv 0$.

\textit{Suspension.} Assume that $\mathrm{Ric}_{M_i} \ge n-1$ and $\mathrm{diam}M_{\infty}=\pi$. 
Let $\{\phi_i\}_{i \le \infty}$ be a sequence of eigenfunctions $\phi_i$ on $M_i$ associated with the first eigenvalue $\lambda_1(M_i)$ with respect to the Dirichlet problem on $M_i$  with $\phi_i \to \phi_{\infty}$ on $M_{\infty}$ and $||\phi_i||_{L^2(M_i)}=1$.
See Remark \ref{eigen conti}.
Then we see that $\sup_i \mathbf{Lip}\phi_i<\infty$, $(\phi_i, d\phi_i) \to (\phi_{\infty}, d\phi_{\infty})$ on $M_{\infty}$ and that 
\[\lim_{i \to \infty}\int_{M_i}|\mathrm{Hess}_{\phi_i}+\phi_ig_{M_i}|^2d\underline{\mathrm{vol}}=0\]
holds.
See the proof of \cite[Lemma $1.4$]{co1}, \cite[Theorem $5.14$]{ch-co} and \cite[Theorem $7.9$]{ch-co3} for the details.
Therefore  we see that $\mathrm{Hess}_{\phi_i}+\phi_ig_{M_i}$ $L^2$-converges strongly to $0$ on $M_{\infty}$ and that $\mathrm{Hess}_{\phi_{\infty}}=-\phi_{\infty}g_{M_{\infty}}$ (note that Theorem \ref{metric} yields that if $(M_{\infty}, m_{\infty})$ is the noncollapsed limit of $\{(M_i, m_i)\}_{i<\infty}$, then $\mathrm{Hess}_{\phi_i}$ $L^2$-converges strongly to $\mathrm{Hess}_{\phi_{\infty}}$ on $M_{\infty}$).
Thus we have $\Delta^{g_{M_{\infty}}}\phi_{\infty}=k\phi_{\infty}$. 
On the other hand, \cite[Theorem $7.9$]{ch-co3} yields $\Delta^{\upsilon}\phi_{\infty}=n\phi_{\infty}$.
In particular  $\Delta^{g_{M_{\infty}}}\phi_{\infty}=\Delta^{\upsilon}\phi_{\infty}$ holds on $M_{\infty}$ if and only if $M_{\infty}$ is the noncollapsed limit space of $\{M_i\}_i$. 
See also \cite{an, otsu2} for examples of interesting singular limit spaces.

\textit{Cone.} 
Assume that $\mathrm{Ric}_{M_i} \ge -\delta_i$ with $\delta_i \to 0$ and that 
\[\lim_{i \to \infty}\frac{\mathrm{vol}\,\partial B_{100R}(m_i)}{\mathrm{vol}\,B_{100R}(m_i)}=\frac{\mathrm{vol}\,\partial B_{100R}(0_n)}{\mathrm{vol}\,B_{100R}(0_n)}\]
holds for some $R>0$,
where $0_n \in \mathbf{R}^n$.
For every $i < \infty$, let $f_i$ be the function on $\overline{B}_{100R}(m_i)$ satisfying that $\Delta f_i \equiv -1$ on $B_{100R}(m_i)$ and $f_{i}|_{\partial B_{100R}(m_i)} \equiv (100R)^2/2n$.
Note that Cheng-Yau's gradient estimate \cite{ch-yau} yields $\sup_i \mathbf{Lip}(f_i|_{B_R(m_i)})<\infty$.
Then \cite[Theorem $4.91$]{ch-co} yields that there exists a compact geodesic space $X$ with $\mathrm{diam}\,X \le \pi$ such that $(B_R(m_{\infty}), m_{\infty})$ is isometric to $(B_R(p_0), p_0)$, where $C(X)$ is the metric cone of $X$ and $p_0$ is the pole of $C(X)$, 
$(f_i, df_i) \to \left(r_{p_0}^2/2n, d(r_{p_0}^2/2n)\right)$ on $B_R(p_0)$ and that 
\[\lim_{i \to \infty}\int_{B_R(m_i)}\left|\mathrm{Hess}_{f_i}+\frac{1}{n}g_{M_i}\right|^2d\underline{\mathrm{vol}}=0\]
holds.
Let $f_{\infty}:=r_{p_0}^2/2n$.
Then we see that $\mathrm{Hess}_{f_i}+g_{M_i}/n$ $L^2$-converges strongly to $0$ on $B_R(m_{\infty})$ and that $\mathrm{Hess}_{f_{\infty}}\equiv -g_{M_{\infty}}/n$.
Note that (\ref{f}) of Theorem \ref{main} yields $\Delta^{\upsilon}f_{\infty} \equiv -1$.
Thus we have $\Delta^{g_{M_{\infty}}}f_{\infty}=(k/n)\Delta^{\upsilon}f_{\infty}$.
In particular $\Delta^{g_{M_{\infty}}}f_{\infty}=\Delta^{\upsilon}f_{\infty}$ holds on $B_R(m_{\infty})$ if and only if  
$(M_{\infty}, m_{\infty})$ is the noncollapsed limit of $\{(M_i, m_i)\}_i$.
See also \cite[Example $1.24$]{ch-co1} for an example of such collapsed limit spaces.
\end{remark}
\begin{remark}\label{notexample}
We now give an example of $\Delta^{g_{M_{\infty}}}f_{\infty} \neq \Delta^{\upsilon}f_{\infty}$.
Let $X_n$ be the quotient metric space $\mathbf{S}^2/\mathbf{Z}_n$ of $\mathbf{S}^2$ (with the canonical orbifold metric) by the action of $\mathbf{Z}_n$ generated by the rotation of angle $2\pi/n$ around a fixed axis.
Then it is easy to check that every $X_n$ is the Gromov-Hausdorff limit space of a sequence of compact $2$-dimensional Riemannian manifolds $\{M_i\}_i$ with $K_{M_i} \ge 1$, where $K_{M_i}$ is the sectional curvature of $M_i$.
Since $X_n$ Gromov-Hausdorff converges $[0, \pi]$ as $n \to \infty$, \cite[Lemma $1.10$]{co1} yields that there exist a Radon measure $\upsilon$ on $[0, 1]$ and  a sequence of compact $2$-dimensional Riemannian manifolds $\{M_i\}_i$ with $K_{M_i} \ge 1$ such that $\lambda_1(M_i) \to 2$ and 
$(M_i, \underline{\mathrm{vol}})$ Gromov-Hausdorff converges $([0, \pi], \upsilon)$.
Then Remark \ref{exex} and \cite[Theorem $7.9$]{ch-co3} yield that every first eigenfunction $\phi_{\infty}$ of $\Delta^{\upsilon}$ satisfies 
$\Delta^{g_{M_{\infty}}}\phi_{\infty} \neq \Delta^{\upsilon}\phi_{\infty}$.
\end{remark}
\subsection{$L^p$-bounds on Ricci curvature and scalar curvature.}
In this subsection we consider the following setting:
\begin{enumerate}
\item $(M_{\infty}, m_{\infty}, \upsilon)$ is the $(n, K)$-Ricci limit space of $\{(M_i, m_i, \underline{\mathrm{vol}})\}_{i<\infty}$ with $M_{\infty} \neq \{m_{\infty}\}$.
\item $\mathcal{A}_{2nd}$ is a weakly second order differential structure on $(M_{\infty}, \upsilon)$ associated with $\{(M_i, m_i, \underline{\mathrm{vol}})\}_{i<\infty}$.
\item $\sup_{i<\infty}||\mathrm{Ric}_{M_i}||_{L^p(B_R(m_i))}<\infty$ for some $1<p \le \infty$ and some $R>0$.
\end{enumerate}
Then by Proposition \ref{tensor com} there exists a weak convergent subsequence $\{\mathrm{Ric}_{M_{i(j)}}\}_j$.
Thus furthermore we assume the following:
\begin{enumerate}
\item[(4)] There exists $\mathrm{Ric}_{M_{\infty}} \in L^p(T^0_2B_R(m_{\infty}))$ such that $\mathrm{Ric}_{M_{i}}$ converges weakly to $\mathrm{Ric}_{M_{\infty}}$ on $B_R(m_{\infty})$.
\end{enumerate}
We call $\mathrm{Ric}_{M_{\infty}}$ \textit{the Ricci tensor of $B_R(m_{\infty})$ with respect to $\{(M_i, m_i, \underline{\mathrm{vol}})\}_{i<\infty}$}.
In this setting we can get the following \textit{Bochner-type formula}:
\begin{theorem}\label{bochner7}
Let $\{f_i\}_{i \le \infty}$ be as in Theorem \ref{main}. 
Furthermore we assume that there exists $\hat{p} \ge 2q$ with $\hat{p} >2$, where $q$ is the conjugate exponent of $p$, such that the following hold:
\begin{enumerate}
\item $\nabla f_i$ $L^{\hat{p}}$-converges strongly to $\nabla f_{\infty}$ on $B_r(m_{\infty})$ for every $r<R$.
\item $\mathrm{Hess}_{f_i}$ $L^2$-converges strongly to $\mathrm{Hess}_{f_{\infty}}$ on $B_r(m_{\infty})$ for every $r<R$.
\end{enumerate}
Then 
\begin{align*}
-\frac{1}{2}\int_{B_R(m_{\infty})}\langle d\phi_{\infty}, d|df_{\infty}|^2 \rangle d\upsilon &= \int_{B_R(m_{\infty})}\phi_{\infty}|\mathrm{Hess}_{f_{\infty}}|^2d\upsilon \\
&+ \int_{B_R(m_{\infty})}\Delta^{\upsilon}f_{\infty}\mathrm{div}^{\upsilon}(\phi_{\infty}\nabla f_{\infty})d\upsilon \\
&+\int_{B_R(m_{\infty})}\phi_{\infty}\mathrm{Ric}_{M_{\infty}}(\nabla f_{\infty}, \nabla f_{\infty})d\upsilon
\end{align*}
holds for every Lipschitz function $\phi_{\infty}$ on $B_R(m_{\infty})$ with compact support.
\end{theorem}
\begin{proof}
Let $\tau >0$ with $\mathrm{supp} (\phi_{\infty}) \subset B_{R-\tau}(m_{\infty})$ and let $\phi_i$ be a Lipschitz function on $B_R(m_i)$ for every $i \le \infty$ with $\sup_i\mathbf{Lip}\phi_i<\infty$, $\mathrm{supp} (\phi_i) \subset B_{R-\tau}(m_i)$ and $\phi_i \to \phi_{\infty}$ on $B_R(m_{\infty})$.
By the proof of Theorem \ref{bochner1} and Proposition \ref{contr2}, it suffices to check the following:
\begin{claim}
We have
\[\lim_{i \to \infty}\int_{B_R(m_i)}\phi_i\mathrm{Ric}_{M_{i}}(\nabla f_i, \nabla f_i)d\underline{\mathrm{vol}}=\int_{B_R(m_{\infty})}\phi_{\infty}\mathrm{Ric}_{M_{\infty}}(\nabla f_{\infty}, \nabla f_{\infty})d\upsilon.\]
\end{claim}
The proof is as follows.
For every $x_i \to x_{\infty}$ and every $y_i \to y_{\infty}$, since $\langle \nabla f_i, \nabla r_{x_i} \rangle$ $L^{2}$-converges strongly to $\langle \nabla f_{\infty}, \nabla r_{x_{\infty}} \rangle$ on $B_R(m_{\infty})$, we see that
\begin{align*}
\int_{B_t(w_i)}\langle \nabla f_i \otimes \nabla f_i, \nabla r_{x_i}\otimes \nabla r_{y_i}\rangle d\underline{\mathrm{vol}}&=\int_{B_t(w_i)}\langle \nabla f_i, \nabla r_{x_i} \rangle \langle \nabla f_{i}, \nabla r_{y_i} \rangle d\underline{\mathrm{vol}} \\
&\to \int_{B_t(w_{\infty})}\langle \nabla f_{\infty}, \nabla r_{x_{\infty}} \rangle \langle \nabla f_{\infty}, \nabla r_{y_{\infty}} \rangle d\upsilon \\
&=\int_{B_t(w_{\infty})}\langle \nabla f_{\infty} \otimes \nabla f_{\infty}, \nabla r_{x_{\infty}} \otimes \nabla r_{y_{\infty}} \rangle d\upsilon
\end{align*}
holds for every $w_{\infty} \in B_R(m_{\infty})$, every $t>0$ with $\overline{B}_t(w_{\infty}) \subset B_R(m_{\infty})$ and every $w_i \to w_{\infty}$.
Thus $\nabla f_i \otimes \nabla f_i$ converges weakly to $\nabla f_{\infty} \otimes \nabla f_{\infty}$ on $B_r(m_{\infty})$ for every $r<R$.
On the other hand, for every $r<R$, since
\begin{align*}
\int_{B_r(m_i)}|\nabla f_i \otimes \nabla f_i|^{\hat{p}/2}d\underline{\mathrm{vol}}&=\int_{B_r(m_i)}|\nabla f_i|^{\hat{p}}d\underline{\mathrm{vol}}\\
&\to \int_{B_r(m_{\infty})}|\nabla f_{\infty}|^{\hat{p}}d\upsilon \\
&=\int_{B_r(m_{\infty})}|\nabla f_{\infty} \otimes \nabla f_{\infty}|^{\hat{p}/2}d\upsilon.
\end{align*}
Proposition \ref{hon} yields that $\nabla f_i \otimes \nabla f_i$ $L^{\hat{p}/2}$-converges strongly to $\nabla f_{\infty} \otimes \nabla f_{\infty}$ on $B_r(m_{\infty})$.
Thus by Proposition \ref{strong83} we see that $\phi_i \nabla f_i \otimes \nabla f_i$ $L^{\hat{p}/2}$-converges strongly to $\phi_{\infty} \nabla f_{\infty} \otimes \nabla f_{\infty}$ on $B_R(m_{\infty})$.
Therefore Proposition \ref{conv678} yields
\begin{align*}
\int_{B_R(m_i)}\phi_i\mathrm{Ric}_{M_{i}}(\nabla f_i, \nabla f_i)d\underline{\mathrm{vol}}&=\int_{B_R(m_i)}\langle \mathrm{Ric}_{M_{i}}, \phi_i df_i \otimes df_i \rangle d\underline{\mathrm{vol}} \\
& \to \int_{B_R(m_{\infty})}\langle \mathrm{Ric}_{M_{\infty}}, \phi_{\infty} df_{\infty} \otimes df_{\infty}\rangle d\upsilon \\
&=\int_{B_R(m_{\infty})}\phi_{\infty}\mathrm{Ric}_{M_{\infty}}(\nabla f_{\infty}, \nabla f_{\infty})d\upsilon.
\end{align*}
Therefore we have the assertion.
\end{proof}
\begin{corollary}
Let $\{f_i\}_{i \le \infty}$ be as in Theorem \ref{main}.
Assume that $p=\infty$ and that $\mathrm{Hess}_{f_i}$ $L^2$-converges strongly to $\mathrm{Hess}_{f_{\infty}}$ on $B_r(m_{\infty})$ for every $r<R$.
Then we have the same conclusion as in Theorem \ref{bochner7}.
\end{corollary}
\begin{proof}
It follows from (\ref{c}) of Theorem \ref{main} and Theorem \ref{bochner7}.
\end{proof}
We end this subsection by discussing the \textit{scalar curvature of $M_{\infty}$}:
\begin{definition}\label{scal}
Let $s_{M_{\infty}}:=\mathrm{tr} (\mathrm{Ric}_{M_{\infty}}) \in L^p(B_R(m_{\infty}))$.
We say that \textit{$s_{M_{\infty}}$ is the scalar curvature of $B_R(m_{\infty})$ with respect to $\{(M_i, m_i, \underline{\mathrm{vol}})\}_i$}. 
\end{definition}
\begin{corollary}\label{scal conv}
Assume that $M_{\infty}$ is the noncollapsed limit space of $\{(M_i, m_i)\}_i$.
Then the scalar curvatures $s_{M_i}$ of $M_i$ converges weakly to $s_{M_{\infty}}$ on $B_R(m_{\infty})$.
\end{corollary}
\begin{proof}
It follows from Proposition \ref{contr}.
\end{proof}

\end{document}